\documentclass{amsart}
\usepackage{latexsym}
\usepackage{amssymb}
\usepackage{amsmath}
\newtheorem{lemma}{Lemma}[section]
\newtheorem{proposition}[lemma]{Proposition}
\newtheorem{corollary}[lemma]{Corollary}

\newtheorem{theorem}[lemma]{Theorem}
\newtheorem{examples}[lemma]{Examples}

\begin{document}

\author[M.~V.~Lawson]{Mark~V.~Lawson$^1$}
\address{$^1$Department of Mathematics
and the
Maxwell Institute for Mathematical Sciences,
Heriot-Watt University,
Riccarton,
Edinburgh~EH14~4AS,
Scotland}
\email{markl@ma.hw.ac.uk }

\author[ D.~H.~Lenz]{Daniel~H.~Lenz$^2$}
\address{$^2$ Mathematisches Institut, Friedrich-Schiller Universit\"at Jena, Ernst-Abb\'{e} Platz~2, 07743 Jena, Germany}
\email{daniel.lenz@uni-jena.de }

\thanks{The first author was partially supported by an EPSRC grant EP/I033203/1. 
He would also like to thank the Mathematics Department of the University of Ottawa where some of this work was carried out.
Both authors would like to thank Ganna Kudryavtseva and Pedro Resende for conversations on parts of this paper.}

\title[Distributive inverse semigroups]{Distributive inverse semigroups and non-commutative Stone dualities}



\begin{abstract}
We develop the theory of distributive inverse semigroups as the analogue of distributive lattices without top element
and prove that they are in a duality with those \'etale groupoids having a spectral space of identities, where our spectral spaces are
not necessarily compact.
We prove that Boolean inverse semigroups can be characterized as those distributive inverse semigroups
in which every prime filter is an ultrafilter; we also provide a topological characterization in terms of Hausdorffness.
We extend the notion of the patch topology to distributive inverse semigroups and prove that every distributive inverse semigroup
has a Booleanization. 
As applications of this result, we give a new interpretation of Paterson's universal groupoid of an inverse semigroup
and by developing the theory of what we call tight coverages, we also provide a conceptual foundation for Exel's tight groupoid. 
\end{abstract}

\maketitle

%%%%%%%%%%%%%%%%%%%%%%%%%%%%%%%%%%%%%%%%%%%%%%%%%%%%%%%%%%%%%%%%%%%%%%%%%%%%%%%%%%%%%%%%%%%%%%%%%%%%%%%%%%%%%%%%%%%%%%%%%%%
\section{Introduction}
 
This paper is a contribution to the developing field of non-commutative frame theory.
This has its roots in the work of Ehresmann \cite{E} going back to the 1950's 
but has received its most recent impetus through a deeper understanding 
of the role played by inverse semigroups in the theory of $C^{\ast}$-algebras \cite{Exel1,Exel2,Kel1,Kel2,L,LMS,P,Ren,R1,R2}.
This paper develops an approach to non-commutative frame theory initiated by the first author in the papers \cite{Law2,Law3,Law4} 
motivated by the second author's paper \cite{L}.
Here we generalize the classical theory of distributive lattices to what we call distributive inverse semigroups.
We then show the importance of this class of inverse semigroups for understanding the way in which
inverse semigroup theory interacts with $C^{\ast}$-algebra theory.
This introduction consists of three parts: 
the first consists of an historical survey, 
the second describes the background results needed to properly understand our work, 
and the third outlines our main results.

%%%%%%%%%%%%%%%%%%%%%%%%%%%%%%%%%%%%%%%%%%%%%%%%%%%%%%%%%%%%%%%%%%%%%%%%%%%%%%%%%%%%%%%%%%%%%%%%%%%%%%%%%%%%%%%%%%%%%%%%%%%%%%%%%%%%%
\subsection{Historical background}

The lattice of open sets of a topological space is a complete infinitely distributive lattice or {\em frame}.
The theory of frames can be viewed as the theory of topological spaces in which open sets and not points are taken as basic.
As well as being an interesting theory in its own right \cite{J} with important applications, it is also a key ingredient in topos theory \cite{MM}. 
Johnstone discusses the origins of frame theory in his notes to Chapter~II of his classic book on this subject \cite{J}.
One sentence is significant for the goals of this paper.
He writes on page~76
\begin{quote}
It was Ehresmann \ldots and his student B\'enabou \ldots who first took the decisive step in regarding complete
Heyting algebras as `generalized topological spaces'.
\end{quote}
However, Johnstone does not say {\em why} Ehresmann was led to his frame-theoretic viewpoint of 
topological spaces.\footnote{Ehresmann's paper on pseudogroups and local structures that Johnstone refers to can be found collected 
in Partie~II-1 of \cite{E} as paper 47 and, incidently,
is the only paper Ehresmann wrote in his native German.
Its title is {\em Gattungen von lokalen Strukturen}; that is, `Species of local structures'.}.
This we shall now explain.

Differential manifolds of various complexions
are defined by means of atlases whose changes of charts are required to satisfy certain conditions depending on the type of manifold
being defined. For the constructions to work, the changes of charts need to belong to a so-called {\em pseudogroup of transformations}.
Differential geometers, such as Ehresmann, were well aware in the 1950's that a whole range of local structures important in differential geometry,
could be defined in the same way, fibrations and foliations being good examples.
Furthermore, category theory, developed in the 1940's, pointed the way to describing classes of structures in general.
Ehresmann's motivation was to develop an abstract theory of local structures in geometry based on the theory of pseudogroups,
but to realize this goal, Ehresmann needed an abstract formulation of a pseudogroup of transformations.
To do this, he used ordered groupoids but it subsequently became clear that a simpler, but equivalent, formulation was possible using semigroups.

A {\em complete abstract pseudogroup}, to use the terminology of Resende \cite{R1,R2}, is a complete, infinitely distributive inverse monoid.
In this paper, we shall simply call them {\em pseudogroups}. 
Transformation pseudogroups are just pseudogroups of partial homeomorphisms of a topological space
and the idempotents of a transformation pseudogroup are just the partial identities on the open subsets of the space.

{\em Thus the partially ordered sets of idempotents of pseudogroups are frames 
and so frames arose in Ehresmann's work as the lattices of idempotents of pseudogroups.}

It was perhaps natural to disentangle frames from their roots and study them on their own terms.
But the premise of this paper is that we now need to return and generalize the foundations of the classical theory of frames to pseudogroups.
This is not an empty exercise because it has become clear that the resulting theory provides the setting for the significant applications 
of inverse semigroup theory to $C^{\ast}$-algebras.
In fact, the theory in this paper arose out of a detailed analysis of the relationships that exist between inverse semigroups, 
topological groupoids and $C^{\ast}$-algebras \cite{Exel1,Exel2,Kel1,Kel2,Law4,LMS,L,P,Ren}.

Although pseudogroups have frames of idempotents,
frame theory turned its back on pseudogroups and they do not occur at all in either \cite{J} or \cite{MM}.
Ehresmann's work on pseudogroups was generally neglected except in East Germany \cite{R}
and within inverse semigroup theory where his ideas
turned out to be extremely fertile: they form the basis of the book \cite{Law2}.

Inverse semigroups arise naturally as algebraic models of pseudogroups of transformations.
They also arise naturally in the theory of $C^{\ast}$-algebras.
This was first observed by Renault \cite{Ren} and was later developed more explicitly by Paterson \cite{P}. 
The salient idea in this context is that  $C^{\ast}$-algebras can be constructed from topological groupoids 
which in turn can be constructed from inverse semigroups.
Exel's work \cite{Exel1} is  a prime example of this fruitful approach to constructing $C^{\ast}$-algebras. 
In this way many interesting $C^{\ast}$-algebras, such as graph algebras and tiling algebras, can be constructed from inverse semigroups.
This  raises the question of the nature of the relationship between inverse semigroups and topological groupoids 
which is the central question addrsssed in this paper.

The authors' interest in this question was aroused by Kellendonk's work \cite{Kel1,Kel2} on aperiodic tilings as models of quasi-crystals.
The second author's paper \cite{L} reanalysed Paterson's work in the light of Kellendonk's and presented a general order-based approach 
to the construction of groupoids from inverse semigroups. 
This approach yields  an alternative description of Paterson's universal groupoid and at the same  time  
features a certain reduction of the universal groupoid. 
This reduction is the tiling groupoid in the case of tiling semigroups and the graph groupoid in the case of graphs. 
In this way a unified treatment of certain basic properties concerning, for example, the ideal theory  of  tiling groupoids 
and graph groupoids becomes possible.

The order-based approach  to the construction of groupoids from inverse semigroups was then developed  further by the first author 
in collaboration with Stuart Margolis and Ben Steinberg in terms of filters \cite{LMS}.
Thus the topological groupoids arising in the theory of $C^{\ast}$-algebras were groupoids of filters.
This set the stage for \cite{Law3}, where the first author showed that the topological groupoids arising in the case of the Cuntz $C^{\ast}$-algebras
could be constructed from a class of inverse monoids, called Boolean inverse monoids, in a way generalising the classical Stone duality
between unital Boolean algebras and Boolean spaces.
This was subsequently generalized to Boolean inverse semigroup in \cite{Law4} where
the Thompson groups and the Cuntz-Krieger $C^{\ast}$-algebras were shown to arise naturally from
this more general duality.

%%%%%%%%%%%%%%%%%%%%%%%%%%%%%%%%%%%%%%%%%%%%%%%%%%%%%%%%%%%%%%%%%%%%%%%%%%%%%%%%%%%%%%%%%%%%%%%%%%%%%%%%%%%%%%%%%%%%%%%%%%%%%%%%%%
\subsection{Preliminaries}

This paper unites two areas of mathematics: inverse semigroup theory and frame theory.
For background on inverse semigroups we refer the reader to \cite{Law1} and for frame theory \cite{J}.

We shall be viewing inverse semigroups very much from an order-theoretic perspective
and the order used will always be the natural partial order.
The semilattice of idempotents of the inverse semigroup $S$ is denoted $E(S)$;
more generally, if $X \subseteq S$ then $E(X) = X \cap E(S)$.
An inverse semigroup is said to be a {\em $\wedge$-semigroup} if it has binary meets of all pairs of elements.
Such semigroups were first investigated in detail by Leech \cite{Leech1,Leech2}.
The existence of joins in inverse semigroups is more subtle.
A necessary condition for a subset of an inverse semigroup to have a join is that
the elements in the set be pairwise compatible where the elements $s$ and $t$ are {\em compatible}, 
written $s \sim t$, 
if $s^{-1}t$ and $st^{-1}$ are idempotents.
We write $\mathbf{d}(s) = s^{-1}s$ and $\mathbf{r}(s) = ss^{-1}$.

An inverse semigroup is said to be {\em distributive} if it has binary joins of all compatible pairs of elements and multiplication 
distributes over the binary joins that exist.
An inverse semigroup is said to be a {\em pseudogroup} if it has joins of all non-empty compatible subsets and mutiplication
distributes over the joins that exist.
Thus pseudogroups are the complete versions of distributive inverse semigroups.
The goal of this paper is to extend the classical theory of distributive lattices \cite{J} to
distributive inverse semigroups with a view to applying them in developing a non-commutative generalization of Stone duality. 
Distributive inverse semigroups have also recently played a central role in solving a problem in universal algebra \cite{KM}.

Throughout this paper, lattices will not be assumed to have top elements,
but  if we do want to say that a lattice has a top element we shall say that it is {\em unital}.
For us, a {\em Boolean algebra} is what is often called a {\em generalized Boolean algebra}.
What are usually referred to as Boolean algebras are what we call {\em unital Boolean algebras}.
A {\em frame} is a complete infinitely distributive lattice.
The semilattice of idempotents of a distributive inverse semigroup is a distributive lattice
and that of a pseudogroup is a frame.
An inverse semigroup is said to be {\em Boolean} if it is distributive and has a Boolean algebra of idempotents.
We shall also be interested in {\em Boolean $\wedge$-semigroups} which play an important role in applications
and were the subject of \cite{Law4}.

\begin{lemma}\label{le: relative_complement} Let $S$ be a Boolean inverse semigroup.
Let $a,b \in S$ such that $b \leq a$.
Then we may construct a unique element, denoted by $a \setminus b$, such that $b$ and $a \setminus b$ are orthogonal and $a = b \vee (a \setminus b)$.
\end{lemma}
\begin{proof}
We have that $\mathbf{d}(b) \leq \mathbf{d}(a)$.
But the semilattice of idempotents of $S$ is a Boolean algebra.
Thus there exists $e \leq \mathbf{d}(a)$ such that $\mathbf{d}(a) = e \vee \mathbf{d}(b)$ and $e \wedge \mathbf{d}(b) = 0$.
Since we are working in a distributive inverse semigroup,
it follows that $a = b \vee ae$ and $b \wedge ae = 0$.
Suppose that $x \leq a$ is such that $a = b \vee x$ and $b \wedge x = 0$.
Then because we are working inside a principal order ideal we have that $\mathbf{d}(a) = \mathbf{d}(b) \vee \mathbf{d}(x)$ and
$\mathbf{d}(b) \wedge \mathbf{d}(x) = 0$.
But by uniqueness of relative complements in Booleans algebras we have that $e = \mathbf{d}(x)$ and so $x = ae$.
\end{proof}

We denote the element $ae$ by $a \setminus b$, the {\em relative complement of $b$ in $a$}.

Let $(P,\leq)$ be a poset.
A minimum element in $P$ is called  {\em zero} denoted by 0.
We shall usually assume that our posets have zeros because our inverse semigroups will always have zeros.
For $x \in P$ define
$$x^{\downarrow} = \{y \in E \colon y \leq x \},$$
the {\em principal order ideal generated by $x$},
and
$$x^{\uparrow} = \{y \in E \colon y \geq x \},$$
the {\em principal filter generated by $x$}.
We extend this notation to subsets $A \subseteq P$
and define $A^{\downarrow}$ and $A^{\uparrow}$ accordingly.
If $A = A^{\downarrow}$ is called an {\em order ideal}.
Observe that the intersection of order ideals is always an order ideal.
If $A$ is a finite set then $A^{\downarrow}$ is said to be a {\em finitely generated} order ideal.
A subset $A$ of $P$ is said to be {\em directed} if for each $a,b \in A$ there exists $c \in A$ such that $c \leq a,b$.
A {\em filter} in $P$ is a directed subset $A$ such that $A = A^{\uparrow}$; that is, it is directed and {\em closed upwards}.
A filter is called {\em proper} if it does not contain the zero.
All our filters will be assumed proper.
If $A$ is any directed subset then $A^{\uparrow}$ is a filter.
If $a$ and $b$ are elements of $P$ we write $a^{\downarrow} \cap b^{\downarrow} \neq 0$
to mean that there is some non-zero element below both $a$ and $b$.
A {\em filter} in an inverse subsemigroup $S$ is a subset $F$ that is closed upwards under the natural partial order and directed.
An {\em ultrafilter} is a maximal proper filter.
In a distributive inverse semigroup, a filter $F$ is said to be {\em prime} if
whenever $a \vee b$ exists then $a \vee b \in F$ implies that either $a \in F$ or $b \in F$.
The following characterization of ultrafilters, which is Lemma~12.3 of \cite{Exel2}, is very useful.

\begin{lemma}\label{le: ultrafilter_result} Let $F$ be a proper filter in the meet semilattice $X$.
Then $F$ is an ultrafilter iff $a \wedge x \neq 0$ for all $x \in F$ 
implies that $a \in F$.
\end{lemma}

The following is proved using Zorn's lemma.

\begin{lemma}\label{le: ultrafilters} 
Every non-zero element of an inverse semigroup belongs to an ultrafilter.
\end{lemma}

The inverse semigroup $S$ is said to be a {\em weak semilattice} 
if the intersection of any two principal order ideals is finitely generated as an order ideal.
This condition was formally introduced by Steinberg \cite{Stei}.
If $S$ is an inverse $\wedge$-semigroup then in fact $a^{\downarrow} \cap b^{\downarrow} = (a \wedge b)^{\downarrow}$.
Thus inverse semigroups that are weak semilattices generalize inverse $\wedge$-semigroups.
In a weak semilattice, the intersection of any finite number of principal order ideals is finitely generated as an order ideal.

Let $S$ and $T$ be distributive inverse semigroups.
Every homomorphism $\theta \colon S \rightarrow T$ has the property that $a \sim b$ implies that $\theta (a) \sim \theta (b)$.
Such a homomorphism is called a {\em morphism} of distributive inverse semigroups
if $\theta (a \vee b) = \theta (a) \vee \theta (b)$ whenever $a \sim b$.
{\em Morphisms of pseudogroups} are defined analogously.

We now describe a construction due to Boris Schein \cite{S,Law2} and its finitary analogue.
Let $S$ be an inverse semigroup.
Define $\mathsf{C}(S)$ to be the set of all compatible order ideals of $S$ with subset multiplication as the operation.
Then $\mathsf{C}(S)$ is a pseudogroup and the map $\iota \colon S \rightarrow \mathsf{C}(S)$, 
given by $s \mapsto s^{\downarrow}$, is a semigroup homomorphism.

We denote by $\mathsf{D}(S)$ the set of all finitely-generated compatible order ideals of $S$.
Under subset multiplication, $\mathsf{D}(S)$ is a distributive inverse semigroup.
The function $\iota \colon S \rightarrow \mathsf{D}(S)$
defined by $\iota (s) = s^{\downarrow}$ is a semigroup homomorphism.

The following theorem tells us that we can manufacture many examples of distributive inverse semigroups and pseudogroups.

\begin{theorem}[Distributive and Schein completions]\label{the: distributive_completions} \mbox{}

\begin{enumerate}

\item If $\theta \colon S \rightarrow T$ 
is any homomorphism from an inverse semigroup $S$ to a
distributive inverse semigroup $T$ then there is a unique morphism of distributive inverse semigroups
$\theta^{\ast} \colon \mathsf{D}(S) \rightarrow T$ such that $\theta^{\ast} \iota = \theta$.
In addition, $\mathsf{D}(S)$ is an inverse $\wedge$-semigroup if and only if $S$ is a weak semilattice.

\item If $\theta \colon S \rightarrow T$ 
is any homomorphism from an inverse semigroup $S$ to a pseudogroup $T$ 
then there is a unique morphism of pseudogroups
$\theta^{\ast} \colon \mathsf{C}(S) \rightarrow T$ such that $\theta^{\ast} \iota = \theta$.

\end{enumerate}
\end{theorem}

We call $\mathsf{C}(S)$ the {\em Schein completion} of $S$ and $\mathsf{D}(S)$
the {\em distributive completion} of $S$.

In distributive inverse semigroups, 
we assume that binary joins exist but we make no assumption about the existence of meets
in general.
However, for those finite subsets where meets do exist, 
the following lemma shows that they behave as expected with respect to joins.
In pseudogroups, binary meets, in fact arbitrary non-empty meets, always exist by Proposition~2.10(2) of \cite{R2}.
Part (2) of the following is from \cite{R3} and part (1) is just the finitary analogue.

\begin{lemma}\label{le: meets_and_joins} \mbox{}
\begin{enumerate}

\item Let $S$ be a distributive inverse semigroup.
Suppose that $a \vee b$ and $c \wedge (a \vee b)$ both exist.
Then $c \wedge a$ and $c \wedge b$ both exist, 
the join $(c \wedge a) \vee (c \wedge b)$ exists and
$$c \wedge (a \vee b)
=
(c \wedge a) \vee (c \wedge b).$$

\item Let $S$ be a pseudogroup.
Suppose that $\bigvee_{i} a_{i}$ and $b \wedge \left( \bigvee_{i} a_{i}  \right)$ both exist.
Then $\bigvee_{i} b \wedge a_{i}$ exists and 
$$b \wedge \left( \bigvee_{i} a_{i}  \right)
=
\bigvee_{i} b \wedge a_{i}.$$

\end{enumerate}
\end{lemma}

%%%%%%%%%%%%%%%%%%%%%%%%%%%%%%%%%%%%%%%%%%%%%%%%%%%%%%%%%%%%%%%%%%%%%%%%%%%%%%%%%%%%%%%%%%%%%%%%%%%%%%%%%%%%%%%%%%%
A  {\em groupoid} $G$ is a category in which every morphism is an isomorphism.
Our groupoids will be sets and the set of identities of $G$ will be denoted by $G_{o}$.
Each groupoid comes with the maps
$\mathbf{d} \colon G\rightarrow G_{o}$ and $\mathbf{r} \colon G\rightarrow G_{o}$ defined by $\mathbf{d}(x) = x^{-1}x$ and $\mathbf{r}(x) = xx^{-1}$.
A {\em local bisection} $A$ of a groupoid $G$ is a subset such that $A^{-1}A,AA^{-1} \subseteq G_{o}$.
The product of two local bisections is a local bisection.
In fact, we have the following well-known result.

\begin{lemma} 
The set of local bisections of a discrete groupoid $G$ forms a Boolean inverse $\wedge$-semigroup.
\end{lemma}

If $G$ carries a topology making the multiplication and inversion continuous, it is called a \textit{topological groupoid}. 
The most important class of topological groupoids are the {\em \'etale groupoids}.
Classically, an \'etale groupoid is a topological groupoid in which the domain map is a local homeomorphism.
However, we shall use the following results due to Resende in our dealings with \'etale groupoids.
Part (1) below is Theorem~5.18 of \cite{R2} and part (2) from Exercise~I.1.8 of \cite{R1}.

\begin{proposition}\label{prop: etale_groupoids}\mbox{}
\begin{enumerate}

\item A topological groupoid $G$ is \'etale if $G_{o}$ is an open set and the product of any two open sets is an open set.

\item In an \'etale groupoid, the open local bisections form a basis. 

\end{enumerate}
\end{proposition}

For more background on topological groupoids, we refer the reader to \cite{R1,R2}.
Let $G$ be a topological groupoid.
We denote by $\mathsf{B}(G)$ the set of all open local bisections of $G$.
The proof of the following is straightforward.

\begin{proposition}\label{prop: B} Let $G$ be an \'etale groupoid.
Then $\mathsf{B}(G)$ is a pseudogroup.
\end{proposition}

%%%%%%%%%%%%%%%%%%%%%%%%%%%%%%%%%%%%%%%%%%%%%%%%%%%%%%%%%%%%%%%%%%%%%%%%%%%%%%%%%%%%%%%%%%%%%%%%%%%%%%%%%%%%%%%%%%%%%%%%%%%%%
We shall now recall some results that are well-known in the unital case;
however, we shall also need them in the non-unital case.

\begin{lemma}[Non-unital results]\label{le: lattice_filters} \mbox{} 
\begin{enumerate}

\item In a distributive lattice every ultrafilter is prime.

\item A distributive lattice is Boolean if and only if every prime filter is an ultrafilter.

\end{enumerate}
\end{lemma}
\begin{proof} (1) Let $F$ be an ultrafilter in a distributive lattice $D$.
Suppose that $c = a \vee b \in F$ but $a,b \notin F$.
By Lemma~\ref{le: ultrafilter_result}, there exists $f_{a} \in F$ such that $a \wedge f_{a} = 0$ and $f_{b} \in F$ such that $f_{b} \wedge b = 0$.
But $F$ is a filter and so $f = f_{a} \wedge f_{b} \in F$ which means that $f \wedge a = 0 = f \wedge b$.
But $f \wedge c = (f \wedge a) \vee (f \wedge b)$.
The lefthand-side is non-zero since it is a meet of elements in a filter,
but the righthand-side is zero.
This is a contradiction and so either $a \in F$ or $b \in F$.

%We prove first the unital case of our result.
%Let $D$ be a unital boolean algebra and let $P$ be a prime filter.
%Let $a \in D$ be such that $a \wedge p \neq 0$ for all $p \in P$.
%Since $D$ is a unital boolean algebra we have that $1 = a \vee a'$.
%But $1 \in P$ and $P$ is a prime filter.
%Thus either $a \in P$ or $a' \in P$.
%We cannot have $a' \in P$ because $a \wedge a' = 0$ which contradicts our assumption.
%Thus $a \in P$ and so by Theorem~1.5, we have proved that $P$ is an ultrafilter.
%Conversely, let $D$ be a unital distributive lattice such that every prime filter is an ultrafilter.
%We prove that $D$ is boolean. 
%We use Corollary~4.9 of \cite{J}.
%Thus we shall prove that every prime ideal is maximal.
%Let $I$ be a prime ideal in $D$.
%Suppose that it is not a maximal proper ideal.
%Then $I \subseteq J$ a maximal proper ideal by Lemma~2.3 of \cite{J}.
%By Corollary~2.4 of \cite{J}, the ideal $J$ is also prime.
%We have that $D \setminus J \subseteq D \setminus I$ where by Proposition~2.2 of \cite{J},
%both $D \setminus J$ and $D \setminus I$ are prime filters.
%By assumption, both must be ultrafilters and so $D \setminus J = D \setminus I$ giving $I = J$.
%It follows that $I$ is a maximal proper ideal, as required.

(2). The result holds in the unital case by part (ii) of Corollary~4.9 of \cite{J}.
Let $D$ be a distributive lattice in which every prime filter is maximal.
We prove that $D$ is a Boolean algebra.
To do this we have to prove that $e^{\downarrow}$ is a Boolean algebra for each $e \in D$.
This can be achieved by showing that every prime filter in $e^{\downarrow}$ is an ultrafilter.
We make some observations first.
Let $F \subseteq e^{\downarrow}$ be a filter in $e^{\downarrow}$.
Then $F^{\uparrow}$ is a filter in $D$.
In addition, if $F \subseteq G \subseteq e^{\downarrow}$ are filters then $F^{\uparrow} \subseteq G^{\uparrow}$,
and if $F^{\uparrow} = G^{\uparrow}$ then $F = G$.
Let $F \subseteq e^{\downarrow}$ be a prime filter in $e^{\uparrow}$.
Then $F^{\uparrow}$ is a prime filter in $D$.
If $F$ is not an ultrafilter in $e^{\downarrow}$ then there is a filter $G$ such that $F \subseteq G \subseteq e^{\downarrow}$.
But $F^{\uparrow}$ a prime filter implies by assumption that $F^{\uparrow}$ is an ultrafilter and so $F^{\uparrow} = G^{\uparrow}$ giving $F = G$.
We have therefore proved that in the unital distributive lattice $e^{\downarrow}$ every prime filter is an ultrafilter and so $e^{\downarrow}$
is a unital Boolean algebra, as required. 

Conversely, let $D$ be a Boolean algebra.
We prove that every prime filter is an ultrafilter.
Let $P$ be a prime filter and let $a \in D$ be an element such that $a \wedge p \neq 0$ for all $p \in P$.
We shall prove that $a \in P$ from which it follows that $P$ is an ultrafilter by Lemma~\ref{le: ultrafilter_result}.
Choose $e \in P$ arbitrarily and put $P' = \{e \wedge p \colon p \in P \}$.
Then $P'$ is a prime filter in the unital Boolean algebra $e^{\downarrow}$.
Observe that the element $a \wedge e$ has a non-empty meet with every element of $P'$.
But in a unital Boolean algebra we have seen that every prime filter is an ultrafilter and so $P'$ is an ultrafilter
and thus by Lemma~\ref{le: ultrafilter_result} we have that $a \wedge e \in P'$.
It follows that $a \in P$, as required. 
\end{proof}

We shall need some basic results from \cite{J} though we shall phrase them slightly differently.
Let $X$ be a topological space.
For each $x \in X$ define $O_{x}$ to be the set of all open subsets of $X$ that contain $x$.
Clearly, $O_{x}$ is a filter; it is in fact a {\em completely prime} filter meaning that
if $\bigcup_{i \in I} U \in O_{x}$ then $U_{i} \in O_{x}$ for some $i$.
The space $X$ is said to be {\em sober} if it satisfies two conditions:
first, every completely prime filter is of the form $O_{x}$ for some point $x$ and  $O_{x} = O_{y}$ if and only if $x = y$.
In other words, each point of the space is uniquely determined by the open sets that contain it.

A topological space $X$ is said to be {\em spectral} if it is sober and has a basis of compact-open sets
that is closed under finite non-empty intersections.
{\em Observe that we do not assume that $X$ is compact as, for example, in \cite{J}.}

Let $L$ be a distributive lattice.
Denote by $\mathsf{S}(L)$ the set of all prime filters of $L$.
For each $a \in L$ denote by $X_{a}$ the set of all prime filters that contain $a$.
The set $\tau = \{X_{a} \colon a \in X \}$ forms the basis of a topology on $\mathsf{S}(L)$. 
The sets $X_{a}$ are compact-open, the set $\tau$ is closed under interesections, and $\mathsf{S}(L)$ is sober.
It follows that $\mathsf{S}(L)$ is a spectral space.
Let $X$ be a spectral space.
Then the set $\mathsf{L}(X)$ of compact-open sets is a distributive lattice.

A homomomorphism of distributive lattices is called {\em proper} if every element of the codomain lies
below an element of the image.
A continuous map is {\em coherent} if the inverse images of compact-open sets are compact-open.

\begin{theorem}[Stone duality for non-unital distributive lattices]\label{the: stone_duality_distributive}
The category of distributive lattices and their proper homomorphisms is dually equivalent to the category of spectral spaces
and their coherent continuous maps.
Under this duality, Boolean algebras correspond to Hausdorff spectral spaces.
\end{theorem}

A Hausdorff spectral space is called a {\em Boolean space}.

%%%%%%%%%%%%%%%%%%%%%%%%%%%%%%%%%%%%%%%%%%%%%%%%%%%%%%%%%%%%%%%%%%%%%%%%%%%%%%%%%%%%%%%%%%%%%%%%%%%%%%%%%%%%%%%%%%%%%%%%%%%%%%%%%%%%%%%%%%%%%%%%%%%%%%
\subsection{Outline of paper}

We describe each section in turn highlighting the key results proved.\\

\noindent
{\bf 2. Prime filters and ultrafilters } This deals with the theory of prime filters and ultrafilters in distributive inverse semigroups
and generalizes well-known results for distributive lattices such as those to be found in \cite{J}. 
In Proposition~\ref{prop: semigroup_filters} we characterize Boolean inverse semigroups
as those distributive inverse semigroups in which every prime filter is an ultrafilter.
In Proposition~\ref{prop: key_result}, we prove that prime filters can be used to separate elements in the following sense. 
If $a \neq b$ then either $a \nleq b$ or $b \nleq a$.
We prove that if $b \nleq a$ then there is a prime filter that contains $b$ and omits $a$.\\

\noindent
{\bf 3. Non-commutative Stone duality }This section and the next form the core of the paper.
We begin by defining the topological objects that take part in our dualties.
These are classes of \'etale groupoids that have a spectral space of identities and so we refer to them simply as {\em spectral groupoids}.
Theorem~\ref{the: duality_theorem} states that a suitable category of distributive inverse semigroups 
is dually equivalent to a suitable category of spectral groupoids, where `suitable' refers to a choice of morphisms.
Two further theorems are proved Theorem~\ref{the: Big_Theorem_Boolean} and Theorem~\ref{the: Big_Theorem_Meet_Boolean}
which specialize the above duality to Boolean inverse semigroups and Boolean inverse $\wedge$-semigroups, respectively.
The above theorems are proved directly but also from a more general perspective via
an adjunction between pseudogroups and \'etale groupoids that shows how our work connects with that of Resende \cite{R1,R2}.\\

\noindent
{\bf 4.~Adjunctions and duality }We set up an adjunction between pseudogroups and \'etale groupoids and rederive
our main duality theorem for distributive inverse semigroups obtained in Section~3 by extending the notion of coherence from frames to pseudogroups.
This connects our work with that of Pedro Resende \cite{R1,R2} and shows that our work is complementary to his.
However, Resende's theory does not include any morphisms whereas ours does.\\ 

\noindent
{\bf 5.~Booleanizations and Paterson's universal groupoid }This section was inspired by Paterson's book \cite{P}.
The question we answer here is how {\em Paterson's universal groupoid} $\mathsf{G}_{u}(S)$ of an inverse semigroup $S$
can be understood within our framework. We prove that it is obtained as part of the process of {\em Booleanizing}
an inverse semigroup; that is, of finding the freest Boolean inverse semigroup generated by an inverse semigroup.
As such, it requires a generalization of what is called the {\em patch topology} in \cite{J}.
The work of this section was directly inspired by the calculations on pp~190--191 of \cite{P}.\\

\noindent
{\bf 6.~Tight completions and Exel's tight groupoid }Our duality theory applies to distributive and Boolean inverse semigroups.
This raises the question of how arbitrary inverse semigroups may be completed to distributive inverse semigroups.
To do this we generalize the concept of a covering \cite{J} to inverse semigroups.
The general idea is that an inverse semigroup equipped with a coverage can be regarded as a presentation of a pseudogroup
and, under suitable conditions on the coverage, as a presentation of a distributive inverse semigroup.
Although we describe elements of the general theory, we concentrate on one particular coverage that we call the
{\em tight coverage} which has its roots in the second author's paper \cite{L} and the work of Exel \cite{Exel1, Exel2}.
This coverage seems particularly important in dealing with the connections between inverse semigroups and $C^{\ast}$-algebras.
Our main theorem is Theorem~\ref{the: tight_completions} that shows that every inverse semigroup equipped with its tight coverage can be completed,
in a suitable sense, to a distribuitve inverse semigroup called the {\em tight completion}.
We say that an inverse semigroup is {\em pre-Boolean} if its tight completion is actually Boolean.
Many naturally occurring examples of inverse semigroups are pre-Boolean, such as the polycyclic inverse monoids.
Their tight completions are what we call the Cuntz inverse monoids and the groups of units of such monoids are Thompson groups
which shows that the theory of this paper has a wider significance.
If the tight completion of an inverse semigroup is not Boolean one can still construct from it the
associated spectral groupoid. When this is Booleanized via the patch-topology one obtains Exels' tight groupoid.\\

\begin{center}
{\bf Notation }
\end{center}

\begin{description}

\item[$\mathsf{D}(S)$] The distributive completion of the inverse semigroup $S$.

\item[$\mathsf{C}(S)$] The pseudogroup completion of the inverse semigroup $S$; the {\em Schein completion}.

\item[$\mathsf{B}(G)$] The pseudogroup of all open local bisections of the \'etale groupoid $G$.

\item[$\mathsf{KB}(G)$] The ordered groupoid of all compact-open local bisections of the \'etale groupoid $G$.

\item[$\mathsf{G}(S)$] The \'etale groupoid of all proper filters on the inverse semigroup $S$.

\item[$\mathsf{G}_{u}(S)$] Paterson's universal groupoid of the inverse semigroup $S$.

\item[$\mathsf{G}_{P}(S)$] The spectral groupoid of the distributive inverse semigroup $S$; the prime spectrum of $S$.

\item[$\mathsf{G}_{P}(S)^{\dagger}$] The above groupoid with the patch topology.

\item[$\mathsf{G}_{CP}(S)$] The \'etale groupoid of the pseudogroup $S$; the completely prime spectrum of $S$.

\item[$\mbox{\rm Idl}(S)$] The pseudogroup of all $\vee$-closed compatible order ideals of the distributive inverse semigroup $S$.

\item[$\mathsf{K}(S)$] The set of finite elements of the pseudogroup $S$.

\end{description}

%%%%%%%%%%%%%%%%%%%%%%%%%%%%%%%%%%%%%%%%%%%%%%%%%%%%%%%%%%%%%%%%%%%%%%%%%%%%%%%%%%%%%%%%%%%%%%%%%%%%%%%%%%%%%%%%%%%%%%%%%%%%%%%%%%%%%%%%%
\section{Prime filters and ultrafilters}

The goal of this section is to generalize the classical theory of prime filters on distributive lattices to distributive inverse semigroups.

Let $F$ be a filter in $S$.
Define $\mathbf{d}(F) = (A^{-1}A)^{\uparrow}$ and $\mathbf{r}(A) = (AA^{-1})^{\uparrow}$;
both of these sets are filters.

\begin{lemma}\label{le: idempotent_filters} Let $F$ be a filter in an inverse semigroup $S$.
Then $x \in \mathbf{d}(F)$ if and only if $a^{-1}a \leq x$ for some $a \in A$.
\end{lemma}
\begin{proof} By assumption, $a^{-1}b \leq x$ for some $a,b \in A$.
But $A$ is down directed and so there is $c \in A$ such that $c \leq a,b$.
It follows that $c^{-1}c \leq x$, as required.
\end{proof}

The following result summarizes some useful properties of filters;
part (1) is Lemma~3.3 of \cite{LMS},
part (2) is Proposition~1.4 of \cite{Law0},
part (3) is Lemma~2.11 of \cite{Law2},
part (4) is Proposition~1.5 of \cite{Law0} and Lemma~3.4 of \cite{LMS}.

\begin{lemma}\label{le: filter_properties} Let $S$ be an inverse semigroup.
\begin{enumerate}

\item For every filter $F = FF^{-1}F$.

\item Let $F$ be a filter. Then $F = (a \mathbf{d}(F))^{\uparrow}$ for any $a \in F$.

\item If $F$ and $G$ are filters such that $\mathbf{d}(F) = \mathbf{d}(G)$ and $F \cap G \neq \emptyset$ then $F = G$,
and dually.

\item The filter $F$ contains an idempotent if and only if it is also an inverse subsemigroup, 
in which case $F = E(F)^{\uparrow}$.
\end{enumerate}
\end{lemma}

Filters that contain idempotents will be called {\em idempotent filters}.
Both $\mathbf{d}(A)$ and $\mathbf{r}(A)$ are idempotent filters.

\begin{lemma}\label{le: filter} Let $A$ be a filter in a distributive inverse semigroup $S$.
\begin{enumerate}

\item $A$ is prime (respectively, an ultrafilter) if and only if $A^{-1}$ is prime (respectively, an ultrafilter).

\item $A$ is prime (respectively, an ultrafilter) if and only if $\mathbf{d}(A)$ is prime (respectively, an ultrafilter).

\end{enumerate}
\end{lemma}
\begin{proof} (1) This is straightforward in both cases.

(2) Suppose that $A$ is prime.
We prove that $\mathbf{d}(A) = (A^{-1}A)^{\uparrow}$ is prime.
Let $x = x_{1} \vee x_{2} \in \mathbf{d}(A)$.
Then $a^{-1}a \leq x$ for some $a \in A$ by Lemma~\ref{le: idempotent_filters}.
Clearly $a^{-1}a = xa^{-1}a$ and so by distributivity we have that $a^{-1}a = x_{1}a^{-1}a \vee x_{2}a^{-1}a$.
Thus again by distributivity,
$a = ax_{1}a^{-1}a \vee ax_{2}a^{-1}a$.
By assumption, $A$ is prime and so without loss of generality $ax_{1}a^{-1}a \in A$.
Thus $ax_{1} \in A$ since $A$ is upwardly closed.
However $a^{-1}ax_{1} \in \mathbf{d}(A)$.
Thus $x_{1} \in \mathbf{d}(A)$, again by upward closure, as required.
Suppose now that $\mathbf{d}(A)$ is prime.
We prove that $A$ is prime.
Let $a = a_{1} \vee a_{2} \in A$.
Then $\mathbf{d}(a) = \mathbf{d}(a_{1}) \vee \mathbf{d}(a_{2}) \in \mathbf{d}(A)$.
Where we use standard properties of compatible joins \cite{Law2}.
Without loss of generality, we have that $\mathbf{d}(a_{1}) \in \mathbf{d}(A)$.
It follows that $a_{1} = a\mathbf{d}(a_{1}) \in A$, as required.

Suppose that $A$ is an ultrafilter in $S$.
We prove that $\mathbf{d}(A)$ is an ultrafilter.
Let $\mathbf{d}(A) \subseteq H$.
Necessarily, $H$ is an idempotent filter.
Let $a \in A$ be arbitrary.
Then $B = (aH )^{\uparrow}$ is a filter and $A \subseteq B$. 
By assumption, $A = B$ and so $H = \mathbf{d}(A)$, as required.
Suppose now that $A$ is a filter such that $\mathbf{d}(A)$ is an ultrafilter.
Let $A \subseteq B$.
Then $\mathbf{d}(A) \subseteq \mathbf{d}(B)$.
By assumption, $\mathbf{d}(A) = \mathbf{d}(B)$ and so $A = B$, as required.
\end{proof}

We may now generalize Lemma~\ref{le: lattice_filters} to distributive inverse semigroups.

\begin{proposition}\label{prop: semigroup_filters} Let $S$ be a distributive inverse semigroup.
\begin{enumerate}

\item Every ultrafilter in $S$ is a prime filter. 

\item The semigroup $S$ is Boolean if and only if every prime filter is an ultrafilter.
\end{enumerate}
\end{proposition}
\begin{proof}
(1) By Lemma~\ref{le: filter}, $F$ is a prime filter (respectively, ultrafilter) in $S$ 
if and only if $\mathbf{d}(F) $ is a prime idempotent filter
(respectively, idempotent ultrafilter).
Next observe that $G$ is an idempotent prime filter (respectively, ultrafilter) in $S$ 
if and only if $E(G)$ is a prime filter in $E(S)$ (respectively, ultrafilter).
We now apply part (1) of Lemma~\ref{le: lattice_filters}

(2) This follows by the argument in (1) above combined with part (2) of Lemma~\ref{le: lattice_filters}.
\end{proof}

%%%%%%%%%%%%%%%%%%%%%%%%%%%%%%%%%%%%%%%%%%%%%%%%%%%%%%%%%%%%%%%%%%%%%%%%%%%%%%%%%%%%%%%%%%%%%%%%%%%%%%%%%%%%%%%%%%%%%%%%%%%%%%%%%%%%%%%%%%%%%%
An order ideal in a distributive inverse semigroup is said to be {\em $\vee$-closed} 
if it is closed under joins of its finite non-empty compatible subsets.
We say that an order ideal of $S$ is {\em proper} if it is not the whole of $S$.
Let $A$ be an order ideal.
Denote by $A^{\vee}$ the set of all joins of non-empty finite compatible subsets of $A$.
Clearly, $A \subseteq A^{\vee}$, and if $A \subseteq B$ then $A^{\vee} \subseteq B^{\vee}$.

\begin{lemma} Let $A$ be an order ideal.
Then $A^{\vee}$ is an order ideal which is $\vee$-closed.
\end{lemma}
\begin{proof} Let $b \leq a \in A^{\vee}$.
By assumption, $a = \bigvee_{i=1}^{n} a_{i}$ for some non-empty finite compatible subset $\{a_{1}, \ldots, a_{n}\} \subseteq A$.
By distributivity, it follows that 
$b =  \bigvee_{i=1}^{n} a_{i} \mathbf{d}(b)$.
But $a_{i} \mathbf{d}(b) \in A$ for each $i$ since $A$ is an order ideal.
It follows that $b \in A^{\vee}$ and $A^{\vee}$ is itself an order ideal.

Let $a,b \in A^{\vee}$ be a compatible pair.
By assumption, $a = \bigvee_{i=1}^{m}a_{i}$ and $b = \bigvee_{j=1}^{n}b_{j}$
where $\{a_{1}, \ldots, a_{m}\}, \{b_{1}, \ldots, b_{n} \} \subseteq A$ are finite compatible non-empty subsets.
But $a \vee b =  ( \bigvee_{i=1}^{m}a_{i} ) \vee ( \bigvee_{j=1}^{n}b_{j} )$.
It follows that $\{a_{1}, \ldots, a_{m}, b_{1}, \ldots, b_{n} \}$ is a finite compatible non-empty subset of $A$ and so
$a \vee b \in A^{\vee}$.
\end{proof}

We call $A^{\vee}$ the {\em $\vee$-closure of $A$}.

\begin{lemma} Let $I$ be a $\vee$-closed order ideal and $a$ an arbitrary element.
Then $I \cup a^{\downarrow}$ is an order ideal and
$$(I \cup a^{\downarrow})^{\vee} = \{ x \vee b \colon x \in I, b \leq a \mbox{ and } x \sim b \}.$$ 
\end{lemma}
\begin{proof}
Clearly, $I \cup a^{\downarrow}$ is an order ideal.
It is also clear that 
$$\{ x \vee b \colon x \in I, b \leq a \mbox{ and } x \sim b \}
\subseteq 
(I \cup a^{\downarrow})^{\vee}.$$
It remains to prove the reverse inclusion.
Let $\{x_{1}, \ldots, x_{m},b_{1}, \ldots, b_{n} \}$ be a compatible subset of $I \cup a^{\downarrow}$
where $\{x_{1}, \ldots, x_{m}\} \subseteq I$ and $\{b_{1}, \ldots, b_{n} \} \subseteq a^{\downarrow}$. 
Let $x = \bigvee_{i=1}^{m}x_{i}$ and $b = \bigvee_{j=1}^{n} b_{j}$.
Then $x$ and $b$ are compatible, $x \in I$ since $I$ is a $\vee$-closed order ideal, $b \leq a$
and $x \vee b = ( \bigvee_{i=1}^{m} x_{i} ) \vee ( \bigvee_{j=1}^{n} b_{j} )$. 
\end{proof}

It can easily be verified that the union of a totally ordered set of $\vee$-closed order ideals
of an inverse semigroup is again a $\vee$-closed order ideal.
The proof of the following result now follows from Zorn's Lemma.

\begin{lemma}\label{le: maximal} Let $S$ be an inverse semigroup.
Let $I$ be a $\vee$-closed order ideal of $S$ and let $F$ be a filter disjoint from $I$.
Then there is a $\vee$-closed order ideal $J$ maximal with respect to the two conditions:
(1) $I \subseteq J$ and (2) $J \cap F = \emptyset$.
\end{lemma}

An order ideal $P$ of an inverse semigroup $S$ is said to be {\em prime}
if $a^{\downarrow} \cap b^{\downarrow} \subseteq P$ implies that either $a \in P$ or $b \in P$.

\begin{lemma}\label{le: ideal_filter} Let $S$ be an inverse semigroup.
Then a subset $F$ is a proper prime filter if and only if $S \setminus F$ is a proper $\vee$-closed prime order ideal.
\end{lemma}
\begin{proof}
Suppose that $F$ is a prime filter.
We prove that $P = S \setminus F$ is a $\vee$-closed order ideal.
Let $a \in P$ and $b \leq a$.
Suppose that $b \notin P$.
Then $b \in F$ and so $a \in F$, which is a contradiction.
Thus $P$ is an order ideal.
Suppose that $a^{\downarrow} \cap b^{\downarrow} \subseteq P$ and that $a,b \in F$.
Then since $F$ is a filter there exists $c \in F$ such that $c \leq a,b$.
But $c \in P$ which is a contradiction.
Finally, suppose that $a,b \in P$ and that $a$ and $b$ are compatible.
If $a \vee b \in F$ then either $a$ or $b$ is in $F$.
It follows that $a \vee b \in P$ and so $P$ is a $\vee$-closed prime ideal.

Conversely, suppose that $P$ is a $\vee$-closed prime ideal.
We prove that $F = S \setminus P$ is a prime filter.
Let $a \in F$ and $a \leq b$.
If $b \in P$ then $a \in P$ and so $b \in F$.
Let $a,b \in F$.
If $a^{\downarrow} \cap b^{\downarrow} \subseteq P$ then either $a$ or $b$ is in $P$.
It follows that there must exist $c \leq a,b$ such that $c \in F$.
Finally, suppose that $a \vee b \in F$.
If $a,b \in P$ then $a \vee b \in P$ so at least one of $a$ or $b$ belongs to $F$.
Thus $F$ is a prime filter.
\end{proof}

\begin{lemma}\label{le: maximal_prime} Let $S$ be a distributive inverse semigroup.
Let $F$ be a filter in $S$ and let $P$ be a $\vee$-closed order ideal of $S$ maximal amongst all $\vee$-closed order ideals disjoint from $F$.
Then $P$ is a prime $\vee$-closed order ideal.
\end{lemma}
\begin{proof}
 Assume that $a^{\downarrow} \cap b^{\downarrow} \subseteq P$.
Define
$$P_{1} = [P \cup a^{\downarrow}]^{\vee}
\text{ and }
P_{2} = [P \cup b^{\downarrow}]^{\vee}.$$
Both are well-defined $\vee$-closed order ideals that contain $P$.
Assume, for the sake of argument, that both intersect the filter $F$ in the elements $f_{1}$ and $f_{2}$ respectively.
We may write
$$f_{1} = p_{1} \vee x_{1}
\text{ and }
f_{2} = p_{2} \vee y_{1}$$
where $p_{1},p_{2} \in P$ and $x_{1} \leq a$ and $y_{1} \leq b$.
Since $F$ is a filter there is an element $f \in F$ such that $f \leq f_{1},f_{2}$.
Thus we may write
$$f = (p_{1} \vee x_{1}) \mathbf{d}(f)
\text{ and }
f = (p_{2} \vee y_{1}) \mathbf{d}(f).$$
By distributivity
$$f = p_{1} \mathbf{d}(f) \vee x_{1}\mathbf{d}(f)
\text{ and }
f = p_{2} \mathbf{d}(f)  \vee y_{1}\mathbf{d}(f).$$
Now $f = f \wedge f$.
Thus by Lemma~\ref{le: meets_and_joins}, 
we have that
$$f =
(p_{1} \mathbf{d}(f) \wedge p_{2} \mathbf{d}(f))
\vee
(p_{1}\mathbf{d}(f) \wedge y_{1} \mathbf{d}(f))
\vee
(x_{1}\mathbf{d}(f) \wedge p_{2} \mathbf{d}(f))
\vee
(x_{1} \mathbf{d}(f) \wedge y_{1} \mathbf{d}(f)).
$$
Each term belongs to $P$, the final term by assumption.
Hence $f \in P$ which is a contradiction.
Thus either $P_{1}$ or $P_{2}$ is disjoint from $F$.
Without loss of generality we may assume that $P_{1}$ is disjoint from $F$.
But then we must have that $P_{1} = P$ and so $a \in P$.
It follows that $P$ is a prime $\vee$-closed order ideal.
\end{proof}

We now come to the key result about the properties of prime filters in distributive inverse semigroups.

\begin{proposition}\label{prop: key_result}
Let $a,b \in S$, a distributive inverse semigroup, such that $b \nleq a$.
Then there exists a prime filter that contains $b$ and omits $a$. 
\end{proposition}
\begin{proof}
Consider the filter $b^{\uparrow}$ and the order ideal $a^{\downarrow}$ which is clearly a $\vee$-closed order ideal.
By assumption, $b^{\uparrow} \cap a^{\downarrow} = \emptyset$.
By Lemma~\ref{le: maximal}, we may find a $\vee$-closed order ideal $J$ such that $a^{\downarrow} \subseteq J$ and $J \cap b^{\uparrow} = \emptyset$
and maximal with respect to these properties.
By Lemma~\ref{le: maximal_prime} above, $J$ is a prime $\vee$-closed order ideal.
Thus by Lemma~\ref{le: ideal_filter}, $S \setminus J$ is a prime filter in $S$.
By construction this prime filter contains $b$ and omits $a$, as required.
\end{proof}

%%%%%%%%%%%%%%%%%%%%%%%%%%%%%%%%%%%%%%%%%%%%%%%%%%%%%%%%%%%%%%%%%%%%%%%%%%%%%%%%%%%%%%%%%%%%%%%%%%%%%%%%%%%%%%%%%%%%%%%%%%%%%%%%%%%%%%%%%%%%%
\section{Non-commutative Stone duality}

Classical Stone dualities link order-theoretic structures to topological ones.
In our generalization, the order-theoretic structures are replaced by appropriate inverse semigroups
and the topological structures by suitable topological groupoids.
Special cases were the subject of \cite{Law2,Law3} 
where the inverse semigroup was a Boolean inverse $\wedge$-semigroup.
In this section, we deal with the general case.

Let $G$ be an \'etale groupoid.
We denote by $\mathsf{KB}(G)$ the set of all compact-open local bisections in $G$.
This is therefore a subset of $\mathsf{B}(S)$.

The following is the one place where we use a little of the theory of ordered groupoids and inductive groupoids;
we refer the reader to \cite{Law1} for details.

\begin{lemma}\label{le: reg} Let $G$ be an \'etale groupoid.
\begin{enumerate}

\item The set of compact-open local bisections $\mathsf{KB}(G)$ is an ordered groupoid

\item The ordered groupoid $\mathsf{KB}(G)$ is an inverse semigroup if and only if the intersection
of any two compact-open subsets of $G_{o}$ is again a compact-open subset,
in which case, $\mathsf{KB}(G)$ is a distributive inverse semigroup.

\end{enumerate}
\end{lemma}
\begin{proof} (1). We prove first that if $A$ is a compact-open local bisection so too is $A^{-1}A$.
Clearly, $A^{-1}$ is a compact-open local bisection and so $A^{-1}A$ is an open local bisection.
It remains to prove that $A^{-1}A$ is compact.
We shall use the fact \cite{R1}, that an \'etale groupoid has a basis of open local bisections.
Suppose that $A^{-1}A \subseteq \bigcup_{i \in I} U_{i}$ where the $U_{i}$ are open local bisections.
Then $A \subseteq \bigcup_{i \in I} AU_{i}$ where $AU_{i}$ are open local bisections.
By assumption, $A  \bigcup_{i=1}^{m} AU_{i}$,
and so $A^{-1}A \subseteq  \bigcup_{i=1}^{m} A^{-1}AU_{i} \subseteq  \bigcup_{i=1}^{m}U_{i}$,
as required.

Let $A$ and $B$ be two compact-open local bisections such that $A^{-1}A = BB^{-1}$.
We shall prove that $AB$ is a compact-open local bisection.
Since $AB$ is \'etale it is immediate that $AB$ is open,
and $AB$ is always a local bisection whenever $A$ and $B$ are.
It remains to show that $AB$ is compact.
We shall use the fact \cite{R1}, that an \'etale groupoid has a basis of open local bisections.
Suppose that $AB = \bigcup_{i \in I} U_{i}$ where the $U_{i}$ are open local bisections.
Then $A^{-1}A =  \bigcup_{i \in I} A^{-1}U_{i}B^{-1}$.
But $A^{-1}A$ is compact and so we may write
$A^{-1}A =  \bigcup_{i \in 1}^{m} A^{-1}U_{i}B^{-1}$.
It follows that $AB =  \bigcup_{i \in I} AA^{-1}U_{i}B^{-1}B \subseteq \bigcup_{i=1}^{m} U_{i}$
and so in fact we have equality.

(2). Let $A,B \in \mathsf{KB}(G)$ be such that $A \sim B$.
Then $A \cup B$ is an open local bisection and the union of two compact subsets is always compact.
\end{proof}

Before we make our next definition, we shall need the following properties.
\begin{description}

\item[{\rm (C1)}] The set $\mathsf{KB}(G)$ of compact-open local bisections forms a basis for the topology on $G$.

\item[{\rm (C2)}]  The set $\mathsf{KB}(G)$ of compact-open local bisections is closed under subset multiplication.

\item[{\rm (C3)}] The \'etale groupoid $G$ is sober.

\end{description}

\begin{lemma}\label{le: compact_open} Let $G$ be an \'etale groupoid.
\begin{enumerate}

\item Then $G$ has a basis of compact-open local bisections if and only if $G_{o}$ has a basis of compact-open sets.

\item $G$ satisfies (C1) and (C2) if and only if $G_{o}$ satisfies (C1) and (C2).

\end{enumerate}
\end{lemma}
\begin{proof} (1) We suppose first that $G_{o}$ has a basis of compact-open sets.
We show that $G$ has a basis of compact-open local bisections.
Let $U$ be any non-empty open bisection in $G$ and let $g \in U$.
Since $G$ is \'etale there is an open local bisection $V$ containing $g$
such that $\mathbf{d}$ restricted to $V$ is a homeomorphism onto its image.
It follows that $\mathbf{d}$ restricted to $U \cap V$ is a homeomorphism onto its image and $g \in U \cap V$.
By assumption we may find a compact-open set (and therefore local bisection) $B$ in $G_{o}$ containing $g^{-1}g$ and contained
in the image of $U \cap V$.
It follows that there is a compact-open local bisection $A$ containing $g$ such that $\mathbf{d}$ maps $A$ to $B$
and which is contained in $U \cap V$.
It follows that every open local bisection in $G$ is a union of compact-open local bisections.
Thus the compact-open local bisections form a basis for the topology.

Suppose now that $G$ has a basis of compact-open local bisections.
We prove that $G_{o}$ has a basis of compact-open sets.
Let $U$ be an open set in $G_{o}$ and let $e \in U$.
There is therefore an open set $V$ in $G$ such that $U = G_{o} \cap V$.
Thus $U$ is also an open set in $G$.
Therefore there exists a compact-open local bisection $W$ such that $e \in W \subseteq U$.
It follows that $W$ is a subset of $G_{o}$.

(2) We suppose first that $G_{o}$ satisfies (C1) and (C2).
This means that we assume that $G_{o}$ has a basis of compact-open sets and the intersection
of any two compact-open sets is again compact-open.

We prove first that if $A$ is a compact-open local bisection then so too is $A^{-1}A$.
We need only prove that it is compact.
Let $A^{-1}A \subseteq \bigcup_{i} O_{i}$ be a covering by open local bisections.
Then $A \subseteq \bigcup_{i} AO_{i}$ is also a covering by open local bisections.
By assumption, $A$ is compact and so we may find a finite number $AO_{1}, \ldots, OA_{m}$ that cover $A$.
Thus $A \subseteq \bigcup_{i=1}^{m} AO_{i}$.
Hence $A^{-1}A \subseteq \bigcup_{i=1}^{m} A^{-1}AO_{i}$.
But $A^{-1}AO_{i} \subseteq O_{i}$ and so $A^{-1}A \subseteq \bigcup_{i=1}^{m} O_{i}$.
Thus $A^{-1}A$ is compact.

Let $A$ and $B$ be two compact-open local bisections.
The product $AB$ is an open local bisection so it only remains to show that it is compact.
Let $AB \subseteq \bigcup_{i} C_{i}$ where the $C_{i}$ are open local bisections.
Then $A^{-1}ABB^{-1} \subseteq \bigcup_{i} A^{-1}C_{i}B^{-1}$.
Now $A^{-1}ABB^{-1} = A^{-1}A \cap BB^{-1}$ and so is compact.
Thus we may write
$A^{-1}ABB^{-1} \subseteq \bigcup_{i=1}^{m} A^{-1}C_{i}B^{-1}$
and so
$AB \subseteq \bigcup_{i=1}^{m} AA^{-1}C_{i}B^{-1}B \subseteq \bigcup_{i=1}^{m} C_{i}$, as required.

The proof of the converse is straightforward.
\end{proof}

An \'etale groupoid is said to be {\em spectral} if $G_{o}$ is a spectral space.
It follows by the above lemma that the compact-open local bisections of a spectral groupoid
form a basis for the topology and that this set is closed under subset multiplication. 
The proof of the following is now immediate by Lemma~\ref{le: reg} and the definitions.

\begin{proposition}\label{prop: distributive_inverse} Let $G$ be a specral groupoid. 
Then  $\mathsf{KB}(G)$ is a distributive inverse semigroup.
\end{proposition}

%%%%%%%%%%%%%%%%%%%%%%%%%%%%%%%%%%%%%%%%%%%%%%%%%%%%%%%%%%%%%%%%%%%%%%%%%%%%%%%%%%%%%%%%%%%%%%%%%%%%%%%%%%%%%%%%%%%%%%%%%%%%%%%%%%%
Let $S$ be a distributive inverse semigroup.
Denote by $\mathsf{G}_{P}(S)$ the set of all prime filters of $S$.
Let $A$ and $B$ be filters such that $\mathbf{d}(A) = \mathbf{r}(B)$.
Define $A \cdot B = (AB)^{\uparrow}$.
Then $A \cdot B$ is a filter by Lemma~3.5 of \cite{LMS}.
If $A$ and $B$ are prime filters then $A \cdot B$ is a prime filter
and
if $A$ is a prime filter then $A^{-1}$ is a prime filter by Lemma~\ref{le: filter}. 
We have proved the following.

\begin{lemma} Let $S$ be a distributive inverse semigroup.
Then $\mathsf{G}_{P}(S)$ is a groupoid whose identities are the idempotent prime filters.
\end{lemma}

For each $a \in S$ define $X_{a}$ to be the set of all prime filters that contain $a$.
Put $\pi = \{X_{a} \colon a \in S \}$.

\begin{lemma}\label{le: prime_topology} Let $S$ be a distributive inverse semigroup.
\begin{enumerate}

%1
\item $X_{a} = \emptyset$ if and only if $a = 0$.

%2
\item $X_{a^{-1}} = X_{a}^{-1}$.

%3
\item $X_{a} = X_{b}$ if and only if $a = b$.

%4
\item $X_{a} \subseteq X_{b}$ if and only if $a \leq b$.

%5
\item $X_{st} = X_{s}X_{t}$.

%6
\item If $a \sim b$ then $X_{a} \cup X_{b} = X_{a \vee b}$.

%7
\item $X_{a} \cap X_{b} = X_{c}$ if and only if $c = a \wedge b$.

%8
\item $X_{a} \cap X_{b} = \bigcup_{c \leq a,b} X_{c}$.

%9
\item $\pi$ is the basis of a topology on $\mathsf{G}_{P}(S)$.

%10
\item The sets $X_{a}$ are compact in the $\pi$ topology.

\end{enumerate}
\end{lemma}
\begin{proof}
(1). If $X_{a} \neq \emptyset$ then clearly $a \neq 0$.
Conversely, let $a \neq 0$.
Then by Lemma~\ref{le: ultrafilters}, $a \in F$ for some ultrafilter $F$.
But all ultrafilters are prime filters by Proposition~\ref{prop: semigroup_filters} 
and so every non-zero element is contained in some prime filter.

(2). Immediate.

(3). Suppose that $X_{a} = X_{b}$. 
We shall prove that $a = b$.
Suppose not.
Then either $a \nleq b$ or $b \nleq a$.
Then by Proposition~\ref{prop: key_result}, 
there is a prime filter that contains $a$ and omits $b$,
or a prime filter that contains $b$ and omits $a$.
In either case, we get a contradiction.
The proof of the converse is immediate.

(4). Suppose that $X_{a} \subseteq X_{b}$.
We shall prove that $a \leq b$.
In fact, we shall prove that $X_{a} = X_{ba^{-1}a}$ and then apply (3) above.
Observe that $X_{a} \subseteq X_{ba^{-1}a}$ because if $F \in X_{a}$ then $F \in X_{b}$ by assumption.
Thus $a,b \in F$.
But $ba^{-1}a \in FF^{-1}F = F$ and so $F \in X_{ba^{-1}a}$.
Now let $F \in X_{ba^{-1}a}$.
Then $ba^{-1}a,b \in F$.
Observe that $b^{-1}ba^{-1}a \in F^{-1}F$ and so $a^{-1}a \in F^{-1}F$.
Put $G = (a \mathbf{d}(F))^{\uparrow}$.
Then $G$ is a prime filter that contains $a$ and so by assumption contains $b$.
But $\mathbf{d}(F) = \mathbf{d}(G)$ and $b \in F \cap G$ and so by Lemma~\ref{le: filter_properties}, we have that $F = G$.
It follows that $a \in F$, as required.

(5). We show first that $X_{s}X_{t} \subseteq X_{st}$.
Let $F \in X_{s}$, $G \in X_{t}$ and suppose that $F \cdot G$ is defined.
Then $F \cdot G$ is a prime filter and $st \in F \cdot G$.
It remains to prove the reverse inclusion.
Let $P \in X_{st}$.
Observe that $\mathbf{d}(st) \leq \mathbf{d}(t)$.
Thus $G = (t \mathbf{d}(P))^{\uparrow}$ is a well-defined prime filter.
Observe that $t\mathbf{d}(st) \in G$ and so $\mathbf{r}(t)\mathbf{d}(s) \in \mathbf{r}(G)$.
It follows that $\mathbf{d}(s) \in \mathbf{r}(G)$.
Thus $F = (d \mathbf{r}(G))^{\uparrow}$ is a well-defined prime filter.
Clearly, $P = F \cdot G$ and $F \in X_{s}$ and $G \in X_{t}$.

(6). Only one direction needs proving. 
Let $F \in X_{a \vee b}$.
Then since $F$ is a prime filter either $a \in F$ or $b \in F$, as required.

(7). Only one direction needs proving.
Suppose that $X_{a} \cap X_{b} = X_{c}$.
Clearly $X_{c} \subseteq X_{a},X_{b}$ and so $c \leq a,b$.
Let $d \leq a,b$.
Then $X_{d} \subseteq X_{a},X_{b}$.
Thus $X_{d} \subseteq X_{c}$.
It follows that $d \leq c$.
Hence $c = a \wedge b$, as required.

(8). Straightforward.

(9). This is now immediate.

(10). Suppose that $X_{a} = \bigcup_{i \in I} X_{a_{i}}$.
Let $A$ be the order ideal generated by the $a_{i}$.
Then $I = A^{\vee}$ is $\vee$-closed order ideal.
Suppose that $a \notin I$.
Then $I \cap a^{\uparrow} = \emptyset$.
Thus by Lemma~\ref{le: maximal}
there exists a maximal $\vee$-ideal $J$ such that $I \subseteq J$ and $J \cap a^{\uparrow} = \emptyset$.
By Lemma~\ref{le: maximal_prime}, $J$ is a prime $\vee$-ideal.
Put $F = S \setminus J$.
Then $F$ is a prime filter by Lemma~\ref{le: ideal_filter} and $a \in F$ and $F \cap I = \emptyset$.
But $a_{i} \in F$ for some $i$, which is a contradiction.
It follows that $a \in I$.
Thus there exist elements $b_{j} \leq a_{j}$ for $1 \leq j \leq m$, and suitable relabelling if necessary
such that $a = \bigvee_{j=1}^{m} b_{j}$.
Now any prime filter containing $a$ must contain one of the $b_{j}$ and so one of the $a_{j}$.
It follows that $X_{a} = \bigcup_{j=1}^{m} X_{a_{j}}$.
\end{proof}

We regard $\mathsf{G}_{P}(S)$ as a topological space relative to the basis $\pi$ of Lemma~\ref{le: prime_topology}.

\begin{proposition}\label{prop: spectral_groupoid} Let $S$ be a distributive inverse semigroup.
Then $\mathsf{G}_{P}(S)$ is a spectral groupoid
whose space of identities is homeomorphic to the space associated with the distributive lattice $E(S)$.
\end{proposition}
\begin{proof}
Denote the set of composable elements in the groupoid $G$ by $G \ast G$ and denote the
multiplication map by $m \colon G \ast G \rightarrow G$.
We observe that
$$m^{-1}(X_{s}) = \left( \bigcup_{0 \neq ab \leq s} X_{a} \times X_{b} \right) \cap (\mathsf{G}_{P}(S) \ast \mathsf{G}_{P}(S))$$
for all $s \in S$.
The proof is straightforward and the same as step~3 of the proof of Proposition~2.22 of \cite{Law2}
and shows that $m$ is a continuous function.

It remains to show that $\mathsf{G}_{P}(S)$ is \'etale.
There are a number of ways to prove this.
We could follow step~4 of the proof of Proposition~2.22 of \cite{Law3}.
We give a different proof here.

We show first that  $\mathsf{G}_{P}(S)_{o}$ is an open subspace of $\mathsf{G}_{P}(S)$.
Let $F$ be an identity in $\mathsf{G}_{P}(S)$.
Then by Lemma~\ref{le: filter_properties}, $F$ is an inverse subsemigroup and so contains idempotents.
Let $e \in F$.
Then $F \in X_{e}$.
But every prime filter in $X_{e}$ contains an idempotent and so is an identity in the groupoid.
Thus $F \in X_{e} \subseteq \mathsf{G}_{P}(S)_{o}$ and is an open set.
Thus $\mathsf{G}_{P}(S)_{o}$ is an open set.

Next we show that the product of two open sets is an open set.
Let $X$ and $Y$ be any open sets.
By the definition of the topology, we may write
$X = \bigcup_{i} X_{s_{i}}$ and $Y = \bigcup_{j} X_{t_{j}}$.
Then we have
$$XY = \bigcup_{i,j} X_{s_{i}t_{j}}$$
by Lemma~\ref{le: prime_topology}.
Thus the product of open sets is always open.

We now prove that the space of identities is homeomorphic to the space associated with the distributive lattice $E(S)$.
We use the fact that there is a bijection between prime filters in $E(S)$ and idempotent prime filters in $S$
which is given by $F \mapsto F^{\uparrow}$ and $A \mapsto E(A)$.
Denote the set of prime filters in $E(S)$ containing the idempotent $e$ by $X_{e}^{E}$.
Observe that the inverse image of $X_{a} \cap \mathsf{G}_{P}(S)_{o}$ is $\bigcup_{e \leq a,a^{-1}a} X_{e}^{E}$.
On the other hand, the image of $X_{e}^{E}$ is $X_{e}$ which is a subset of $\mathsf{G}_{P}(S)_{o}$.
It follows that we have a homeomorphism.
To finish off, we observe that the space associated with a distributive lattice is spectral by Theorem~\ref{the: stone_duality_distributive}.
\end{proof}

%%%%%%%%%%%%%%%%%%%%%%%%%%%%%%%%%%%%%%%%%%%%%%%%%%%%%%%%%%%%%%%%%%%%%%%%%%%%%%%%%%%%%%%%%%%%%%%%%%%%%%%%%%%%%%%%%%%%%%%%%%%%%%%%%%%%%%%%%%%%%%%%%
From each distributive inverse semigroup $S$ we have constructed a spectral groupoid $\mathsf{G}_{P}(S)$,
and from each spectral groupoid $G$ we have constructed a distributive inverse semigroup $\mathsf{KB}(G)$.
We shall now investigate the relationship between these two constructions.

\begin{lemma}\label{le: Xlemma} Let $S$ be a distributive inverse semigroup.
\begin{enumerate}

%1
\item $X_{s}$ is a local bisection of $\mathsf{G}_{P}(S)$. 

%2
\item If every filter in $X_{s}$ is idempotent then $s$ is an idempotent.

%3
\item If $X_{s} \cup X_{t}$ is a local bisection then $s \sim t$.

%4
\item Every compact-open local bisection of $\mathsf{G}_{P}(S)$ is of the form $X_{s}$ for some $s \in S$. 

\end{enumerate}
\end{lemma}
\begin{proof}
(1). This is immediate by part (3) of Lemma~\ref{le: filter_properties}.

(2). If every filter in $X_{s}$ is idempotent then $X_{s} = X_{s}^{-1}$ since idempotent filters are inverse subsemigroups.
It follows that $s = s^{-1}$ by Lemma~\ref{le: prime_topology}
Next observe that $X_{s} \subseteq X_{s^{2}}$.
Thus $s \leq s^{2}$ by Lemma~\ref{le: prime_topology}
It now follows that $s = s^{2}$, as required.

(3). We prove that every filter in $X_{s^{-1}t}$ is an idempotent and so the result follows by (2) above and symmetry.
By assumption, $X_{s}^{-1}X_{t} = X_{s^{-1}t}$ is a subset of the space of identities and so every element in $X_{s^{-1}t}$ 
is an idempotent filter.

(4). Let $A$ be a compact-open local bisection of $\mathsf{G}_{P}(S)$. 
Then $A = \bigcup_{i=1}^{m} X_{s_{i}}$ for a finite set of elements $s_{1}, \ldots, s_{m} \in S$.
The result now follows by (3) above.
\end{proof}

Given a distributive inverse semigroup $S$, we have proved by Proposition~\ref{prop: spectral_groupoid} 
that $\mathsf{G}_{P}(S)$ is a spectral groupoid.
It follows by Proposition~\ref{prop: distributive_inverse} 
that $\mathsf{KB}(\mathsf{G}_{P}(S))$ is a distributive inverse semigroup.
By Lemmas~\ref{le: prime_topology}, \ref{le: Xlemma},
the set $X_{s}$ is a compact-open local bisection and so the map $s \mapsto X_{s}$ is
a well-defined function from  $S$ to $\mathsf{KB}(\mathsf{G}_{P}(S))$.
This function is a homomorphism, injective and surjective.
We have therefore proved the following.

\begin{theorem}\label{the: dist_duality} Let $S$ be a distributive inverse semigroup.
Define $\varepsilon \colon S \rightarrow \mathsf{KB}(\mathsf{G}_{P}(S))$ by $\varepsilon (s) = X_{s}$.
Then $\varepsilon$ is an isomorphism.
\end{theorem}

%%%%%%%%%%%%%%%%%%%%%%%%%%%%%%%%%%%%%%%%%%%%%%%%%%%%%%%%%%%%%%%%%%%%%%%%%%%%%%%%%%%%%%%%%%%%%%%%%%%%%%%%%%%%%%%%%%%%%%%%%%%%%%%%%%%%
The above theorem is one of a pair.
We shall now work towards proving its companion.
Let $G$ be a spectral groupoid.
For each $g \in G$ define $F_{g}$ to be the set of all compact-open local bisections that contain $g$.

\begin{lemma}\label{le: Flemma1}
Let $G$ be a spectral groupoid.
Then for each $g \in G$ the set $F_{g}$ is a prime filter in the distributive inverse semigroup $\mathsf{KB}(G)$.
\end{lemma}
\begin{proof}
The set $F_{g}$ is non-empty because $G$ has a basis of compact-open local bisections by Lemma~\ref{le: compact_open}.
If $U,V \in F_{g}$ then $U \cap V$ is an open local bisection that contains $g$.
It follows that there is a compact-open local bisection that contains $g$ and is contained in $U \cap V$ by Lemma~\ref{le: compact_open}
It follows that $F_{g}$ is down directed.
It is clear that $F_{g}$ is upwardly closed.
The fact that $F_{g}$ is prime is straightforward to prove.
\end{proof}

\begin{lemma}\label{le: Flemma2} Let $G$ be a spectral groupoid.
\begin{enumerate}

%1
\item $\mathbf{d}(F_{g}) = F_{\mathbf{d}(g)}$, and dually.

%2
\item If $gh$ is defined in $G$ then $F_{g} \cdot F_{h} = F_{fh}$.

%3
\item $F_{g} = F_{h}$ if and only if $g = h$.

%4
\item Each prime filter in $\mathsf{KB}(G)$ is of the form $F_{g}$ for some $g \in G$.

\end{enumerate}
\end{lemma}
\begin{proof} 
(1). Let $V \in F_{\mathbf{d}(g)}$, a compact-open local bisection containing $\mathbf{d}(g)$.
Since the multiplication map in a topological groupoid is continuous, the inverse image of $V$ under the multiplication map is open.
It follows that we may find compact-open local bisections $g^{-1} \in A$ and $g \in B$ such that $AB \subseteq V$.
It is now easy to check that $\mathbf{d}(F_{g}) = F_{\mathbf{d}(g)}$. 

(2). By (1) above, if $gh$ is defined in the groupoid $G$ then the product $F_{g} \cdot F_{h}$ is defined.
Clearly, $F_{g} \cdot F_{h} \subseteq F_{gh}$.
The reverse inclusion uses the fact that multiplication is continuous and that the groupoid has
a basis of compact-open local bisections.

(3). Suppose that $F_{g} = F_{h}$.
We show first that $\mathbf{d}(g) = \mathbf{d}(h)$.
Suppose not.
Then in the sober space $G_{o}$, and without loss of generality, 
we may find a compact-open subset $U$ that contains $\mathbf{d}(g)$ and omits $\mathbf{d}(h)$.
The function $\mathbf{d}$ is continuous, and so $\mathbf{d}^{-1}(U)$ is an open set in $G$ that contains $g$.
We may therefore find a compact-open local bisection $V$ such that $g \in V$ and $\mathbf{d}(V) \subseteq U$.
By assumption, $h \in V$ which implies that $\mathbf{d}(h) \in U$, which is a contradiction.

It follows that we may write, $F_{gh^{-1}} = F_{hh^{-1}}$ using (2) above .
Now $hh^{-1}$ is an identity and so is contained in the open space $G_{o}$.
Since $G_{o}$ is spectral, $G_{o}$ has a basis of compact-open subsets.
It follows that there is a compact-open subset of $G_{o}$ that contains $hh^{-1}$.
It follows that $gh^{-1}$ must also be an identity.
Since $gh^{-1},hh^{-1} \in G_{o}$ we may now use the sobriety of $G_{o}$ to deduce that $gh^{-1} = hh^{-1}$.
It follows that $g = h$, as required.

(4). Let $F$ be a prime filter of compact-open local bisections of $G$.
Then $\mathbf{d}(F) \cap G_{o}$ is a prime filter of compact-open subsets of $G_{o}$.
Since $G_{o}$ is a spectral space, there is a point $e \in G_{o}$ such that $\mathbf{d}(F) \cap G_{o}$ 
is precisely the set of all compact-open subsets of $G_{o}$ that contain $e$.
Let $A \in F$.
Then $A^{-1}A$ contains $e$.
It follows that there is an element $g \in A$, necessarily unique, such that $\mathbf{d}(g) = e$.
Similarly, if $B \in F$ there exists a unique element $h \in B$ such that $\mathbf{d}(h) = e$.
By assumption, there exists $C \in F$ such that $C \subseteq A,B$.
Again we may find a unique element $k \in C$ such that $\mathbf{d}(k) = e$.
But by uniqueness, we get that $g = h = k$.
It follows that $F \subseteq F_{g}$.
However, $\mathbf{d}(F) = \mathbf{d}(F_{g})$ and so $F = F_{g}$.
\end{proof}

Let $G$ be a spectral groupoid.
Then $\mathsf{KB}(G)$ is a distributive lattice by Proposition~\ref{prop: distributive_inverse}
and $\mathsf{G}_{P}(\mathsf{KB}(G))$ is a spectral groupoid by Proposition~\ref{prop: spectral_groupoid}.
The map $\eta \colon G \rightarrow \mathsf{G}_{P}(\mathsf{KB}(G))$ given by $g \mapsto F_{g}$
is well-defined by Lemma~\ref{le: Flemma1} and by Lemma~\ref{le: Flemma2} it is an isomorphism
of groupoids.

\begin{theorem}\label{the: spectral_duality} Let $G$ be a spectral groupoid.
Then $\eta \colon G \rightarrow \mathsf{G}_{P}(\mathsf{KB}(G))$ is an isomorphism of topological groupoids.
\end{theorem}
\begin{proof} It remains to prove that $\eta$ is continuous and open.
Let $A$ be a compact-open local bisection of $G$.
Then $\eta (g) = F_{g} \in X_{A} \Leftrightarrow g \in A$.
It follows that $A = \eta^{-1}(X_{A})$ and so $\eta$ is continuous
and that $\eta (A) = X_{A}$ and so $\eta$ is open.
\end{proof}

%%%%%%%%%%%%%%%%%%%%%%%%%%%%%%%%%%%%%%%%%%%%%%%%%%%%%%%%%%%%%%%%%%%%%%%%%%%%%%%%%%%%%%%%%%%%%%%%%%%%%%%%%%%%%%%%%%%%%%%%%%%%%%%%%%%%
It is now time to describe the categories that will form the setting for our duality theorem.
We shall not attempt to define the most general possible morphisms, merely those sufficient 
for our present goals.\footnote{The discussion that follows owes a big debt to Ganna Kudryavtseva who suggested some of the key ideas.}
Our first result is well-known and is included for the sake of completeness. 

\begin{lemma} Let $\theta \colon S \rightarrow T$ be a morphism of distributive lattices.
Then for each prime filter $P$ in $T$ the inverse image $\theta^{-1}(P)$ is non-empty if and only if for
each $t \in T$ there exists $s \in S$ such that $t \leq \theta (s)$.
\end{lemma}
\begin{proof} Only one direction needs proving.
We assume that for each prime filter $P$ in $T$ the inverse image $\theta^{-1}(P)$ is non-empty. 
Let $t \in T$ and suppose that there is no $s \in S$ such that $t \leq \theta (s)$.
This is equivalent to saying that $t \notin I = \mbox{im}(\theta)^{\downarrow}$.
Observe that $I$ is an order ideal that is $\vee$-closed; the latter claim following from the fact
that $\theta$ is a morphism of distributive lattices.
We therefore have that $t^{\uparrow} \cap I = \emptyset$.
By Lemma~\ref{le: maximal} there is a $\vee$-closed order ideal $J$ such that $I \subseteq J$ and $t^{\uparrow} \cap J = \emptyset$
and that is maximal with respect to these properties.
By Lemma~\ref{le: maximal_prime}, $J$ is prime and so $F = T \setminus P$ is a prime filter by Lemma~\ref{le: ideal_filter}.
Thus $t \in F$ and $F \cap \mbox{im}(\theta)^{\downarrow} = \emptyset$.
But this contradicts our assumption that $\theta^{-1}(F) \neq \emptyset$.
\end{proof}

Morphisms of distributive lattices that satisfy the condition of the above lemma are called {\em proper}.
The following is immediate when we recall that morphisms of distributive lattices must preserve binary meets.

\begin{corollary}  Let $\theta \colon S \rightarrow T$ be a proper morphism of distributive lattices.
Then $\theta^{-1}(F)$ is a prime filter in $S$ for each prime filter $F$ in $T$.
\end{corollary}

We now turn to morphisms between distributive inverse semigroups.

\begin{lemma} Let $\theta \colon S \rightarrow T$ be a morphism of distributive inverse semigroups.
Let $A$ be a prime filter in $T$.
Put $F = E(\mathbf{d}(A))$, a prime filter in $E(T)$.
If $\theta^{-1}(A)$ is non-empty, then it is a disjoint union of prime filters $B$ of $S$
such that $E(\mathbf{d}(B)) = E(\theta^{-1}(F))$.
\end{lemma}  
\begin{proof} Put $H = (E(\theta^{-1}(F))^{\uparrow}$, an idempotent prime filter in $S$.
Let $h \in H$.
Then $e \leq h$ where $\theta (e) = F$.
It follows that $\theta (h) \in \mathbf{d}(A)$.
Now let $b \in \theta^{-1}(A)$ be arbitrary.
Then $\theta (b^{-1}b) \in \mathbf{d}(A)$.
Thus $\theta (b^{-1}b) \in F$ and so $b^{-1}b \in H$.
It follows that $B = (bH)^{\uparrow}$ is a well-defined prime filter in $S$.
Let $x \in B$.
Then $bh \leq x$ for some $h \in H$.
Then $\theta (bh) \leq \theta (x)$.
But $\theta (bh) = \theta (b) \theta (h) \in A\mathbf{d}(A) = A$.
\end{proof}

We shall restrict our attention to those morphisms of distributive inverse semigroups
with the additional property that the inverse images of prime filters are prime filters.
This is a global property that refers explicitly to prime filters.
We shall reformulate this condition without reference to them.

\begin{lemma} Let $\theta \colon S \rightarrow T$ be a morphism of distributive inverse semigroups.
Then for each prime filter $P$ in $T$ the inverse image $\theta^{-1}(P)$ is non-empty if and only if for
each $t \in T$ we may find a finite non-empty compatible set $t_{1}, \ldots, t_{m}$ such that
$t = \bigvee_{i=1}^{m} t_{i}$ and such that for each $i$ there exists $s_{i} \in S$ such that
$t_{i} \leq \theta (s_{i})$. 
\end{lemma}
\begin{proof} Suppose first that the condition holds.
Let $A$ be a prime filter in $T$.
Choose $t \in A$.
By assumption, 
$t = \bigvee_{i=1}^{m} t_{i}$ and such that for each $i$ there exists $s_{i} \in S$ such that
$t_{i} \leq \theta (s_{i})$. 
But $A$ is a prime filter and so $t_{i} \in A$ for some $i$.
It follows that $\theta (s_{i}) \in A$, and so $\theta^{-1}(A) \neq \emptyset$, as required.

To prove the converse, assume that for each prime filter $P$ in $T$ the inverse image $\theta^{-1}(P)$ is non-empty. 
Let $t \in T$.
Suppose that $t \notin (\mbox{im}(\theta)^{\downarrow})^{\vee} = I$.
Observe that $I$ is a $\vee$-closed order ideal.
By assumption, $t^{\uparrow} \cap I = \emptyset$.
Thus by Lemma~\ref{le: maximal} there is a $\vee$-closed order ideal $J$ such that $I \subseteq J$ and $t^{\uparrow} \cap J = \emptyset$
and that is maximal with respect to these properties.
By Lemma~\ref{le: maximal_prime}, $J$ is prime and so $A = T \setminus J$ is a prime filter by Lemma~\ref{le: ideal_filter}.
Thus $t \in A$ and $A \cap \mbox{im}(\theta)^{\downarrow} = \emptyset$.
In particular, $A$ is disjoint from the image of $\theta$ which contradicts our assumption.
Hence $t \in (\mbox{im}(\theta)^{\downarrow})^{\vee}$ which translates into our condition. 
\end{proof}

A morphism of distributive inverse semigroups that satisfies the condition of the above lemma is called {\em proper}.
Given a proper morphism we now need a condition that will force the inverse image of a prime filter to
be a single prime filter. 
We say that a morphism $\theta \colon S \rightarrow T$ of distributive inverse semigroups
is {\em weakly meet preserving} if given $t \leq \theta (a), \theta (b)$ there exists
$c \leq a,b$ such that $t \leq \theta (c)$.

\begin{lemma} Let $\theta \colon S \rightarrow T$ be a morphism of distributive inverse semigroups
that is weakly meet preserving.
If $A$ is a prime filter in $T$ and $B = \theta^{-1}(A)$ is non-empty, then $B$ is a prime filter.
\end{lemma}
\begin{proof} Let $a,b \in B$.
Then $\theta (a), \theta (b) \in A$.
Since $A$ is a filter there exists $t \in A$ such that $t \leq a,b$.
By assumption, we may find an element $c \leq a,b$ such that $t \leq \theta (c)$.
Since $A$ is a filter $\theta (c) \in A$ and so $c \in B$.
Thus $B$ is down directed.
It is clearly closed upwards.
Finally, suppose that $a \in b$ and $a \vee b \in B$.
Then $\theta (a) \vee \theta (b) \in A$ and so either $\theta (a) \in A$ or $\theta (b) \in A$.
It follows that either $a \in B$ or $b \in B$.   
\end{proof}

A morphism of distributive inverse semigroups 
is called {\em callitic}\footnote{This is just a nonce word derived from the Greek word for `good'.} if it is proper and weakly meet preserving.
We define the category {\bf Dist} to have as objects the distributive inverse semigroups
and as morphisms the callitic morphisms.

A functor $\theta \colon G \rightarrow H$ between groupoids is called a {\em covering functor}
if it is {\em star injective}, meaning that
if $\mathbf{d}(g) = \mathbf{d}(h)$ and $\theta (g) = \theta (h)$ implies that $g = h$,
and 
{\em star surjective}, meaning that if $e \in G_{o}$ is such that $\theta (e) = \mathbf{d}(h)$
then there exists $g \in G$ such that $\theta (g) = h$ and $\mathbf{d}(g) = e$. 
%The following is a well-known property of covering functors.

%\begin{lemma}\label{le: covering} Let $\theta \colon G \rightarrow H$ be a covering functor.
%Then if $\theta (x) = ab$ there exist $u,v \in G$ such that $x = uv$, $\theta (u) = a$ and $\theta (v) = b$.
%\end{lemma}

A continuous map between topological spaces is called {\em coherent} if the inverse images of compact-open sets
are compact-open sets. 
We define the category {\bf Spec} to have as objects the spectral groupoids and as morphisms the coherent continuous covering functors.
 
\begin{lemma}\label{le: callitic} Let $\theta \colon G \rightarrow H$ be a coherent continuous covering functor
between spectral groupoids.
Then $\theta^{-1}$ induces a callitic map from $\mathsf{KB}(H)$ to $\mathsf{KB}(G)$.
\end{lemma}
\begin{proof}
Let $A$ be a compact-open local bisection of $H$.
Since $\theta$ is continuous and coherent we have that $\theta^{-1}(A)$ is compact-open.
Since $\theta$ is a covering functor we have that $\theta^{-1}(A)$ is a local bisecton by Proposition~2.17 of \cite{Law2}.
It also follows by this proposition that $\theta^{-1}$ induces a morphism of distributive inverse semigroups.
It remains to prove that $\theta^{-1}$ is proper and weakly meet preserving.

We prove first that $\theta^{-1}$ is proper. 
Let $B$ be a non-empty compact-open local bisection in $G$ and let $g \in B$.
Then $\theta (g) \in H$.
Clearly $H$ is an open set containing $\theta (g)$.
Since $H$ is \'etale, it follows that $H$ is a union of compact-open local bisections 
and so $\theta (g) \in C_{g}$ for some an compact-open local bisection $C_{g}$ in $H$.
Since $\theta$ is continuous and coherent $g \in \theta^{-1}(C_{g})$ is compact-open 
and because $\theta$ is a covering functor $\theta^{-1} (C_{g})$ is a local bisection.
It follows that $B \subseteq \bigcup_{g \in B} \theta^{-1}(C_{g})$.
Since $B$ is compact, we may in fact write
$B \subseteq \bigcup_{i=1}^{m} \theta^{-1}(C_{g_{i}})$
for some finite set of elements $g_{1}, \ldots, g_{m} \in B$.
Put $B_{i} = B \cap \theta^{-1}(C_{g_{i}})$.
This is clearly an open local bisection and $B = \bigcup_{i=1}^{m} B_{i}$ 
and each $B_{i} \subseteq \theta^{-1}(C_{i})$ where $C_{i} = C_{g_{i}}$.
We prove that we may find compact-open local bisections $D_{i}$ such that
$B$ is the union of the $D_{i}$ and $D_{i} \subseteq \theta^{-1}(C_{i})$.
Since $B_{i}$ is an open local bisection it is a union of compact-open local bisections.
Amalgamating these unions we have that $B$ is a union of compact-open local bisections
each of which is a subset of one of the $\theta^{-1}(C_{i})$.
It follows that $B$ is a union of a finite number of such compact-open local bisections.
Define $D_{i}$ to be the union of those which are contained in $\theta^{-1}(C_{i})$ and the result follows.
 
We now prove that $\theta^{-1}$ is weakly meet preserving.
Let $A$ and $B$ be compact-open local bisections of $H$.
Let $Y$ be a compact-open local bisection of $G$ such that $Y \subseteq \theta^{-1}(A), \theta^{-1}(B)$.
Clearly $Y \subseteq \theta^{-1}(A \cap B)$.
We can at least say that $A \cap B$ is an open local bisection
and 
$\theta (Y) \subseteq A \cap B$.
Since $\theta$ is continuous, we know that $\theta (Y)$ is compact.
It is also immediate that $\theta (Y)$ is a local bisection.
Now $H$ is a spectral groupoid and so has a basis of compact-open local bisections.
It follows that $A \cap B$ is a union of compact-open local bisections.
But $\theta (Y)$ is compact and so $\theta (Y)$ is contained in a finite union of compact-open local
bisections that is also contained in $A \cap B$.
Thus $\theta (Y) \subseteq V = \bigcup_{i=1}^{m} V_{i} \subseteq A \cap B$.
Now $A \cap B$ a local bisection implies that $V$ is a local bisection.
It is evident that $V$ is a compact-open local bisection itself.
We therefore have $\theta (Y) \subseteq V \subseteq A \cap B$.
Hence $Y \subseteq \theta^{-1}(\theta (Y)) \subseteq \theta^{-1}(V) \subseteq \theta^{-1}(A \cap B)$. 
\end{proof}

\begin{lemma}\label{le: covering} 
Let $\theta \colon S \rightarrow T$ be a callitic morphism between distributive inverse semigroup.
Then $\theta^{-1}$ induces a coherent continuous covering functor from $\mathsf{G}_{P}(T)$ to
$\mathsf{G}_{P}(S)$.
\end{lemma}
\begin{proof} The function $\theta^{-1}$ is well-defined by definition.
The bulk of the proof is taken up with showing that $\theta^{-1}$ is a functor.

Let $F$ be an idempotent prime filter in $T$.
Then $\theta^{-1}(F)$ is certainly a prime filter in $S$.
To show that it is an idempotent filter we use Lemma~\ref{le: filter_properties}.
We have that $\theta^{-1}(E(F)) \subseteq \theta^{-1}(F)$ and
$E(\theta^{-1}(E(T)))$ is a prime filter in $E(S)$.
Thus $\theta^{-1}(F)$ contains idempotents and so is itself idempotent.

We have shown that $\theta^{-1}$ maps identities to identities.

We next prove that if $F$ and $G$ are prime filters such that $F^{-1} \cdot F = G \cdot G^{-1}$ then
$$( \theta^{-1} (F) \theta^{-1} (G) )^{\uparrow} = \theta^{-1} ( (FG)^{\uparrow}  ).$$
We prove first that
$$\theta^{-1} (F) \theta^{-1} (G) \subseteq \theta^{-1} (FG).$$
Let $s \in \theta^{-1} (F) \theta^{-1} (G)$.
Then $s = ab$ where $a \in \theta^{-1} (F)$ and $b \in \theta^{-1} (G)$.
Thus $\theta (s) = \theta (a) \theta (b) \in FG$.
It follows that $s \in \theta^{-1} (FG)$.
Observe that $\theta^{-1} (X)^{\uparrow} \subseteq \theta^{-1} (X^{\uparrow})$.
It follows that
$$( \theta^{-1} (F) \theta^{-1} (G) )^{\uparrow} \subseteq \theta^{-1} ( (FG)^{\uparrow}  ).$$

We now prove the reverse inclusion.
Let $s \in \theta^{-1} ( (FG)^{\uparrow} ).$
Then $\theta (s) \in F \cdot G$ and so
$fg \leq \theta(s)$ for some $f \in F$ and $g \in G$.
The map $\theta$ is assumed proper and so we may quickly deduce that there exists $v \in S$ such that $\theta (v) \in G$.
Consider the product $\theta (s) \theta (v)^{-1}$.
Since $\theta (s) \in F \cdot G$ and $\theta (v)^{-1} \in G^{-1}$ we have that
$\theta (s)\theta (v)^{-1} \in F \cdot G \cdot G^{-1} = F \cdot F^{-1} \cdot F = F$.
Thus $\theta (sv^{-1}) \in F$,  and we were given $\theta (v) \in G$, and clearly $(sv^{-1})v \leq s$.
Put $a = sv^{-1}$ and $b = v$.
Then $ab \leq s$ where $\theta (a) \in F$ and $\theta (b) \in G$.
It follows that $s \in (\theta^{-1} (F) \theta^{-1} (G) )^{\uparrow}$.

We may now show that $\theta^{-1}$ is a functor.
Let $F$ be a prime filter.
Observe that $\theta^{-1}(F)^{-1} = \theta^{-1} (F^{-1})$.
We have that
$$(\theta^{-1} (F^{-1}) \theta^{-1} (F))^{\uparrow}
=
(\theta^{-1} (F)^{-1} \theta^{-1} (F))^{\uparrow}
=
\mathbf{d} ( \theta^{-1} (F))$$
and
$$\theta^{-1} ( (F^{-1}F)^{\uparrow} ) = \theta^{-1} (\mathbf{d} (F)).$$
Hence by our result above
$$\theta^{-1} (\mathbf{d} (F))
=
\mathbf{d} ( \theta^{-1} (F)).$$
A dual result also holds and so $\theta^{-1}$ preserves the domain and codomain operations.
Suppose that $\mathbf{d}(F) = \mathbf{r} (G)$ so that $F \cdot G$ is defined.
By our calculation above $\mathbf{d} (\theta^{-1}(F)) = \mathbf{r} (\theta^{-1} (G))$
and so the product $\theta^{-1} (F) \cdot \theta^{-1} (G)$ is defined.
By our main result above we have that
$$\theta^{-1} (F \cdot G) = \theta^{-1} (F) \cdot \theta^{-1} (G),$$
as required.

We have therefore shown that $\theta^{-1}$ is a functor.

The proof that $\theta^{-1}$ is a covering functor follows the same lines as the proof of Proposition~2.15 of \cite{Law3}:
the proof of star injectivity uses part (3) of Lemma~\ref{le: filter_properties},
and the proof of star surjectivity uses this same lemma and part (2) of Lemma~\ref{le: filter}.

To show that $\theta^{-1}$ is continuous, observe that 
a basic open set of $\mathsf{G}_{P}(S)$ has the form $X_{s}$ for some $s \in S$.
It is simple to check that this is pulled back to the set $X_{\theta (s)}$.
This also shows that the inverse image of every compact-open local bisection is a compact-open
local bisection.

We now prove coherence.
To do this we prove the following result.
Let $\phi \colon G \rightarrow H$ be a continuous covering functor between spectral groupoids
with the property that the inverse image of every compact-open local bisection is a compact-open local bisection.
Then the inverse image of every compact-open set is a compact-open set.
Let $X$ be a compact-open subset of $H$.
Since the groupoid $H$ is spectral, the compact-open local bisections form a basis.
Thus we may write $X$ as a union of compact-open local bisections
and so by compactness, we may write it as a finite union of compact-open local bisections.
It follows that the inverse image of $X$ under $\phi$ can be written as finite union of compact-open local bisections.
Thus since $\phi^{-1}(X)$ is a finite union of compact sets it is compact.
\end{proof}

Combining Lemmas~\ref{le: callitic}, \ref{le: covering} and Theorems~\ref{the: dist_duality}, \ref{the: spectral_duality},
we have proved the following.

\begin{theorem}[Stone duality for distributive inverse semigroups]\label{the: duality_theorem} 
The category {\bf Dist} of distributive inverse semigroups is dually equivalent to the category {\bf Spec}
of spectral groupoids.
\end{theorem}

%%%%%%%%%%%%%%%%%%%%%%%%%%%%%%%%%%%%%%%%%%%%%%%%%%%%%%%%%%%%%%%%%%%%%%%%%%%%%%%%%%%%%%%%%%%%%%%%%%%%%%%%%%%%%%%%%%%%%%%%%%%%%%%%%
We now derive some special cases of the above theorem.

The following are the analogues of results that are well-known for distributive lattices.

\begin{lemma}\label{le: hausdorff} Let $S$ be a distributive inverse semigroup.
\begin{enumerate}

\item The topology on the space of identities of $\mathsf{G}_{P}(S)$ is Hausdorff if and only if $S$ is a Boolean inverse semigroup.

\item The topology on $\mathsf{G}_{P}(S)$ is Hausdorff if and only if $S$ is a Boolean inverse $\wedge$-semigroup.

\end{enumerate}
\end{lemma}
\begin{proof} (1). This follows by Theorem~\ref{the: stone_duality_distributive} and the proof of Proposition~\ref{prop: spectral_groupoid}.

%Suppose that $S$ is a Boolean inverse semigroup.
%We have already seen that in such a semigroup, all prime filters are ultrafilters.
%Let $F$ and $G$ be distinct idempotent prime filters.
%Then $F = E(F)^{\uparrow}$ and $G = E(G)^{\uparrow}$.
%It follows that $E(F)$ and $E(G)$ are distinct ultrafilters in $E(S)$.
%Let $e \in E(F) \setminus E(G)$.
%By Lemma~\ref{le: ultrafilter_result}, there exists $f \in E(G)$ such that $e \wedge f = 0$.
%It follows that $X_{e} \cap X_{f} = \emptyset$, $F \in X_{e}$, $G \in X_{f}$ and all the ultrafilters
%in $X_{e}$ and $X_{f}$ are idempotent.
%We have proved that the space of identities of $\mathsf{G}_{P}(S)$ is Hausdorff.
%To prove the converse, we use the fact that a Hausdorff spectral space is Boolean
%which means that $E(S)$ is a Boolean algebra.

(2). Suppose that $S$ is a Boolean inverse $\wedge$-semigroup.
We have already proved that in such a semigroup, all prime filters are ultrafilters.
Let $F$ and $G$ be two distinct ultrafilters.
Let $s \in F \setminus G$.
By Lemma~\ref{le: ultrafilter_result}, there exists $g \in G$ such that $s \wedge g = 0$.
Then $F \in X_{s}$, $G \in X_{g}$ and $s \wedge g = 0$.
It follows that $X_{s} \cap X_{g} = \emptyset$ and we have proved that $\mathsf{G}_{P}(S)$ is Hausdorff.

To prove the converse, observe that if $\mathsf{G}_{P}(S)$ is Hausdorff then its space of identities is Hausdorff
and so by (1) above, we know that $S$ is a Boolean inverse semigroup.
It remains to prove that $S$ is an inverse $\wedge$-semigroup.
Let $s,t \in S$.
Suppose that $s^{\downarrow} \cap t^{\downarrow} \neq 0$.
We shall prove that $s \wedge t$ exists and is non-zero.
In  Hausdorff space, compact sets are closed.
Thus both $X_{s}$ and $X_{t}$ are clopen and so $X_{s} \cap X_{t}$ is clopen.
But $X_{s} \cap X_{t}$ is a closed subset of $X_{s}$ which is compact and Hausdorff.
It follows that $X_{s} \cap X_{t}$ is compact.
Thus we may write $X_{s} \cap X_{t} = \bigcup_{i=1}^{m} X_{a_{i}}$.
Observe that $X_{a_{i}} \subseteq X_{s}$ for all $i$.
It follows that $a_{i} \leq s$ for all $i$.
Thus the $a_{i}$ are pairwise compatible.
Put $a = \bigvee_{i=1}^{m} a_{i}$.
Then $X_{s} \cap X_{t} = X_{a}$.
The result now follows by Lemma~\ref{le: prime_topology}.
\end{proof}

The following wee lemma will prepare the ground for the final two theorems of this section.

\begin{lemma} Let $\theta \colon S \rightarrow T$ be a weakly meet preserving homomorphism between
Boolean inverese $\wedge$-semigroups.
Then $\theta$ preserves the meet operation.
\end{lemma}
\begin{proof} Since $s \wedge t \leq s,t$ we have that $\theta (s \wedge t) \leq \theta (s) \wedge \theta (t)$.
Now let $u \leq \theta (s), \theta (t)$.
Then by assumption, there exists $v \leq s,t$ such that $u \leq \theta (v)$.
But $v \leq s,t$ implies that $v \leq s \wedge t$ and so $\theta (v) \leq \theta (s \wedge t)$.
Hence $u \leq \theta (s \wedge t)$. 
We have therefore proved that $\theta (s \wedge t) = \theta (s) \wedge \theta (t)$.
\end{proof}

It follows by the above result that the callitic maps between Boolean inverse $\wedge$-semigroups
are precisely the proper $\wedge$-morphisms.
We define a {\em Boolean groupoid} to be a spectral groupoid whose space of identites is Hausdorff.
The following two theorems now follow by Theorem~\ref{the: duality_theorem} and Lemma~\ref{le: hausdorff}.

\begin{theorem}[Stone duality for Boolean inverse semigroups]\label{the: Big_Theorem_Boolean}
The category of Boolean inverse semigroups is dually equivalent to the category of Boolean groupoids.
\end{theorem}

The following theorem was first proved in the monoid case in \cite{Law2} and in the general case in \cite{Law4}.

\begin{theorem}[Stone duality for Boolean inverse $\wedge$-semigroups]\label{the: Big_Theorem_Meet_Boolean}
The category of Boolean inverse $\wedge$-semigroups is dually equivalent to the category of Hausdorff Boolean groupoids.
\end{theorem}

%%%%%%%%%%%%%%%%%%%%%%%%%%%%%%%%%%%%%%%%%%%%%%%%%%%%%%%%%%%%%%%%%%%%%%%%%%%%%%%%%%%%%%%%%%%%%%%%%%%%%%%%%%%%%%%%%%%%%%%%%%%%%%%%
%%NEW MATERIAL STARTS HERE

\section{Adjunctions and duality}

The goal of this section is to set the results of Section~3 in a wider perspective.
We shall be interested in two categories.
The category $\mathbf{PG}$ of pseudogroups and the category $\mathbf{Etale}$ of \'etale groupoids.
The morphisms in the former category are a class of homomorphisms we call {\em callitic}
whereas those in the latter are continuous covering functors.
We shall define functors
$\mathsf{G} \colon \mathbf{PG}^{op} \rightarrow \mathbf{Etale}$
and
$\mathsf{B} \colon \mathbf{Etale} \rightarrow \mathbf{PG}^{op}$
and prove that $\mathsf{G}$ is right adjoint to $\mathsf{B}$.
We shall then show how Stone duality for distributive inverse semigroups can be derived from
this adjunction by extending the concept of `coherence' from frame theory to our more general set-up.

%%%%%%%%%%%%%%%%%%%%%%%%%%%%%%%%%%%%%%%%%%%%%%%%%%%%%%%%%%%%%%%%%%%%%%%%%%%%%%%%%%%%%%%%%%%%%%%%%%%%%%%%%%%%%%%%%%%%%%%%%%%
\subsection{An adjunction theorem}

The object part of the functor that takes \'etale groupoids to pseudogroups is given by $G \mapsto \mathsf{B}(G)$
by Proposition~\ref{prop: B}.
The description of the object part of our second functor depends on a class of filters.
A filter $F$ in a pseudogroup $S$ is said to be {\em completely prime} if $\bigvee a_{i} \in F$ implies that $a_{i} \in F$ for some $i$.
Such filters were defined in \cite{R1} where they were called {\em compatibly prime} and are a generalization of 
a concept important in frame theory \cite{J}.
Given a pseudogroup $S$, we denote the set of all completely prime filters on $S$ by $\mathsf{G}_{CP}(S)$.
The following is the analogue of Lemma~\ref{le: filter}.

\begin{lemma} Let $A$ be a filter in a pseudogroup $S$.
\begin{enumerate}

\item $A$ is completely prime if and only if $A^{-1}$ is completely prime.

\item $A$ is completely prime if and only if $\mathbf{d}(A)$ is completely prime.

\end{enumerate}
\end{lemma}
\begin{proof} (1) This is straightforward.

(2) Suppose that $A$ is completely prime.
We prove that $\mathbf{d}(A) = (A^{-1}A)^{\uparrow}$ is completely prime.
Let $x = \bigvee_{i} x_{i} \in A^{-1} \cdot A$.
Then $a^{-1}a \leq x$ for some $a \in A$;
this is always possible since if $a,b \in A$ then $(a \wedge b)^{-1}(a \wedge b) \leq a^{-1}b$.
Clearly $a^{-1}a = xa^{-1}a$ and so by infinite distributivity we have that $a^{-1}a = \bigvee_{i} x_{i}a^{-1}a$.
Thus again by infinite distributivity,
$a = \bigvee_{i} ax_{i}a^{-1}a$.
By assumption, $A$ is completely prime and so $ax_{i}a^{-1}a \in A$ for some $i$.
Thus $ax_{i} \in A$ since $A$ is upwardly closed.
However $a^{-1}ax_{i} \in A^{-1} \cdot A$.
Thus $x_{i} \in A^{-1} \cdot A$, again by upward closure, as required.

Suppose now that $\mathbf{d}(A)$ is completely prime.
We prove that $A$ is completely prime.
Let $a = \bigvee a_{i} \in A$.
Then $\mathbf{d}(a) = \bigvee \mathbf{d}(a_{i}) \in \mathbf{d}(A)$.
Where we use standard properties of compatible joins \cite{Law2}
By assumption $\mathbf{d}(a_{i}) \in \mathbf{d}(A)$.
It follows that $a_{i} = a\mathbf{d}(a_{i}) \in A$, as required.
\end{proof}

Let $A$ and $B$ be completely prime filters such that $\mathbf{d}(A) = \mathbf{r}(B)$.
Then $(AB)^{\uparrow}$ is a filter such that $\mathbf{d}((AB)^{\uparrow}) = \mathbf{d}(B)$ and
$\mathbf{r}((AB)^{\uparrow}) = \mathbf{r}(A)$ and is a completely prime filter 
by the above lemma and Lemma~\ref{le: filter_properties}.
On $\mathsf{G}_{CP}(S)$, define a partial binary operation by
$$A \cdot B = (AB)^{\uparrow} \text{ iff } \mathbf{d}(A) = \mathbf{r}(B).$$
The proof of the following is now immediate.

\begin{lemma}
For each pseudogroup $S$, the structure $(\mathsf{G}_{CP}(S),\cdot)$ is a groupoid.
\end{lemma}

For each $s \in S$ define $X_{s}$ to be the set of all completely prime filters that contains $s$.
Clearly $X_{0} = \emptyset$ although other sets $X_{s}$ could also be empty.

\begin{lemma}\label{le: top_completely_prime} Let $S$ be a pseudogroup.
\begin{enumerate}

\item $X_{s}$ is a local bisection.

\item $X_{s}^{-1} = X_{s^{-1}}$.

\item $X_{s}X_{t} = X_{st}$.

\item $X_{s} \cap X_{t} = X_{s \wedge t}$.

\item If $s = \bigvee_{i} s_{i}$ then $\bigcup_{i} X_{s_{i}} = X_{s}$.

\end{enumerate}
\end{lemma}
\begin{proof} (1) Let $F,G \in X_{s}$ such that $\mathbf{d}(F) = \mathbf{d}(G)$.
By Lemma~2.11 of \cite{Law3}, this implies that $F = G$.
The dual result can be proved similarly.

(2) Immediate from the properties of the natural partial order.

(3) It is clear that $X_{s}X_{t} \subseteq X_{st}$.
Let $F \in X_{st}$.
Put $H = F^{-1} \cdot F$.
Then $F = (stH)^{\uparrow}$.
Put $A = (s(tHt^{-1})^{\uparrow} )^{\uparrow}$ and $B = (tH)^{\uparrow}$.
Then by Lemma~2.2(2), we have that $A \in X_{s}$ and $B \in X_{t}$, and $A \cdot B = F$.

(4) This is straightforward since filters are closed under binary meets.

(5) This is immediate from the definition of completely prime filters.
\end{proof}

Put $\tau = \{X_{s} \colon s \in S \}$.
By the lemma above, $\tau$ is a basis for a topology on $\mathsf{G}_{CP}(S)$
and in what follows we shall always regard $\mathsf{G}_{CP}(S)$ equipped with this topology.

\begin{lemma}
$\mathsf{G}_{CP}(S)$ is a topological groupoid.
\end{lemma}
\begin{proof}
By the above lemma the inversion map is continuous.
Denote the set of composable elements in the groupoid $G$ by $G \ast G$ and denote the
multiplication map by $m \colon G \ast G \rightarrow G$.
We observe that
$$m^{-1}(X_{s}) = \left( \bigcup_{0 \neq ab \leq s} X_{a} \times X_{b} \right) \cap (\mathsf{G}_{CP}(S) \ast \mathsf{G}_{CP}(S))$$
for all $s \in S$.
The proof is straightforward and the same as step~3 of the proof of Proposition~2.22 of \cite{Law3}
and shows that $m$ is a continuous function.
\end{proof}

We can now state our second main result.

\begin{proposition} Let $S$ be a pseudogroup.
Then $\mathsf{G}_{CP}(S)$, the set of all proper completely prime filters, is an \'etale groupoid.
\end{proposition}
\begin{proof} It remains to show that $\mathsf{G}_{CP}(S)$ is \'etale.
There are a number of ways to prove this.
We could follow step~4 of the proof of Proposition~2.22 of \cite{Law3}.
We give a different proof here.

We show first that  $\mathsf{G}_{CP}(S)_{o}$ is an open subspace of $\mathsf{G}_{CP}(S)$.
Let $F$ be an identity in $\mathsf{G}_{CP}(S)$.
Then by Lemma~\ref{le: filter_properties}, $F$ is an inverse subsemigroup and so contains idempotents.
Let $e \in F$.
Then $F \in X_{e}$.
But every completely prime filter in $X_{e}$ contains an idempotent and so is an identity in the groupoid.
Thus $F \in X_{e} \subseteq \mathsf{G}_{CP}(S)_{o}$ and is an open set.
Thus $\mathsf{G}_{CP}(S)_{o}$ is an open set.

Next we show that the product of two open sets is an open set.
Let $X$ and $Y$ be any open sets.
By the definition of the topology, we may write
$X = \bigcup_{i} X_{s_{i}}$ and $Y = \bigcup_{j} X_{t_{j}}$.
Then we have
$$XY = \bigcup_{i,j} X_{s_{i}t_{j}}$$
by Lemma~\label{le: top_completely_prome}.
Thus the product of open sets is always open.
\end{proof}

%%%%%%%%%%%%%%%%%%%%%%%%%%%%%%%%%%%%%%%%%%%%%%%%%%%%%%%%%%%%%%%%%%%%%%%%%%%%%%%%%%%%%%%%%%%%%%%%%%%%%%%%%%%%%%%%%%%%%%%%%%%%%%%%%%%%%%%%%%%%
We shall now describe the relationships between
$$S \text{ and } \mathsf{B}(\mathsf{G}_{CP}(S)),
\text{ and }
G \text{ and } \mathsf{G}_{CP}(\mathsf{B}(G)).$$
%%%%%%%%%%%%%%%%%%%%%%%%%%%%%%%%%%%%%%%%%%%%%%%%%%%%%%%%%%%%%%%%%%%%%%%%%%%%%%%%%%%%%%%%%%%%%%%%%%%%%%%%%%%%%%%%%%%%%%%%%%%%%%%%%%%%%%%%%%%%%

Let $S$ be a pseudogroup.
The set $X_{s}$, a set of completely prime filters, is a local bisection by Lemma~\ref{le: top_completely_prime} 
and it is by definition open.
Thus $X_{s} \in \mathsf{B}(\mathsf{G}_{CP}(S))$.
Define $\varepsilon \colon S \rightarrow \mathsf{B}(\mathsf{G}_{CP}(S))$ by $s \mapsto X_{s}$.
A homomorphism $\theta \colon S \rightarrow T$ is called a {\em pseudogroup morphism}
if it induces a frame map between the respective frames of idempotents.

\begin{proposition}\label{prop: counit} The function $\varepsilon \colon S \rightarrow \mathsf{B}(\mathsf{G}_{CP}(S))$ has the following properties:
\begin{enumerate}

\item It is a pseudogroup $\wedge$-morphism.

\item Every element of $\mathsf{B}(\mathsf{G}_{CP}(S))$ is a compatible join of elements in the image of $\varepsilon$.

\item The map $\varepsilon$ is an isomorphism of monoids if and only if the pseudogroup $S$ has the additional property that for all
$s,t \in S$ we have that $X_{s} = X_{t}$ implies that $s = t$.

\end{enumerate}
\end{proposition}
\begin{proof} (1) The map is a semigroup homomorphism by Lemma~\ref{le: top_completely_prime}(3),
it preserves binary meets by Lemma~\ref{le: top_completely_prime}(4) and it preserves compatible joins by Lemma~\ref{le: top_completely_prime}(5).
It is a monoid map because the idempotent completely prime filters are precisely the ones 
containing idempotents and so are precisely the ones that contain the identity of $S$.

(2) Each element of $\mathsf{B}(\mathsf{G}_{CP}(S))$ is an open local bisection and every open set, by definition,
is a union of open sets of the form $X_{s}$.

(3) Assume that $\varepsilon$ is injective.
It remains to prove that it is also surjective.
Suppose that $X_{s} \sim X_{t}$.
Then, in particular, $X_{s^{-1}t}$ is an identity.
It follows that every completely prime filter containing $s^{-1}t$ is an inverse subsemigroup.
Put $e = \mathbf{d}(s^{-1}t)$.
Clearly, $X_{s^{-1}t} = X_{e \wedge s^{-1}t}$.
We now use injectivity to deduce that $s^{-1}t$ is an idempotent.
By symmetry, we deduce that $s \sim t$.
from this result, and (2) above, we may deduce surjectivity. 
\end{proof}

A pseudogroup $S$ is said to be {\em spatial} if $X_{s} = X_{t}$ implies that $s = t$ for all $s,t \in S$. 
By  part (3) of  the preceding proposition, a pseudogroup is spatial if and only if $\varepsilon$ is an isomorphism of monoids.

%%%%%%%%%%%%%%%%%%%%%%%%%%%%%%%%%%%%%%%%%%%%%%%%%%%%%%%%%%%%%%%%%%%%%%%%%%%%%%%%%%%%%%%%%%%%%%%%%%%%%%%%%%%%%%%%%%%%%%%%%%%%%%%%%%%%%%%%%%%%
The proof of the following is straightforward.

\begin{lemma} Let $G$ be an \'etale groupoid.
For each $g \in G$ define $F_{g}$ to be the set of all open local bisections that contain $g$.
Then $F_{g}$ is a completely prime filter in the pseudogroup $\mathsf{B}(G)$.
\end{lemma}

Let $G$ be a \'etale groupoid.
Define $\eta \colon G \rightarrow \mathsf{G}_{CP}(\mathsf{B}(G))$ by $g \mapsto F_{g}$.
By the above lemma this is a well-defined map.

\begin{proposition}\label{prop: unit}
The function  $\eta \colon G \rightarrow \mathsf{G}_{CP}(\mathsf{B}(G))$ is a continuous covering functor.
\end{proposition}
\begin{proof} By Proposition~\ref{prop: etale_groupoids}
in an \'etale groupoid the open local bisections form a basis for the topology.
We may deduce from this, and the fact that the multiplication function is continuous,
that if $O_{g}$ is an open local bisection containing $g \in G$ and $g = hk$
then there are open local bisections $h \in O_{h}$ and $k \in O_{k}$ such that $O_{h}O_{k} \subseteq O_{g}$.
It now readily follows that $\eta$ is a functor.

We now prove that $\eta$ is a covering functor.
Suppose that $\mathbf{d}(g) = \mathbf{d}(h)$ and $\eta (g) = \eta (h)$.
Then there is an open local bisection $O$ that contains both $g$ and $h$.
But it then follows immediately from the definition of bisection that $g = h$.
Now suppose that $\eta (e) = \mathbf{d}(F)$ where $e$ is an identity and $F$
is a completely prime filter in $\mathsf{B}(G)$.
By definition $\mathbf{d}(F) = F_{e}$.
Let $b \in F$ be any open local bisection.
By assumption $e \in b^{-1}b$.
Thus we may find $g \in b$ such that $e = g^{-1}g$.
Consider $F_{g}$.
Then $\mathbf{d}(F_{g}) = \mathbf{d}(F)$ and $b \in F_{g} \cap F$.
By Lemma~\ref{le: filter_properties}, we have that $F_{g} = F$, as required.

It remains to show that $\eta$ is continuous.
Let $b \in \mathsf{B}(G)$ be an open local bisection.
Then
$$g \in \eta^{-1} (X_{b}) \Leftrightarrow F_{g} \in X_{b} \Leftrightarrow b \in F_{g} \Leftrightarrow g \in b.$$
Thus $\eta^{-1}(X_{b}) = b$.
\end{proof}

An \'etale groupoid is said to be {\em sober} if the map $\eta$ is a homeomorphism.
The following result shows that whether an \'etale groupoid is sober or not is determined by
its space of identities. 

\begin{proposition} Let $G$ be an \'etale groupoid.
Then $G$ is sober as an \'etale groupoid if and only if the space $G_{0}$ is sober.
\end{proposition}
\begin{proof} Observe first that a covering functor $\eta \colon G \rightarrow H$ is bijective if and only if
the function $\eta \mid G_{o} \colon G_{o} \rightarrow H_{o}$ is bijective.
It follows that $\eta \colon G \rightarrow \mathsf{G}_{CP}(\mathsf{B}(G))$ is bijective
if and only if $\eta \mid G_{o} \colon G_{o} \rightarrow \mathsf{G}(\mathsf{B}(G)_{o})$ is a bijective.
However, filters that are identities are determined by their idempotent elements,
and the idempotents in $\mathsf{B}(G)$ are the open subsets of $G_{o}$.
It follows that $\eta \colon G \rightarrow \mathsf{G}_{CP}(\mathsf{B}(G))$ is bijective
if and only if 
$\eta \colon G_{o} \rightarrow \mathsf{G}_{CP}(\mathsf{B}(G_{o}))$ is bijective.
Now observe that if $b$ is an open local bisection in $G$ then $\eta (b) = X_{b}$.
It follows that $\eta$ is always an open map.
We have therefore proved that $G$ is sober if and only if $G_{o}$ is sober.
\end{proof}

It follows from the above result that spectral groupoids are sober groupoids.

\begin{proposition}\label{prop: spatial_sober} \mbox{}
\begin{enumerate}

\item For every \'etale groupoid $G$ the pseudogroup $\mathsf{B}(G)$ is spatial.

\item For every pseudogroup $S$ the \'etale groupoid $\mathsf{G}_{CP}(S)$ is sober.

\end{enumerate}
\end{proposition}
\begin{proof} (1) Let $U$ and $V$ be two distinct open bisections in $\mathsf{B}(G)$.
Without loss of generality, there exists $g \in U$ and $g \notin V$.
But then $F_{g}$ is a completely prime filter in $\mathsf{B}(G)$ that contains $U$ and omits $V$.

(2) Let $S$ be a pseudogroup.
We show that every completely prime filter in  $\mathsf{B}(\mathsf{G}_{CP}(S))$
is of the form $F_{f}$ where $f \in \mathsf{G}_{CP}(S)$ is a uniquely determined element.
We show first that such an $f$ exists.
Define $f = \{s \in S \colon X_{s} \in F \}$.
From the fact that $F$ is completely prime and that the sets $X_{s}$ form a basis of open bisections for $\mathsf{G}_{CP}(S))$
it follows that $f$ is non-empty.
Using Lemma~\ref{le: top_completely_prime}, 
it is routine to verify that $f$ is a completely prime filter and by construction $F_{f} \subseteq F$.
Let $O \in F$.
Then $O$ can be written as a union of open bisections of the form $X_{s}$ for some $s$.
It follows that $O \in F_{f}$.

Now suppose that $F_{f} = F_{g}$ for completely prime filters $f$ and $g$ in $S$.
Let $s \in f$.
Then $f \in X_{s}$ and so by assumption $X_{s} \in F_{g}$ which gives $s \in g$.
It follows that $f \leq g$.
The reverse inclusion follows by symmetry.

It remains to show that $\eta$ is an open map.
Let $X_{s}$ be a basic open bisection in $\mathsf{G}_{CP}(S)$.
Then $\eta (X_{s})$ consists of all $F_{f}$ where $f \in X_{s}$.
But this is precisely the set $\{F_{f} \colon X_{s} \in F_{f} \}$
which is the basic open set $X_{X_{s}}$.\end{proof}

%%%%%%%%%%%%%%%%%%%%%%%%%%%%%%%%%%%%%%%%%%%%%%%%%%%%%%%%%%%%%%%%%%%%%%%%%%%%%%%%%%%%%%%%%%%%%%%%%%%%%%%%%%%%%%%%%%%%%%%%%%%%%%%%
We shall prove, for suitable definitions of morphisms,
that the functor
$$\mathsf{G} \colon \mathbf{PG}^{op} \rightarrow \mathbf{Etale}$$
is right adjoint to the functor
$$\mathsf{B} \colon \mathbf{Etale} \rightarrow \mathbf{PG}^{ op}.$$
Since the idempotent pseudogroups are the frames and the \'etale groupoids in which every element is an identity are just the
topological spaces, this would generalize the classical adjunction between categories of these structures;
see Theorem~1, page~476 of \cite{MM} and Theorem~1.4, page~42 of \cite{J}.
The problem is in defining appropriate morphisms.
Frames have top elements preserved by frame morphisms but this is not true of general pseudogroups.
This means that the inverse images of completely prime filters might be empty.
We do, however, have the following.

\begin{lemma} Let $\theta \colon S \rightarrow T$ be a pseudogroup $\wedge$-morphism.
If $F$ is a completely prime filter in $T$ and $\theta^{-1}(F)$ is non-empty then it is a completely prime filter.
\end{lemma}
\begin{proof} Let $a,b \in \theta^{-1}(F)$.
Then $\theta (a),\theta (b) \in F$.
But $F$ is a filter in a pseudogroup and so $\theta (a) \wedge \theta (b) \in F$.
We have assume that $\theta$ is a $\wedge$-morphism and so $\theta (a \wedge b) = \theta (a) \wedge \theta (b)$.
Thus $a \wedge b \in \theta^{-1}(F)$.
It is clear that $\theta^{-1}(F)$ is closed upwards, and it is completely prime because $\theta$ is a pseudogroup morphism.
\end{proof}

A function $\theta \colon S \rightarrow T$ between pseudogroups will be called {\em callitic}
if it satisfies two conditions:
\begin{enumerate}

\item it is a $\wedge$-morphism of pseudogroups, and

\item for each completely prime filter $F$ in $T$,  we have that $F \cap \mbox{im}(\theta) \neq \emptyset$.

\end{enumerate}

Thus $\mathbf{PG}$ denotes the category whose objects are pseudogroups and whose morphisms are the callitic maps.
We now have the following.

\begin{lemma}\label{le: u} Let $\theta \colon S \rightarrow T$ be a callitic morphism of pseudogroups.
Then
$$\theta^{-1} \colon \mathsf{G}_{CP}(T) \rightarrow \mathsf{G}_{CP}(S)$$
is a continuous covering functor.
\end{lemma}
\begin{proof} The assumption that $\theta$ is callitic simply ensures that for each completely prime filter $F$ the set
$\theta^{-1}(F)$ is non-empty and so by our above lemma is a completely prime filter.
It follows that $\theta^{-1} \colon \mathsf{G}_{CP}(T) \rightarrow \mathsf{G}_{CP}(S)$ is a well-defined function.
The bulk of the proof is taken up with showing that $\theta^{-1}$ is a functor.
Let $F$ be an identity completely prime filter in $T$.
Then $F$ contains idempotents by Lemma~\ref{le: filter_properties}.
In particular, it must contain the top idempotent in the frame $E(T)$ by upward closure.
Since $\theta$ is a frame morphism when restricted to the semilattice of idempotents
it follows that $\theta^{-1}(F)$ contains the top element of $E(S)$.
Thus $\theta^{-1}(F)$ is a completely prime filter containing idempotents and so it is an identity in the groupoid.

We prove that if $F$ and $G$ are completely prime filters such that $F^{-1} \cdot F = G \cdot G^{-1}$ then
$$( \theta^{-1} (F) \theta^{-1} (G) )^{\uparrow} = \theta^{-1} ( (FG)^{\uparrow}  ).$$
We prove first that
$$\theta^{-1} (F) \theta^{-1} (G) \subseteq \theta^{-1} (FG).$$
Let $s \in \theta^{-1} (F) \theta^{-1} (G)$.
Then $s = ab$ where $a \in \theta^{-1} (F)$ and $b \in \theta^{-1} (G)$.
Thus $\theta (s) = \theta (a) \theta (b) \in FG$.
It follows that $s \in \theta^{-1} (FG)$.
Observe that $\theta^{-1} (X)^{\uparrow} \subseteq \theta^{-1} (X^{\uparrow})$.
It follows that
$$( \theta^{-1} (F) \theta^{-1} (G) )^{\uparrow} \subseteq \theta^{-1} ( (FG)^{\uparrow}  ).$$
We now prove the reverse inclusion.
Let $s \in \theta^{-1} ( (FG)^{\uparrow} ).$
Then $\theta (s) \in F \cdot G$ and so
$fg \leq \theta(s)$ for some $f \in F$ and $g \in G$.
The map $\theta$ is assumed callitic and so there exists $v \in S$ such that $\theta (v) \in G$.
Consider the product $\theta (s) \theta (v)^{-1}$.
Since $\theta (s) \in F \cdot G$ and $\theta (v)^{-1} \in G^{-1}$ we have that
$\theta (s)\theta (v)^{-1} \in F \cdot G \cdot G^{-1} = F \cdot F^{-1} \cdot F = F$.
Thus $\theta (sv^{-1}) \in F$,  and we were given $\theta (v) \in G$, and clearly $(sv^{-1})v \leq s$.
Put $a = sv^{-1}$ and $b = v$.
Then $ab \leq S$ where $\theta (a) \in F$ and $\theta (b) \in G$.
It follows that $s \in (\theta^{-1} (F) \theta^{-1} (G) )^{\uparrow}$.

We may now show that $\theta^{-1}$ is a functor.
Let $F$ be any completely prime filter.
Observe that $\theta^{-1}(F)^{-1} = \theta^{-1} (F^{-1})$.
We have that
$$(\theta^{-1} (F^{-1}) \theta^{-1} (F))^{\uparrow}
=
(\theta^{-1} (F)^{-1} \theta^{-1} (F))^{\uparrow}
=
\mathbf{d} ( \theta^{-1} (F))$$
and
$$\theta^{-1} ( (F^{-1}F)^{\uparrow} ) = \theta^{-1} (\mathbf{d} (F)).$$
Hence
$$\theta^{-1} (\mathbf{d} (F))
=
\mathbf{d} ( \theta^{-1} (F)).$$
A dual result also holds and so $\theta^{-1}$ preserves the domain and codomain operations.
Suppose that $\mathbf{d}(F) = \mathbf{r} (G)$ so that $F \cdot G$ is defined.
By our calculation above $\mathbf{d} (\theta^{-1}(F)) = \mathbf{r} (\theta^{-1} (G))$
and so the product $\theta^{-1} (F) \cdot \theta^{-1} (G)$ is defined.
By our main result above we have that
$$\theta^{-1} (F \cdot G) = \theta^{-1} (F) \cdot \theta^{-1} (G),$$
as required.

The proof that $\theta^{-1}$ is a covering functor follows the same lines as the proof of Proposition~2.15 of \cite{Law3}.
It remains to show that it is continuous.
A basic open set of $\mathsf{G}_{CP}(S)$ has the form $X_{s}$ for some $s \in S$.
It is simple to check that this is pulled back to the set $X_{\theta (s)}$.
\end{proof}

%%%%%%%%%%%%%%%%%%%%%%%%%%%%%%%%%%%%%%%%%%%%%%%%%%%%%%%%%%%%%%%%%%%%%%%%%%%%%%%%%%%%%%%%%%%%%%%%%%%%%%%%%%%%%%%%%%%%%%%%%%%%%%%%%%%%%%%%%%%%%%
Let $\theta \colon S \rightarrow T$ be a pseudogroup $\wedge$-morphism.
We say that $\theta$ is {\em hypercallitic} if for each $t \in T$ we may write
$t = \bigvee_{i} t_{i}$ where $t_{i} \leq \theta (s_{i})$ for some $s_{i} \in S$.
It is immediate that hypercallitic maps are callitic.
Observe that $\theta$ is hypercallitic if and only if we may write
$t = \bigvee_{i} (t \wedge \theta (s_{i}))$ for some $s_{i} \in S$.

Hypercallitic maps can be seen to arise naturally by using a construction of Resende \cite{R1,R2}.\footnote{The following 
grew out of conversations with Pedro Resende}
Let $S$ be a pseudogroup.
Define $\mathcal{L}^{\vee} (S)$ to be the set of all order ideals of $S$ that are closed under compatible joins.
This is called the {\em enveloping quantale} of $S$.
It is, in particular, a frame with top element $S$.
Let $\theta \colon S \rightarrow T$ be a morphism of pseudogroups.
Then we may define a function $\bar{\theta} \colon \mathcal{L}^{\vee} (S) \rightarrow \mathcal{L}^{\vee} (T)$
by $\bar{\theta} (A) = [A^{\downarrow}]^{\vee}$
which means the downward closure of $A$ followed by the closure under compatible joins.

\begin{lemma}
The map $\bar{\theta}$ defined above is a frame map if and only if $\theta$ is hypercallitic.
\end{lemma}
\begin{proof} The map $\bar{\theta}$ is a frame map if and only if $\bar{\theta}(S) = T$.
That is if and only if $[\theta(S)^{\downarrow}]^{\vee} = T$.
This means that for each $t \in T$ we may find $t_{i} \in \theta (S)^{\downarrow}$ such that $t = \bigvee t_{i}$.
But $t_{i} \in \theta (S)^{\downarrow}$ means that $t_{i} \leq \theta (s_{i})$ for some $s_{i} \in S$.
\end{proof}

%The following is a simple alternative characterization of hypercallitic maps.

%\begin{lemma} Let $\theta \colon S \rightarrow T$ be a pseudogroup $\wedge$-morphism.
%Then $\theta$ is hypercallitic if and only if 
%for each $t \in T$ we have that
%$$t = \bigvee_{s \in S} (t \wedge \theta (s)).$$
%\end{lemma}
%\begin{proof} Suppose that $\theta$ is hypercallitic.
%Then $t = \bigvee_{i} t_{i}$ where $t_{i} \leq \theta (s_{i})$ for some $s_{i} \in S$.
%Now $t_{i} \leq t$ and so $t_{i} = t_{i} \wedge \theta (s_{i}) \leq t \wedge \theta (s_{i})$.
%But $t \wedge \theta (s_{i}) \leq t$.
%It follows that $t = \bigvee_{i} t \wedge \theta (s_{i})$.
%Clearly,  $t \leq \bigvee_{s \in S} (t \wedge \theta (s))$
%and equality holds because $t \wedge \theta (s) \leq t$.
%The converse is immediate.
%\end{proof}

We now have the companion to Lemma~\ref{le: u} going in the opposite direction.

\begin{lemma}\label{le: v} Let $\theta \colon G \rightarrow H$ be a continuous covering functor between two \'etale groupoids.
Then $\theta^{-1} \colon \mathsf{B}(H) \rightarrow \mathsf{B}(G)$ is hypercallitic.
\end{lemma}
\begin{proof}
The proof that we have a $\wedge$-morphism of pseudogroups basically follows the proof of Proposition~2.17 of \cite{Law3}.
It remains to show that $\theta^{-1}$ is hypercallitic.
By Proposition~\ref{prop: etale_groupoids},
the \'etale groupoid $G$ has a basis consisting of open local bisections.
Let $B$ be a non-empty open local bisection in $G$ and let $g \in B$.
Then $\theta (g) \in H$.
Clearly $H$ is an open set containing $\theta (g)$ but not a bisection.
However, since $H$ is \'etale, it follows that $H$ is a union of open local bisections and so $\theta (g) \in C_{g}$ an open local bisection $C_{g}$ in $H$.
Since $\theta$ is continuous $g \in \theta^{-1}(C_{g})$ is open and because $\theta$ is a covering functor $\theta^{-1} (C_{g})$ is a local bisection.
Thus $g \in B \cap \theta^{-1}(C_{g})$ an open local bisection in $G$.
It follows that we may write
$$B = \bigcup_{g \in B} (B \cap \theta^{-1} (C_{g})).$$
\end{proof}

The following will be needed in the proof of the adjunction theorem below.

\begin{lemma} Let $\theta \colon S \rightarrow T$ be a callitic morphism of pseudogroups where $T$ is spatial.
Then $\theta$ is in fact hypercallitic.
\end{lemma}
\begin{proof} Let $t \in T$ where $t \neq 0$.
Then the set of all completely prime filters $X_{t}$ containing $t$ cannot be empty
because then $X_{t} = X_{0}$ would imply that $t = 0$.
Put $t' = \bigvee_{s \in S} (t \wedge \theta (s))$.
We prove that
$$X_{t} = X_{t'}$$
from which the result follows by the spatiality of $T$.
Let $F \in X_{t}$.
Since $\theta$ is callitic there exists $\theta (s) \in F$ for some $s \in S$.
Thus $t \wedge \theta (s) \in F$.
It follows that $F \in  X_{t'}$.
Conversely, if $F \in  X_{t'}$ then clearly $F \in X_{t}$.
\end{proof}

We now come to our main theorem.

\begin{theorem}[Adjunction]
The functor
$$\mathsf{G}_{CP} \colon \mathbf{PG}^{op} \rightarrow \mathbf{Etale}$$
is right adjoint to the functor
$$\mathsf{B} \colon \mathbf{Etale} \rightarrow \mathbf{PG}^{ op}.$$
\end{theorem}
\begin{proof}
Let $\alpha \colon G \rightarrow \mathsf{G}_{CP}(S)$ be a continuous covering functor.
Now $\varepsilon \colon S \rightarrow \mathsf{B}(\mathsf{G}_{CP}(S))$ is a
pseudogroup $\wedge$-morphism and every element of the codomain is a join of elements in the image.
It follows that $\varepsilon$ is hypercallitic.
We also have by our results above that $\alpha^{-1}$ is hypercallitic.
Thus the map $\alpha^{-1} \varepsilon \colon S \rightarrow \mathsf{B}(G)$
given by
$$s \mapsto \alpha^{-1}(X_{s})$$
is hypercallitic.
Observe that $\mathsf{B}(G)$ is spatial and so all callitic maps into it are automaticaly hypercallitic.

Let $\beta \colon S \rightarrow \mathsf{B}(G)$ be a (hyper)callitic map.
Now $\eta \colon G \rightarrow \mathsf{G}_{CP} (\mathsf{B}(G))$ is a continuous covering functor
and so is $\beta^{-1} \colon \mathsf{G}_{CP}(\mathsf{B}(G)) \rightarrow \mathsf{G}_{CP}(S)$.
Thus the map
$\beta^{-1} \eta \colon G \rightarrow \mathsf{G}_{CP}(S)$
given by
$$g \mapsto \beta^{-1} (F_{g})$$
is a continuous covering functor.

We shall that these two constructions are mutually inverse.

Let $\beta \colon S \rightarrow \mathsf{B}(G)$ be a (hyper)callitic morphism of pseudogroups.
Define $\alpha (g) = \beta^{-1}(G_{g})$.
Then the map we get from $S$ to $\mathsf{B}(G)$
after applying the above procedures twice is the map $s \mapsto \alpha^{-1} (X_{s})$.
We have that
$g \in \alpha^{-1} (X_{s})
\Leftrightarrow \alpha (g) \in X_{s}
\Leftrightarrow \beta^{-1} (F_{g}) \in X_{s}
\Leftrightarrow s \in \beta^{-1} (F_{g})
\Leftrightarrow \beta (s) \in F_{g}
\Leftrightarrow g \in \beta (s)$.
It follows that $\beta (s) = \alpha^{-1}(X_{s})$, as required.

Let $\alpha \colon G \rightarrow \mathsf{G}_{CP}(S)$ be a continuous covering functor.
Define $\beta (s) = \alpha^{-1} (X_{s})$.
Then then map we get from $G$ to $\mathsf{G}_{CP}(S)$
after applying the above procedures twice is the map $g \mapsto \beta^{-1} (F_{g})$.
We have that
$s \in \beta^{-1} (F_{g})
\Leftrightarrow \beta (s) \in F_{g}
\Leftrightarrow \alpha^{-1}(X_{s}) \in F_{g}
 \Leftrightarrow g \in \alpha^{-1} (X_{s})
\Leftrightarrow \alpha (g) \in X_{s}
\Leftrightarrow s \in \alpha (g)$, as required.

Naturality is straightfoward to prove here, and so we have an adjunction.
\end{proof}

The map $\eta \colon G \rightarrow \mathsf{G}_{CP}(\mathsf{B} (G))$ of Proposition~\ref{prop: unit} is the
{\em unit} of the adjunction.
The map $\varepsilon \colon S \rightarrow \mathsf{B} (\mathsf{G}_{CP} (S))$ of Proposition~\ref{prop: counit}
is the {\em counit} of the adjunction.

Let $\mathbf{PG}_{sp}$ be the category of spatial pseudogroups and callitic pseudogroup morphisms.
Let $\mathbf{Etale}_{so}$ be the category of sober \'etale groupoids and continuous covering functors.
From the above theorem, Proposition~\ref{prop: spatial_sober} and general category theory we have proved the following.
At the level of objects, it was first proved in \cite{R2}.

\begin{theorem}\label{the: duality}[Duality between spatial pseudogroups and sober \'etale groupoids]
The category  $\mathbf{PG}_{sp}^{op}$ is equivalent to the category  $\mathbf{Etale}_{so}$.
\end{theorem}

%%%%%%%%%%%%%%%%%%%%%%%%%%%%%%%%%%%%%%%%%%%%%%%%%%%%%%%%%%%%%%%%%%%%%%%%%%%%%%%%%%%%%%%%%%%%%%%%%%%%%%%%%%%%%%%%%%%%%%%%%%%%%%%%%%%%%%%%%%%%%%%%%%
\subsection{Deducing dualities}

The main goal of this section is to derive the duality theorem for distributive inverse semigroups, Theorem~\ref{the: duality_theorem},
from Theorem~\ref{the: duality} above.
We shall first need a way of completing distributive inverse semigroups to pseudogroups.
Recall that $\mathsf{C}(S)$ is the Schein completion of $S$.
When $S$ is a distributive inverse semigroup, 
we shall work with a cut down version of $\mathsf{C}(S)$ which uses a finitary version of a construction due to Rinow \cite{R}.
Observe that if $A$ is a compatible order ideal
then it becomes a $\vee$-closed compatible order ideal when we include all the joins of finite subsets of $A$.
The operation $A \mapsto A^{\vee}$ satisfies the following properties:
\begin{description}

\item[{\rm (Cl1)}] $A \subseteq A^{\vee}$.

\item[{\rm (Cl2)}] If $A \subseteq B$ then $A^{\vee} \subseteq B^{\vee}$.

\item[{\rm (Cl3)}] $A^{\vee} = (A^{\vee})^{\vee}$.

\item[{\rm (Cl4)}] $A^{\vee}B^{\vee} = (AB)^{\vee}$.

\item[{\rm (Cl5)}] If $A$ consists entirely of idempotents then so does $A^{\vee}$.
\end{description}
The proofs are all straightforward except for (Cl4) which needs some comment.
The proof of the inclusion $A^{\vee}B^{\vee} \subseteq (AB)^{\vee}$ follows from the fact that multiplication distributes over compatible joins
whereas the proof of the reverse inclusion uses the fact that $A^{\vee}B^{\vee}$ is an order ideal. 
We denote by
$\mbox{Idl} (S)$
the set of all $\vee$-closed elements of $\mathsf{C}(S)$.

\begin{proposition} Let $S$ be a distributive inverse semigroup.
Then $\mbox{\rm Idl} (S)$ is a pseudogroup and the homomorphism $\iota \colon S \rightarrow \mbox{\rm Idl} (S)$
given by $s \mapsto s^{\downarrow}$
preserves binary joins of compatible pairs of elements.

In addition, $\mbox{\rm Idl}$ is left adjoint to the forgetful functor from the category
of pseudogroups and pseudogroup morphisms to the category of distributive inverse semigroups and their morphisms.

If $\theta \colon S \rightarrow T$ is a morphism of distributive inverse semigroups then
$\Theta \colon \mbox{\rm Idl}(S) \rightarrow \mbox{\rm Idl}(T)$ defined by $\Theta (A) = [\theta (A)^{\downarrow}]^{\vee}$
is the induced morphism of pseudogroups.

\end{proposition}
\begin{proof} It is clear that $\mbox{\rm Idl} (S)$ is closed under inverses, and it is closed under multiplication by (C4) above.
It follows that $\mbox{\rm Idl} (S)$ is an inverse semigroup.
Observe that $\mbox{\rm Idl} (S)$ is actually an inverse subsemigroup of $\mathsf{C}(S)$ and so the natural partial orders agree.
A compatible set of elements in  $\mbox{\rm Idl} (S)$ has a join in $\mathsf{C}(S)$ and this can be reflected into  $\mbox{\rm Idl} (S)$
using the operation $A \mapsto A^{\vee}$.
Thus every compatible subset of $\mbox{\rm Idl} (S)$ has a join.
It is now easy to prove using the properties of the $\vee$-closure operation that $\mbox{\rm Idl} (S)$ is infinitely distributive.
Observe that the idempotents of $\mbox{\rm Idl} (S)$ are the $\vee$-closed order ideals in the meet semilattice $E(S)$
and that there is a maximum idempotent $E(S)$ and so the semilattice of idempotents of $\mbox{\rm Idl} (S)$ forms a frame.
In the monoid case, this can be deduced from Corollary in Section~2.11 of \cite{J}.
It follows that $\mbox{\rm Idl} (S)$ is a pseudogroup.

The map $\iota \colon S \rightarrow \mbox{\rm Idl} (S)$ is a homomorphism.
Suppose that $c = a \vee b$ in $S$.
Clearly $\iota (a), \iota (b) \subseteq \iota (c)$.
But any {\em $\vee$-closed} element of $\mathsf{C}(S)$ that contains $a$ and $b$ must contain $c$.
It follows that $\iota (c) = \iota (a) \vee \iota (b)$.

Let $\alpha \colon S \rightarrow T$ be a homomorphism to a pseudogroup that preserves finite compatible joins.
Then there is a unique morphism of pseudogroups $\bar{\alpha} \colon \mbox{\rm Idl} (S) \rightarrow T$
such that $\bar{\alpha} \iota = \alpha$ defined by $\bar{\alpha} (A) = \bigvee A$.

The proof of the last claim is routine.
\end{proof}

We call the pseudogroup $\mbox{\rm Idl}(S)$ the {\em $\mbox{\rm Idl}$-completion} of $S$.

Prime filters in distributive inverse semigroups and completely prime filters in their $\mbox{Idl}$-completions are related as follows.

\begin{lemma}\label{prime-vs-completelyprime} Let $S$ be a distributive inverse semigroup.

If $P$ is a prime filter in $S$ define
$$P^{u} = \{ A \in  \mbox{\rm Idl} (S) \colon  A \cap P \neq \emptyset \}.$$
Then $P^{u}$ is a completely prime filter in $\mbox{\rm Idl} (S)$.

If $F$ is a completely prime filter in $\mbox{\rm Idl} (S)$
define
$$F^{d} = \{s \in S \colon s^{\downarrow} \in F \}.$$
Then $F^{d}$ is a prime filter in $S$.

The above two operations are mutually inverse and set up an order
isomorphism between the poset of prime filters on $S$ and the poset of
completely prime filters on  $\mbox{\rm Idl} (S)$.
\end{lemma}
\begin{proof}
Clearly the set $P^{u}$ is closed upwards.
Let $A,B \in P^{u}$.
Then $s \in A \cap P$ and $t \in B \cap P$.
But $P$ is a filter and so there exists $p \in P$ such that $p \leq s,t$.
But then $p \in A \cap B$ and so $P^{u}$ is closed under binary intersections.
Suppose that $\bigvee A_{i} \in P^{u}$.
Thus there exists $p \in P$ such that $p \in \bigvee_{i} A_{i}$.
By definition $p = \vee_{j=1}^{m} a_{j}$ for some finite set of elements $a_{j}$ in the $A_{i}$.
But $P$ is a prime filter and so $a_{k} \in P$ for some $k$.
Thus one of the $A_{i}$, the one containing $a_{k}$, belongs to $P^{u}$ as required.
Thus $P^{u}$ is a completely prime filter.

We now show that $F^{d}$ is a prime filter.
Let $s,t \in F^{d}$.
Then $s^{\downarrow},t^{\downarrow} \in F$.
Thus $A = s^{\downarrow} \cap t^{\downarrow} \in F$.
Now $A = \bigvee_{a \in A} a^{\downarrow} \in F$ and so $a^{\downarrow} \in F$ for some $a \in A$.
Thus $a \in F^{d}$ and $a \leq s,t$.
It is clear that $F^{d}$ is closed upwards.
It remains to show that $F^{d}$ is a prime filter.
Let $s \vee t \in F^{d}$.
Then $(s \vee t)^{\downarrow} \in F$.
But $(s \vee t)^{\downarrow} = s^{\downarrow} \vee t^{\downarrow} \in F$.
It follows that $s^{\downarrow} \in F$ or $t^{\downarrow} \in F$.
Thus $s \in F$ or $t \in F$.

It is now routine to check that these two operations are mutually inverse and order-preserving.
\end{proof}

We shall now characterize the pseudogroups that arise as  $\mbox{\rm Idl}$-completions.
To do this we need the following definition.
We say that the compatible subset $X$ of a pseudogroup $S$ is a {\em covering} of the element $a$ if $a \leq \bigvee X$.
An element $a \in S$ in a pseudogroup $S$ is said to be {\em finite} if for any compatible subset $X \subseteq S$ such that
$a \leq \bigvee X$ there exists a finite subset $Y$ of $X$ such that $a \leq \bigvee Y$.
In other words, every covering has a finite subcovering.
In the case of the frames of open sets of a topological space the finite elements are just the compact ones.
It is worth noting that the inequalities can be replaced by equalities; see page~63 of \cite{J}.
We denote the set of finite elements of a pseudogroup $S$ by $\mathsf{K}(S)$.

\begin{lemma} Let $S$ be a pseudogroup.
\begin{enumerate}

\item If $a$ is finite then $a^{-1}$ is finite.

\item If $a$ is any element and $e$ is a finite idempotent such that $e \leq a^{-1}a$ then $ae$ is finite.

\item If $a$ is finite then $a^{-1}a$ is finite, and dually.

\item If $a$ and $b$ are finite and $a^{-1}a = bb^{-1}$ then $ab$ is finite.

\end{enumerate}
\end{lemma}
\begin{proof} (1) Straightforward.

(2) Let $a$ be any element and $e$ a finite idempotent $e \leq a^{-1}a$.
We prove that $ae$ is finite.
Suppose that $ae \leq \bigvee x_{i}$.
Then $e = a^{-1}ae \leq \bigvee a^{-1}x_{i}$.
But $e$ is finite and so $e \leq \bigvee_{i=1}^{m} a^{-1}x_{i}$.
Thus $ae \leq \bigvee_{i=1}^{m} aa^{-1}x_{i} \leq \bigvee_{i=1}^{m} x_{i}$.
It follows that $ae$ is finite.

(3) Let $a$ be any finite element.
Suppose that $a^{-1}a \leq \bigvee x_{i}$.
Then $a \leq \bigvee ax_{i}$.
Thus $a \leq \bigvee_{i=1}^{m} ax_{i}$ since $a$ is finite.
Hence $a^{-1}a \bigvee_{i=1}^{m} a^{-1}ax_{i} \leq  \bigvee_{i=1}^{m} x_{i}$.
Thus $a^{-1}a$ is finite.

(4) Let $a$ and $b$ be any finite elements where $a^{-1}a = bb^{-1}$.
We prove that $ab$ is finite.
Suppose that $ab \leq \bigvee x_{i}$.
Then $a^{-1}abb^{-1} \leq \bigvee a^{-1}x_{i}b^{-1}$.
By assumption $a^{-1}abb^{-1}$ is a finite idempotent.
Thus we may write $a^{-1}abb^{-1} \leq \bigvee_{i=1}^{m} a^{-1}x_{i}b^{-1}$.
Hence $ab  \leq \bigvee_{i=1}^{m} aa^{-1}x_{i}b^{-1}b \leq \bigvee_{i=1}^{m} x_{i}$.
It follows that $ab$ is finite.
\end{proof}

The above lemma tells us that the finite elements in a pseudogroup always form an ordered groupoid \cite{Law2}.

\begin{lemma} Let $S$ be a pseudogroup.
\begin{enumerate}

\item The finite elements of $S$ form an inverse subsemigroup if and only if the finite idempotents form a subsemigroup.

\item If the finite elements form an inverse subsemigroup they form a distributive inverse semigroup.

\item Every element of $S$ is a join of finite elements if and only if every idempotent is a join of finite idempotents.

\end{enumerate}
\end{lemma}
\begin{proof} (1) Let $a$ and $b$ be arbitrary finite elements.
Then $a^{-1}a$ and $bb^{-1}$ are both finite and so $e = a^{-1}abb^{-1}$ is finite
and consequently $ab = (ae)(eb)$ is finite.

(2) Observe that if $a$ and $b$ are compatible finite elements then $a \vee b$ is finite.

(3) Only one direction needs proving.
Suppose that every idempotent is a join of finite idempotents.
Let $a$ be an arbitrary element.
By assumption we may write $a^{-1}a = \bigvee e_{i}$ where $e_{i} \leq a^{-1}a$ and are finite.
Thus $a = \bigvee ae_{i}$ and by Lemma~3.18(1) the elements $ae_{i}$ are all finite.
\end{proof}

A pseudogroup $S$ is said to be {\em coherent} if the set of its finite elements forms a distributive inverse subsemigroup
and if every element of $S$ is a join of finite elements.
  
\begin{proposition}\label{prop: coherent_pseudogroups} A pseudogroup $S$ is coherent if and only if there exists a distributive inverse semigroup $T$ such
that $S$ is isomorphic to $\mbox{\rm Idl} (T)$.  
In fact, any coherent pseudogroup $S$ is canonically isomorphic to $\mbox{\rm Idl} (\mathsf{K}(S))$.
\end{proposition}
\begin{proof} 
Let $T$ be a distributive inverse semigroup.
We prove first  that the finite elements of $\mbox{\rm Idl} (T)$ are precisely the elements of the form $t^{\downarrow}$ where $t \in T$.
Observe that $t^{\downarrow} = \overline{t^{\downarrow}}$.
Let $t^{\downarrow} = \bigvee A_{i}$.
Then $t$ is in the $\vee$-closure of  $\bigcup A_{i}$.
Thus there is a finite set of elements $a_{1}, \ldots, a_{m} \in \bigcup A_{i}$ such that $t = \vee a_{j}$.
But this implies that $t^{\downarrow}$ is the join of only finitely many of the $A_{i}$.
Thus $t^{\downarrow}$ is finite.
Suppose now that $A$ is a finite element.
We have that $A = \bigvee_{a \in A} a^{\downarrow}$.
By assumption there are finitely many elements $a_{1}, \ldots, a_{m} \in A$ such that
$A = \bigvee a_{i}^{\downarrow}$.
But if $a = \vee a_{i}$ then $A = a^{\downarrow}$, as required.
Clearly, every element of  $\mbox{\rm Idl} (T)$ is a compatible join of finite elements.
It follows that $\mbox{\rm Idl} (T)$ is coherent and that its finite elements form a distributive
inverse semigroup isomorphic to $T$. 

Now suppose that $S$ is a coherent pseudogroup.
Put $T = \mathsf{K}(S)$, a distributive inverse semigroup by assumption.
Define $\theta \colon \mbox{\rm Idl} (T) \rightarrow S$ by $\theta (A) = \bigvee A$.
This is surjective since every element of $S$ is the join of finite elements.
Suppose that $\theta (A) = \theta (B)$.
Let $a \in A$.
Then $a \leq \bigvee A$.
Thus $a \leq \bigvee B$.
But $a$ is a finite element and so there is a finite subset $b_{1}, \ldots, b_{m} \in B$ such that
$a \leq \vee_{i} b_{i}$.
But $B$ is  $\vee$-closed and so $\bigvee_{i} b_{i} \in B$ that implies $a \in B$.
We have proved that $A \subseteq B$.
The reverse inclusion is proved similarly.
It follows that $\theta$ is a bijection.
It is clearly a homomorphism.
We have proved that $\mathsf{K}(\mbox{\rm Idl} (T))$ is isomorphic to $S$.
\end{proof}

A pseudogroup morphism between coherent pseudogroups is said to be {\em coherent} if it preserves finite elements.
The proof of the following is now straightforward. In fact, the previous proposition just gives the object-part of the statement.

\begin{proposition}
The category of distributive inverse semigroups and their morphisms is equivalent to the category of coherent pseudogroups
and coherent pseudogroup morphisms.
\end{proposition}

The above result is not quite what we need because the morphisms are too general.

\begin{proposition} Let $\theta \colon S \rightarrow T$ be a morphism of distributive inverse semigroups
and let $\Theta \colon \mbox{\rm Idl}(S) \rightarrow \mbox{\rm Idl}(T)$ be the induced pseudogroup map.
Then $\theta$ is callitic if and only if $\Theta$ is callitic.
\end{proposition}
\begin{proof}
We prove first that $\theta^{-1}$ prime filters to non-empty sets 
if and only if
$\Theta^{-1}$ maps completely prime filters to non-empty sets.
We use Lemma~\ref{prime-vs-completelyprime}.
Suppose that $\Theta^{-1}$ maps completely prime filters to non-empty sets.
Let $P$ be any prime filter in $T$.
Then $P^{u}$ is a completely prime filter in $\mbox{\rm Idl}(T)$.
By assumption $\Theta^{-1}(P^{u})$ is non-empty.
Thus there is an element $A \in  \mbox{\rm Idl}(S)$ such that $\Theta (A) \in P^{u}$.
By definition, $[\theta(A)^{\downarrow}]^{\vee} \in P^{u}$.
We quickly deduce, using the fact that $P^{u}$ is completely prime, 
that for some $a \in A$ we have that $\theta (a)^{\downarrow} \in P^{u}$.
This implies that $\theta (a) \in P$, as required.
The proof of the converse is straightforward.

We now prove that $\theta$ is weakly meet preserving if and only if $\theta$ is a $\wedge$-map.
Suppose first that $\theta$ is weakly meet preserving.
we need to show that $\Theta (A \cap B) = \Theta (A) \cap \Theta (B)$.
It is immediate that the lefthand-side is contained in the righthand-side.
We prove the reverse inclusion.
That is, we have to prove that
$$[\theta (A)^{\downarrow}]^{\vee} \cap [\theta (B)^{\downarrow}]^{\vee} \subseteq [\theta (A \cap B)^{\downarrow}]^{\vee}.$$ 
Let $x$ belong to the lefthand-side.
This means that
$$x = \vee_{i=1}^{m} x_{i} = \vee_{j=1}^{n} y_{j}$$
where $x_{i} \in \theta (A)^{\downarrow}$ and $y_{j} \in \theta (B)^{\downarrow}$.
Thus $x_{i} \leq \theta (a_{i})$ for some $a_{i} \in A$ 
and $y_{j} \leq \theta (b_{j})$ for some $b_{j} \in B$.
By Lemma~\ref{le: meets_and_joins} , we have that $x = \vee_{i,j} (x_{i} \wedge y_{j})$.
It follows that $x_{i} \wedge y_{j} \leq \theta (a_{i}), \theta (b_{j})$.
We now use the fact that $\theta$ is weak meet preserving.
It follows that there is $c_{ij} \leq a_{i},b_{j}$ such that $x_{i} \wedge x_{j} \leq \theta (c_{ij})$.
It follows that $x$ belongs to $[\theta (A \cap B)^{\downarrow}]^{\vee}$.

We now prove the converse.
Assume that $\Theta$ is a $\wedge$-map.
We prove that $\theta$ is weak meet preserving.
Let $t \leq \theta (a), \theta (b)$.
We have that $a^{\downarrow}, b^{\downarrow} \in  \mbox{\rm Idl}(S)$.
Thus $a^{\downarrow} \cap b^{\downarrow} \in  \mbox{\rm Idl}(S)$, since we are working inside a pseudogroup.
By assumption, we have that 
$$[\theta (a^{\downarrow} \cap b^{\downarrow})^{\downarrow}]^{\vee}
=
[\theta (a^{\downarrow})^{\downarrow}]^{\vee} \cap [\theta (b^{\downarrow})^{\downarrow}]^{\vee}
= \theta (a)^{\downarrow} \cap \theta (b)^{\downarrow}.$$
By assumption 
$t \in [\theta (a^{\downarrow} \cap b^{\downarrow})^{\downarrow}]^{\vee}$.
Then $t = \vee_{i=1}^{m} t_{i}$ where $t_{i} \in \theta (a^{\downarrow} \cap b^{\downarrow})^{\downarrow})$.
It follows that $t_{i} \leq \theta (c_{i})$ where $c_{i} \leq a,b$.
Put $c = \vee_{i=1}^{m} c_{i}$ which exists since the $c_{i}$ are all bounded above.
Then $c \leq a,b$ and $t \leq \theta (c)$, as required, using the fact that
$\theta$ is a morphism of distributive inverse semigroups.
\end{proof}

We have therefore proved the following.

\begin{corollary}\label{cor: edna}
The category of distributive inverse semigroups and their callitic morphisms is equivalent to the category of coherent pseudogroups
and coherent callitic morphisms.
\end{corollary}

The following relies on Section~2.

\begin{proposition}\label{coherent-implies-spatial}
Every coherent pseudogroup is spatial.
\end{proposition}
\begin{proof}
Let $S$ be a coherent pseudogroup.
By Proposition~\ref{prop: coherent_pseudogroups}, 
we may assume that $S = \mbox{Idl}(T)$ where $T$ is a distributive inverse semigroup.
Let $A,B \in \mbox{Idl}(T)$ be distinct elements.
We shall construct a completely prime filter that contains one of these elements but not the other.
Without loss of generality, we may assume that there is $b \in B$ such that $b \notin A$.
It follows that
$$b^{\uparrow} \cap B = \emptyset.$$
Clearly $B$ is a $\vee$-closed order ideal.
Thus by Lemma~\ref{le: maximal}, there exists a $\vee$-closed order ideal $P$ that contains $B$,
is disjoint from $b^{\uparrow}$,
and is a maximal $\vee$-closed order ideal with respect to these two conditions.
By Lemma~\ref{le: maximal_prime}, $P$ is a prime $\vee$-closed order ideal.
Thus by Lemma~\ref{le: ideal_filter}, the set $F = T \setminus P$ is a prime filter in $S$ that contains $b$ and is disjoint from $B$.
By Lemma~\ref{prime-vs-completelyprime}, we have that $F^{u}$ is a completely prime filter in $S = \mbox{Idl}(T)$.
By definition $B \in F^{u}$ and $A \notin F^{u}$, as required.
\end{proof}

We shall now turn to the spectral groupoids introduced in Section~3.
We have seen that if $G$ is such a groupoid, then its set of compact-open local bisections
$\mathsf{KB}(G)$ is a distributive inverse semigroup by Proposition~\ref{prop: distributive_inverse}.
This sits inside the pseudogroup $\mathsf{B}(S)$.
It is clear that the finite elements of $\mathsf{B}(S)$ are just the elements of $\mathsf{KB}(G)$
and so the finite elements form a distributive inverse semigroup.
In addition, in a spectral groupoid, the compact-open local bisections form a basis.
We have therefore proved the following.

\begin{proposition}\label{prop: doris} Let $G$ be a spectral groupoid.
Then $\mathsf{B}(G)$ is a coherent pseudogroup.
\end{proposition}

We now go in the opposite direction.

\begin{proposition}\label{prop: fred} Let $S$ be a coherent pseudogroup.
Then $\mathsf{G}_{CP}(S)$ is a spectral groupoid. 
Moreover, the isomorphism  $\varepsilon$ establishes a bijection 
between  the finite elements of $S$ and the compact-open local bisections.
\end{proposition}
\begin{proof} By Proposition~\ref{prop: spatial_sober}, the groupoid  $\mathsf{G}_{CP}(S)$ is sober.
Now $S$ is isomorphic to $\mathsf{B} (\mathsf{G}_{CP}(S))$ via $\varepsilon$ 
since every coherent pseudogroup is spatial by Proposition \ref{coherent-implies-spatial}.
Under this isomorphism, finite elements of $S$ are mapped to the compact-open local bisections of $\mathsf{G}_{CP}(S)$.
By the coherence of $S$, every $s\in S$ is the joint of finite elements and hence every $X_s$ is a union of compact-open local bisections 
which stem from the  range of $\varepsilon$.
Thus every open local bisection is a union of compact-open local bisections.  
Thus the compact-open local bisections form a basis.
Moreover, every compact-open local bisection comes comes from a finite element of $S$ 
since the finite elements are closed under finite joins by coherence.
We have that the set of compact-open local bisections is closed under subset multiplication
because the finite elements of $S$ are closed under multiplication. 
\end{proof}

We may now derive Theorem~\ref{the: duality_theorem} from Theorem~\ref{the: duality}.

\begin{theorem}[New proof for duality for distributive inverse semigroups]\label{ddis}
The category of distributive inverse semigroups and their callitic morphisms is dually equivalent to the category
of spectral groupoids and coherent continuous covering functors.
\end{theorem}
\begin{proof} 
Let $\theta \colon \: G \rightarrow H$ be a coherent continous covering functor between spectral groupoids.
Then $\mathsf{B}(\theta) \colon \mathsf{B}(H) \rightarrow \mathsf{B}(G)$ is a callitic morphism of pseudogroups.
Since $G$ and $H$ are both spectral, both  $\mathsf{B}(G)$ and $\mathsf{B}(H)$ are coherent by Proposition~\ref{prop: doris}.
Since $\theta$ is coherent,  $\mathsf{B}(\theta)$ maps finite elements to finite elements by Lemma~\ref{le: callitic}.
Thus $\mathsf{B}(\theta)$ is coherent.

Let $\theta \colon S \rightarrow T$ be a callitic coherent morphism between coherent pseudogroups.
Then $\mathsf{G}_{CP}(\theta) \colon \mathsf{G}_{CP}(T) \rightarrow \mathsf{G}_{CP}(S)$ is a continuous covering functor.
By Proposition~\ref{prop: fred}, both $\mathsf{G}_{CP}(S)$ and $\mathsf{G}_{CP}(T)$ are spectral groupoids.
Since $\theta$ is also coherent, the inverse image under $\mathsf{G}_{CP}(\theta)$ of every compact-open local bisection is a compact-open local bisection.
It remains to show that the inverse image of every compact-open set is compact-open.
This is certainly open so it only remains to show that it is compact.
However, since the groupoid is spectral, the compact-open bisections form a basis.
Thus every compact-open set may be written as a finite union of compact-open local bisections.
It follows, that the inverse image is a finite union of compact sets and so is compact.
It follows that $\mathsf{G}_{CP}(\theta)$ is coherent.

Finally, the category of  distributive inverse semigroups and their callitic morphisms is equivalent to
the category of coherent pseudogroups and their callitic morphisms by Corollary~\ref{cor: edna}.
\end{proof}

%%%%%%%%%%%%%%%%%%%%%%%%%%%%%%%%%%%%%%%%%%%%%%%%%%%%%%%%%%%%%%%%%%%%%%%%%%%%%%%%%%%%%%%%%%%%%%%%%%%%%%%%%%%%%%%%%%%%%%%%%%%%%%%%
\section{Booleanizations and Paterson's universal groupoid}

In this section, we shall develop the theory of Booleanizations of spectral groupoids and thereby
of distributive inverse semigroups and apply our theory to providing a new interpretation of Paterson's universal groupoid.

%%%%%%%%%%%%%%%%%%%%%%%%%%%%%%%%%%%%%%%%%%%%%%%%%%%%%%%%%%%%%%%%%%%%%%%%%%%%%%%%%%%%%%%%%%%%%%%%%%%%%%%%%%%%%%%%%%%%%%%%%%%%%%%%%
\subsection{Booleanization}
We shall show in this section how to construct a Boolean inverse semigroup from each distributive inverse semigroup
in the freest possible way.

Let $S$ be a distributive inverse semigroup and $\mathsf{G}_{P}(S)$ its associated spectral groupoid.
Recall that a basis is given by $\pi = \{X_{s} \colon s \in S \}$ where $X_{s}$ is the set of all prime filters containing $s$.
Define $\Pi = \{ X_{s} \cap X_{t}^{c} \colon s,t \in S, \,  t \leq s \}$.
It is convenient to define $X_{s;t} = X_{s} \cap X_{t}^{c}$ where $t \leq s$.

\begin{lemma}\label{le: patch} Let $S$ be a distributive inverse semigroup.
\begin{enumerate}

%1
\item $\Pi$ is a basis for a topology on the groupoid $\mathsf{G}_{P}(S)$.

%2
\item If $S$ is Boolean then the topologies generated by $\pi$ and $\Pi$ are the same.
 
%3
\item Let $s \sim t$ and $u \sim v$ and $u \vee v \leq s \vee t$.
Then
$$X_{s \vee t ; u \vee v} = X_{s; (u \vee v)s^{-1}s} \cup X_{t ; (u \vee v)t^{-1}t}.$$

%4
\item $X_{s;t}X_{u;v} = X_{su;sv \vee tu \vee tv}$.

%5
\item $X_{s;t}^{-1} = X_{s^{-1} ; t^{-1}}$.

\end{enumerate}
\end{lemma}
\begin{proof} (1). Observe that $X_{s:0} = X_{s}$ and that $X_{s;s} = \emptyset$.
Suppose that $(X_{s} \cap X_{t}^{c}) \cap (X_{u} \cap X_{v}^{c}) \neq \emptyset$ where $t \leq s$ and $v \leq u$.
Let $P$ be any prime filter belonging to this set.
Then $s,u \in P$ and $t,v \notin P$.
It follows that there exists $z \in P$ such that $z \leq s,u$.
Put $a = zt^{-1}t \leq t$ and $b = zv^{-1}v \leq v$.
Since $a,b \leq z$ we have that $c = a \vee b$ exists.
Consider now the set $X_{z} \cap X_{c}^{c}$.
Suppose that $c \in P$.
Then since $P$ is a prime filter either $a \in P$ or $b \in P$.
Without loss of generality, suppose that $a \in P$.
Then $t \in P$ which is a contradiction.
It follows that $P \in X_{z} \cap X_{c}^{c}$.
Now let $Q \in X_{z} \cap X_{c}^{c}$.
Then $a,b \in Q$.
Suppose that $t \in Q$.
Then since $a = zt^{-1}t$ we would have $a \in Q$ which is a contradiction since $Q$ omits $c$.
Thus $t \notin Q$.
Similarly, $v \notin Q$.
It follows that $P \in X_{z} \cap X_{c}^{c} \subseteq (X_{s} \cap X_{t}^{c}) \cap (X_{u} \cap X_{v}^{c})$.
Thus $\Pi$ is the basis for a topology on $\mathsf{G}_{P}(S)$. 

(2). Let $t \leq s$. Then $\mathbf{d}(t) \leq \mathbf{d}(s)$.
By Lemma~\ref{le: relative_complement}, we may construct an element $s \backslash t$ such that
$s \backslash t \leq s$, $(s \backslash t) \wedge t = 0$ and $(s \backslash t) \vee t = s$.
Put $u = s \backslash t$.
We prove that $X_{s} \cap X_{t}^{c} = X_{u}$.
Let $P$ be a prime filter in  $X_{s} \cap X_{t}^{c}$.
Then $s \in P$.
But $P$ is a prime filter and so either $u \in P$ or $t \in P$.
We cannot have the latter and so $u \in P$ and $P \in X_{u}$.
Conversely, suppose that $P \in X_{u}$.
Then $s \in P$ and we cannot have $t \in P$ because $u \wedge t = 0$.

(3). Observe first that $(u \vee v)s^{-1}s \leq s$.
We have that $(u \vee v)s^{-1}s \leq (s \vee t)s^{-1}s$.
But $(s \vee t)s^{-1}s = s \vee ts^{-1}s$ and since $s$ and $t$ are compatible $ts^{-1}$ is an idempotent.
Thus $ts^{-1}s \leq s$.
It follows that $s \vee ts^{-1}s = s$, as required.

Let $P \in X_{s \vee t ; u \vee v}$.
Thus $s \vee t \in P$.
Since $P$ is a prime filter either $s \in P$ or $t \in P$.
Without loss of generality, suppose that $s \in P$.
Observe that if $(u \vee v)s^{-1}s \in P$
then $u \vee v \in P$ which is a contradiction.
It follows that the lefthandside is contain in the righthandside.

Now let $P \in X_{s; (u \vee v)s^{-1}s} \cup X_{t ; (u \vee v)t^{-1}t}$.
Without loss of generality, suppose that $P \in X_{s; (u \vee v)s^{-1}s}$.
Then $s \in P$ implies that $s \vee t \in P$.
Suppose that $u \vee v \in P$.
Then $(u \vee v)s^{-1}s \in P$, using the fact that since $P$ is a filter we have that $P = PP^{-1}P$,
which is a contradiction.
It follows that the righthandside is contained in the lefthandside. 

(4). Observe that since $t \leq s$ and $v \leq u$ we have that $sv, tu, tv \leq su$.
Thus the join $sv \vee tu \vee tv$ exists.

Let $X \in X_{s;t}$, $Y \in  X_{u;v}$ where $\mathbf{d}(X) = \mathbf{r}(Y)$.
Then $su \in X \cdot Y$.
Suppose that $sv \vee tu \vee tv \in X \cdot Y$.
But $X \cdot Y$ is a prime filter and so either $sv \in X \cdot Y$ or $tu \in X \cdot Y$ or $tv \in X \cdot Y$ but each of these is ruled out
as follows.
Suppose that $sv \in X \cdot Y$.
Then $s^{-1}sv \in X^{-1} \cdot X \cdot Y = Y$ and so $v \in Y$, a contradiction.
A similar argument applies to $tu$.
Suppose that $tv \in X \cdot Y$.
Then $s^{-1}tv \in X^{-1} \cdot X \cdot Y = Y$.
But $s^{-1}t$ is an idempotent.
It follows that $v \in Y$, which is a contradiction.
Thus $X_{s;t}X_{u;v} \subseteq X_{su;sv \vee tu \vee tv}$.

We now prove the reverse inclusion.
Put $w = sv \vee tu \vee tv$. 
Let $Z \in  X_{su;w}$.
Observe that $\mathbf{d}(u) \in \mathbf{d}(Z)$.
It follows that $Y = (u \mathbf{d}(Z))^{\uparrow}$ is a prime filter containing $u$.
Suppose $v \in Y$.
Then $ue \leq v$ for some $e \in E(\mathbf{d}(Z))$.
Thus $sue \leq sv$
But $(su)e \in ZZ^{-1}Z = Z$.
Thus $sv \in Z$, which is a contradiction.
Thus $Y \in X_{u;v}$.
Observe that $\mathbf{d}(s) \in \mathbf{r}(Y)$.
It follows that $X = (s \mathbf{r}(Y))^{\uparrow}$ is a prime filter containing $s$,
that $\mathbf{d}(X) = \mathbf{r}(Y)$ and that $Z = X \cdot Y$.
Suppose $t \in X$.
Then we may write $sueu^{-1} \leq t$ for some $e \in E(\mathbf{d}(Z))$.
Thus $(su)e(u^{-1}s^{-1})(su) \leq ts^{-1}su$.
But $(su)e(u^{-1}s^{-1})(su) \in Z$ and since $t \leq s$ we have that $ts^{-1}su = tu$.
Thus we deduce that $tu \in Z$, which is a contradiction.
Hence $X \in X_{s;t}$.

(5). We have that $X \in X_{s;t}$ if and only if $X^{-1} \in X_{s^{-1};t^{-1}}$
and we know that $X$ is a prime filter if and only if $X^{-1}$ is a prime filter.
\end{proof}

The topology generated by $\Pi$ is a generalization of the {\em patch topology} described on page~72 of \cite{J}.
We denote the groupoid  $\mathsf{G}_{P}(S)$ equipped with the patch topology by  $\mathsf{G}_{P}(S)^{\dagger}$.
Part (2) of the above lemma suggests that we only get a new topology if the distributive inverse semigroup is not Boolean.
We shall prove this claim later.

\begin{proposition} Let $S$ be a distributive inverse semigroup.
Then the groupoid $\mathsf{G}_{P}(S)$ equipped with the patch topology is Boolean.
\end{proposition}
\begin{proof} 
We prove first that $\mathsf{G}_{P}(S)$ is also an \'etale groupoid with respect to the patch topology.
By Lemma~\ref{le: patch}, it is clear that the inversion map is continuous with respect to the patch topology.

We now show that the multiplication map is continuous with respect to the patch topology.
We use the same idea as in the proof of Proposition~4.3 of \cite{L}.
We prove that $m^{-1} (X_{s;t})$ is open.
Let $U$ and $V$ be prime filters such that $\mathbf{d}(U) = \mathbf{r}(V)$
where $U \cdot V = X \in X_{s;t}$.
Then $s \in U\cdot V$ and so we may find $u \in U$ and $v \in V$ such that $uv \leq s$.
In addition, we may assume that $u^{-1}u = vv^{-1}$ since $U^{-1} \cdot U = V \cdot V^{-1}$.
Consider now $X_{u;tt^{-1}u}$ and $X_{v;vt^{-1}t}$.
Their product is obtained by part (4) of Lemma~\ref{le: patch} and is
$$X_{uv; uvt^{-1}t \vee tt^{-1}uv \vee tt^{-1}uvt^{-1}t}.$$
Let $P$ be an element of this set.
Then $uv \in P$ and so $s \in P$.
Suppose that $t \in P$.
Then $(uv)t^{-1}t \in PP^{-1}P = P$, which is a contradiction.
Thus $P \in X_{s;t}$.
It follows that
$$X_{uv; uvt^{-1}t \vee tt^{-1}uv \vee tt^{-1}uvt^{-1}t} \subseteq X_{s;t}.$$
We claim that $U \in Y_{u;tt^{-1}u}$ and $V \in Y_{v;vt^{-1}t}$.
Suppose that $tt^{-1}u \in U$.
Then $tt^{-1}uv \in X$.
But $tt^{-1}uv \leq tt^{-1}s = t$ which implies that $t \in X$, which is a contradiction.
Suppose that $vt^{-1}t \in V$.
Then $uvt^{-1}t \in X$ and $uvt^{-1}t \leq st^{-1}t = t$, which is a contradiction.

We now show that the groupoid with the patch topology is \'etale.
First we show that the set of identities is an open subspace with respect to the patch topology.
Let $P$ be an identity in the groupoid $\mathsf{G}_{P}(S)$. 
Then it is an inverse subsemigroup and so contains an idempotent $e$, say.
Then $P \in X_{e}$ which consists entirely of idempotent prime filters.
Next we have by part (4) of Lemma~\ref{le: patch} that the product of open sets is open.

To show that the groupoid with the patch topology is Boolean it is enough
to show that the set of identities of $\mathsf{G}_{P}(S)$
with respect to the subspace topology, is a Boolean space.
The case where the distributive lattice has a top element is dealt with in Proposition~II.4.5 of \cite{J}.
We deal with the general case here.
The idea for this proof goes back to Section~4.3 of Paterson \cite{P}.
Observe first that we can restrict our attention to the distributive lattice $E(S)$ 
since idempotent filters are determined by the idempotents they contain.
We give the set $\mathbf{2}^{E(S)}$ the discrete topology.
By the Axiom of Choice this is a compact space.
A subbase for this topology is given by sets of the form $U_{e}$ and $U_{e}^{c}$ where $e \in E(S)$ and
$$U_{e} = \{\theta \colon E(S) \rightarrow \mathbf{2} \colon \theta (e)(1) = 1\}
\mbox{ and }
U_{e}^{e} = \{\theta \colon E(S) \rightarrow \mathbf{2} \colon \theta (e)(1) = 1\}.$$
With each filter, we may associate an element $j(X)$ of $\mathbf{2}^{E(S)}$.
This is an injective function.
If we restrict our attention to the prime filters on $E(S)$ then the function $j$ restricts to a bijection with a closed subset of $\mathbf{2}^{E(S)}$.
We denote this closed subset by $\mathbf{P}$.
The simple argument to prove that it really is closed is made explicit in \cite{Law3}.
The restriction of the product topology to $\mathbf{P}$ is Hausdorff and has a basis of compact-open sets.
The topology generated by $\Pi$ restricted to the prime filters on $E(S)$ is easily seen to be 
homeomorphic to the restriction of the product topology to the set $\mathbf{P}$.
Thus we have shown that the space $\mathsf{G}_{P}(S)$ equipped with the patch topology is Boolean.
\end{proof}

We may now complete what was started in part (2) of Lemma~\ref{le: patch}.

\begin{proposition}\label{prop: patch_equal_usual} Let $S$ be a distributive inverse semigroup.
Then the patch topology is the same as the usual topology if and only if $S$ is Boolean.
\end{proposition}
\begin{proof} Only one direction needs proving.
Suppose that the patch topology is the same as the usual topology.
Let $t \leq s$.
Then $X_{s;t}$ is a compact-open set and so can be written as a finite union $\bigcup_{i=1}^{m} X_{a_{i}}$.
Clearly $X_{a_{i}} \subseteq X_{s}$.
It follows that $a_{i} \leq s$ for all $i$ and so the set $\{a_{1}, \ldots, a_{m} \}$ is compatible and so has join; $a$, say.
Hence $X_{s;t} = X_{a}$.
We have used Lemma~\ref{le: prime_topology} throughout.
We prove that $a = s \setminus t$.
Suppose that $0 \neq z \leq t,a$.
Let $P$ be a prime filter containing $z$.
Then $P$ contains $t$ and $a$.
But a prime filter containing $a$ cannot contain $t$
and so we get a contradiction.
It follows that $z = 0$ and that $t \wedge a = 0$.
Let $P$ be any prime filter containing $s$.
If $t \notin P$ then $a \in P$.
We have shown that $X_{s} \subseteq X_{a} \cup X_{t}$.
We clearly have equality and so again by Lemma~\ref{le: prime_topology}
we have that $s = a \vee t$.
We have therefore proved that $S$ is Boolean.
\end{proof}

If $S$ is a distributive inverse semigroup then the groupoid $\mathsf{G}_{P}(S)^{\dagger}$ is called the 
{\em Booleanization of the groupoid $\mathsf{G}_{P}(S)$}.
We may therefore form the Boolean inverse semigroup $\mathbf{B}(S) = \mathsf{KB}(\mathsf{G}_{P}(S)^{\dagger})$
which we call the {\em Booleanization of $S$}.
There is a map $\beta \colon S \rightarrow \mathbf{B}(S)$ given by $s \mapsto X_{s}$
which is an injective homomorphism.
We shall characterize the properties of this homomorphism.
First we need a lemma.

\begin{lemma}\label{le: products_complements} Let $S$ be a Boolean inverse semigroup.
Suppose that $t \leq s$ and $v \leq u$.
Then 
$(s \setminus t)(u \setminus v) = su \setminus (sv \vee tu \vee tv )$.
\end{lemma}
\begin{proof}
We have that
$$su = (s \setminus t)(u \setminus v)
\vee
tv
\vee
t(u \setminus v)
\vee
(s \setminus t)v.$$
But
$$tu = tv \vee t(u \setminus v) \mbox{ and } sv = tv \vee (s \setminus t)v.$$
Thus
$$su = (s \setminus t)(u \setminus v)
\vee
tv
\vee
sv
\vee
tu.$$
It remains to show that
$(s \setminus t)(u \setminus v)$
and
$tv \vee sv \vee tu$
have only zero as a common lower bound.
But this follows from the fact that 
$tv$, 
$t(u \setminus v)$   
and
$(s \setminus t)v$ are pairwise orthogonal.
\end{proof}

The main result of this section is the following.

\begin{theorem}[Booleanization of distributive inverse semigroups]\label{the: Booleanization} 
Let $S$ be a distributive inverse semigroup and let $\theta \colon S \rightarrow T$ be a morphism to a Boolean inverse semigroup
with the property that the inverse image under $\theta$ of each prime filter in $T$ is a prime filter in $S$.
Then there is a unique morphism $\bar{\theta} \colon \mathbf{B}(S) \rightarrow T$ of distributive inverse semigroups 
such that 
$\bar{\theta} \beta = \theta$.
\end{theorem}
\begin{proof} 
We shall define the map $\bar{\theta}$ is stages.
Suppose first that $t \leq s$.
Put
$$\bar{\theta}(X_{s;t}) = \theta (s) \backslash \theta (t).$$
Observe that the righthandside is well-defined since $t \leq s$ implies that $\theta (t) \leq \theta (s)$.
To show this is well-defined, suppose that 
$X_{s;t} = X_{u;v}$ and that $\theta (s) \setminus \theta (t) \neq \theta (u) \setminus \theta (v)$.
Then by Proposition~\ref{prop: key_result} and Proposition~\ref{prop: semigroup_filters}
we may find an ultrafilter  $P$ in $T$ which contains, without loss of generality, 
$\theta (s) \setminus \theta (t)$ but not $\theta (u) \setminus \theta (v)$.
By assumption $Q = \theta^{-1}(P)$ is a prime filter in $S$.
Clearly $s \in Q$ and $t \notin Q$ and so $Q \in X_{s;t}$.
By assumption $Q \in X_{u;v}$.
But this implies that $\theta (u) \in P$ and $\theta (v) \notin P$ 
and so $\theta (u) \setminus \theta (v) \in P$
by Lemma~\ref{le: relative_complement} and the fact that $P$ is an ultrafilter and so a prime filter,
and this is a contradiction.
Thus this map is well-defined.
The fact that $\bar{\theta}(X_{s;t}) \bar{\theta}(X_{u;v}) = \bar{\theta}(X_{su;su \vee tu \vee tv})$
follows from Lemma~\ref{le: products_complements} and the fact that $\theta$ is a morphism of
distributive inverse semigroups.

We next prove that if $X_{s;t}$ and $X_{u;v}$ are compatible in $\mathbf{B}(S)$ then
$\theta (s)\setminus \theta (t)$ is compatible with $\theta (u) \setminus \theta (v)$.
By Lemma~\ref{le: products_complements} we have that
$$(\theta (s)\setminus \theta (t))^{-1}(\theta (u) \setminus \theta (v)) 
= 
\theta (s)^{-1}\theta (u) \setminus \theta (s)^{-1}\theta (v) \vee \theta (t)^{-1}\theta (u) \vee \theta (t)^{-1}\theta (v).$$
Let $P$ be a prime filter containing $(\theta (s)\setminus \theta (t))^{-1}(\theta (u) \setminus \theta (v))$.
Then $\theta^{-1}(P)$ is a prime filter that contains $s^{-1}u$ and omits $s^{-1}v \vee t^{-1}u \vee t^{-1}v$.
Thus $\phi^{-1}(P) \in X_{s;t}^{-1}X_{u;v}$.
By assumption, this contains only idempotent filters.
Thus $\phi^{-1}(P)$ is an idempotent filter and so $P$ is an idempotent prime filter.
Thus all the prime filters containing $(\theta (s)\setminus \theta (t))^{-1}(\theta (u) \setminus \theta (v))$ are idempotent.
By Lemma~\ref{le: Xlemma},
it follows that $(\theta (s)\setminus \theta (t))^{-1}(\theta (u) \setminus \theta (v))$ is an idempotent.
By symmetry, we deduce that $\theta (s)\setminus \theta (t)$ and $\theta (u) \setminus \theta (v)$ are compatible.

It remains to extend the map $\bar{\theta}$ to the whole of $\mathbf{B}(S)$.
Every compact-open local bisection in $\mathbf{B}(S)$
is of the form
$\bigcup_{i=1}^{m} Y_{s_{i} ; t_{i}}$ 
for a finite compatible set of basic elements.
Put
$$\bar{\theta}(\bigcup_{i=1}^{m} X_{s_{i} ; t_{i}}) = \bigvee_{i=1}^{m} \theta (s_{i}) \setminus \theta (t_{i}).$$
We show that this map is well-defined.
Suppose that
$$\bigcup_{i=1}^{m} X_{s_{i} ; t_{i}}
=
\bigcup_{j=1}^{n} X_{u_{j} ; v_{j}}.$$
We need to prove that 
$$\bigvee_{i=1}^{m} \theta (s_{i}) \setminus \theta (t_{i})
=
\bigvee_{j=1}^{n} \theta (u_{j}) \setminus \theta (v_{j}).$$
It is enough to show that the set of prime filters containing the
lefthandside is the same as the set of prime filters containing the righthandside.
This is straightforward to prove given our assumption on $\theta$.
By construction, the map $\bar{\theta}$ is a morphism of distributive inverse semigroups.

It is clear that $\bar{\theta} \beta = \theta$.
It only remains to prove uniqueness.
Let $\psi \colon \mathbf{B}(S) \rightarrow T$ be a morphism such that
$\psi \beta = \theta$.
It follows that $\psi (\beta (s)) = \theta (s)$.
Thus $\psi (X_{s;0}) = \theta (s)$.
Let $t \leq s$.
Then $\theta (t) \leq \theta (s)$ and so $\theta (s) \backslash \theta (t)$ is defined.
We shall prove that $\psi (X_{s;t}) = \theta (s) \backslash \theta (t)$.
Since $\psi$ is a morphism of Boolean inverse semigroups we have that
$\psi (X_{s} \setminus X_{t}) = \theta (s) \setminus \theta (t)$ and clearly $X_{s} \setminus X_{t} = X_{s ; t}$.
\end{proof}

%%%%%%%%%%%%%%%%%%%%%%%%%%%%%%%%%%%%%%%%%%%%%%%%%%%%%%%%%%%%%%%%%%%%%%%%%%%%%%%%%%%%%%%%%%%%%%%%%%%%%%%%%%%%%%%%%%%%%%%%%%%%%%%%%%%%%%%%%%%
\subsection{Paterson's universal groupoid}

The goal of this section is to apply Theorem~\ref{the: Booleanization} to help us understand
Paterson's universal groupoid \cite{L,LMS,P}.

Let $S$ be an arbitrary inverse semigroup with zero.
We denote by $\mathsf{G}(S)$ the set of all proper filters of $S$.
This is a groupoid when we define $A \cdot B = (AB)^{\uparrow}$ if $\mathbf{d}(A) = \mathbf{r}(B)$ \cite{LMS}.
For each $s \in S$ define $U_{s}$ to be the set of all proper filters that contain $s$.
If $s \neq 0$ then $s^{\uparrow}$ is a proper filter containing $s$, it follows that $U_{s} \neq \emptyset$ if and only if $s \neq 0$.
Put $\tau = \{U_{s} \colon  s \in S \}$.
The proof of the following lemma is similar to the proofs from part (1) of Lemma~\ref{le: Xlemma} 
and parts (2), (5) and (8) of Lemma~\ref{le: prime_topology}.

\begin{lemma} Let $S$ be an inverse semigroup with zero.
\begin{enumerate}

\item $U_{s}$ is a local bisection.

\item $U_{s^{-1}} = U_{s}^{-1}$.

\item $U_{s}U_{t} = U_{st}$.

\item $U_{s} \cap U_{t} = \bigcup_{a \leq s,t} U_{a}$. 
\end{enumerate}
\end{lemma}

It follows that $\tau$ is the basis for a topology on $\mathsf{G}(S)$.
The proof of the following is similar to that of Proposition~\ref{prop: spectral_groupoid}.
 
\begin{lemma}  Let $S$ be an inverse semigroup with zero.
Then  $\mathsf{G}(S)$ is an \'etale topological groupoid.
\end{lemma}

The following is a key step and makes substantive use of Theorem~\ref{the: distributive_completions}.

\begin{proposition}\label{prop: filter_prime_filter} Let $S$ be an inverse semigroup with zero.
Then the groupoid $\mathsf{G}(S)$ is homeomorphic to the groupoid $\mathsf{G}_{P}( \mathsf{D}(S))$.
\end{proposition}
\begin{proof} The idea is very simple: proper filters in $S$ correspond bijectively prime filters in $\mathsf{D}(S)$.
We shall also use the fact that an element $\{a_{1}, \ldots, a_{m}\}^{\downarrow}$ of $\mathsf{D}(S)$ may also be written
as $\bigvee_{i=1}^{m} a_{i}^{\downarrow}$.
Let $F$ be a proper filter in $S$.
Define 
$$F^{u} = \{ A \in \mathsf{D}(S) \colon A \cap F \neq \emptyset \}.$$ 
Then $F^{u}$ is a prime filter in $\mathsf{D}(S)$.
Let $P$ be a prime filter in $\mathsf{D}(S)$.
Define
$$P^{d} = \{s \in S \colon s^{\downarrow} \in P \}.$$
The fact that $P^{d}$ is non-empty uses the fact that $P$ is a prime filter
as does the fact that $P^{d}$ is down-directed.
Then $P^{d}$ is a filter.
It is routine to check that these maps are mutually inverse sand set up a bijection between
$\mathsf{G}(S)$ and $\mathsf{G}_{P}( \mathsf{D}(S))$.
It is routine to check that for any proper filter $F$ we have that $(F^{-1} \cdot F)^{u} = (F^{u})^{-1} \cdot F^{u}$.
It follows that if $F \cdot G$ is defined then $F^{u} \cdot G^{u}$ is also defined.
It is routine to check that $(F \cdot G)^{u} = F^{u} \cdot G^{u}$.
It follows that $F \mapsto F^{u}$ is a bijective functor.
Finally, The basic open set $U_{s}$ is mapped to the basic open set $X_{s^{\downarrow}}$.
Also, the basic open set $X_{A}$ where $A = \{a_{1}, \ldots, a_{m}\}^{\downarrow}$ can be written
$X_{A} = \bigcup_{i=1}^{m} X_{a_{i}^{\downarrow}}$.
But the inverse image of any basic open set of the form $X_{a^{\downarrow}}$ is $U_{a}$.
It follows that our bijective is continuous and open.
\end{proof}

We now describe Paterson's universal groupoid $\mathsf{G}_{u}(S)$ of the inverse semigroup $S$.
The underlying groupoid is still $\mathsf{G}(S)$ but a different topology is defined.
For $x$, $x_1,\ldots, x_n \in S$ with $x_1,\ldots, x_n \leq x$, 
the set $U_{x; x_1, \ldots, x_n} $ is defined by 
$$ U_{x; x_1, \ldots, x_n} = U_x \cap U_{x_1}^c \cap\ldots \cap U_{x_n}^c$$
where $U_{x}^c$ is the complement of $U_x$ in the groupoid.
Let $\Omega$ be the set of all such subsets.
With respect to this topology, 
the groupoid is called the {\em universal groupoid} and is denoted by means of $\mathsf{G}_{u}(S)$ where `u' stands for `universal'.
It is proved in \cite{L}, in our terminology, that $\mathsf{G}_{u}(S)$ is Boolean.
The main result of this section is the following.

\begin{theorem} Let $S$ be an inverse semigroup with zero.
The universal groupoid $\mathsf{G}_{u}(S)$ is homeomorphic to $\mathsf{G}_{P}(\mathsf{D}(S))^{\dagger}$.
\end{theorem}
\begin{proof} We shall use the same groupoid isomorphism as in Proposition~\ref{prop: filter_prime_filter}.
The basic open set $U_{s;s_{1}, \ldots, s_{m}}$ is mapped to the basic open set $X_{s^{\downarrow};\{s_{1}, \ldots, s_{m}\}^{\downarrow}}$.
To go in the other direction, we use induction on part (3) of Lemma~\ref{le: patch} to reduce
to the case where we need only consider basic open sets of the form $X_{s^{\downarrow};\{s_{1}, \ldots, s_{m}\}^{\downarrow}}$
whose inverse image is the basic open set $U_{s;s_{1}, \ldots, s_{m}}$. 
\end{proof}

With each inverse semigroup $S$, we may associate a Boolean inverse semigroup $\mathbf{BS}(S) = \mathbf{B}(\mathsf{D}(S))$
togther with an injective homomorphism 
$\gamma \colon S \rightarrow \mathbf{BS}(S)$ which takes $s$ to $X_{s^{\downarrow}}$.
Observe that $\gamma (s) = \beta \iota (s)$ where $\beta \colon \mathsf{D}(S) \rightarrow \mathbf{B}(\mathsf{D}(S))$.

\begin{lemma} Let $S$ be an inverse semigroup.
Then the map $\gamma \colon S \rightarrow \mathbf{BS}(S)$ pulls prime filters back to filters.
\end{lemma}
\begin{proof} Let $P$ be a prime filter.
Using primality, we deduce immediately that $P$ contains elements of the form $X_{s^{\downarrow}}$ and
so the inverse image of $P$ is non-empty.
Let $s$ and $t$ be in the inverse image of $P$.
Since $P$ is downwards directed we may find $X_{a^{\downarrow};\{a_{1}, \ldots, a_{m} \}^{\downarrow}}$ in $P$ and below
both $X_{s^{\downarrow}}$ and $X_{t^{\downarrow}}$.
Clearly $X_{a^{\downarrow}}$ also belongs to $P$.
We know that $(a^{\uparrow})^{u}$ is a prime filter.
It contains $a^{\downarrow}$ and omits $\{a_{1}, \ldots, a_{m}\}^{\downarrow}$ since $a_{1}, \ldots, a_{m} < a$.
Thus $(a^{\uparrow})^{u} \in X_{s^{\downarrow}},X_{t^{\downarrow}}$.
It follows that $a \leq s,t$ and so we have proved that the inverse image of $P$ is downwards directed.
It is clearly upwardly closed.
\end{proof}

The theorem below was inspired by the calculations on pp~190--191 of Paterson's book \cite{P}.

\begin{theorem}[Booleanization for inverse semigroups]
Let $S$ be an inverse semigroup and let $\theta \colon S \rightarrow T$ be a homomorphism to a Boolean inverse semigroup
with the property that the inverse image under $\theta$ of each prime filter in $T$ is a filter in $S$.
Then there is a unique homomorphism of distributive inverse semigroups 
$\bar{\psi} \colon \mathbf{BS}(S) \rightarrow T$ such that $\bar{\theta} \gamma = \theta$.
\end{theorem}
\begin{proof} The semigroup $T$ is, in particular, distributive, and so there is a morphism
$\psi \colon \mathsf{D}(S) \rightarrow T$ such that $\psi \iota = \theta$.
We need to prove that $\psi$ pulls prime filters in $T$ back to prime filters in $\mathsf{D}(S)$.
Let $P$ be a prime filter in $T$.
We shall prove that $\psi^{-1}(P)$ is a prime filter in $\mathsf{D}(S)$.
Observe first that $\psi^{-1}(P)$ is non-empty because $\theta^{-1}(P)$ is a filter in $S$.
Let $A,B \in \psi^{-1}(P)$ where $A = \{a_{1}, \ldots, a_{m}\}^{\downarrow}$ and $B = \{b_{1}, \ldots, b_{n}\}^{\downarrow}$.
Then since $P$ is a prime ideal we have that $\theta (a_{i}), \theta (b_{j}) \in P$ for some $i$ and some $j$.
It follows that $a_{i}^{\downarrow},b_{j}^{\downarrow} \in \psi^{-1}(P)$.
Now $a_{i},b_{j} \in \theta^{-1}(P)$.
Thus by assumption, there is an element $c \in S$ such that $c \leq a_{i},b_{j}$
where $\theta (c) \in P$.
It follows that $c^{\downarrow} \leq A,B$ and so $\psi^{-1}(P)$ is a filter
and it is easy to check that it is a prime filter.

By Theorem~\ref{the: Booleanization}, 
there exists $\bar{\psi} \colon \mathbf{B}(\mathsf{D}(S)) \rightarrow T$ such that $\bar{\psi} \beta \iota = \theta$.
Thus $\bar{\psi} \gamma = \theta$.

It remains to prove uniqueness.
Let $\phi \colon \mathbf{BS}(S) \rightarrow T$ be such that $\phi \gamma = \theta$.
We use the same notation as above.
The map $\psi$ is defined and it is straightforward to check that $\phi \beta = \psi$.
We can now use the uniqueness guaranteed by  
Theorem~\ref{the: Booleanization}.
\end{proof}

%%%%%%%%%%%%%%%%%%%%%%%%%%%%%%%%%%%%%%%%%%%%%%%%%%%%%%%%%%%%%%%%%%%%%%%%%%%%%%%%%%%%%%%%%%%%%%%%%%%%%%%%%%%%%%%%%%%%%%%%%%%%%%%%%%%%%%%%%%%%%
\section{Tight completions and Exel's tight groupoid}

The goal of this section is to develop the theory of what we call tight completions
and apply it to providing a conceptual understanding of Exel's tight groupoid.

 %%%%%%%%%%%%%%%%%%%%%%%%%%%%%%%%%%%%%%%%%%%%%%%%%%%%%%%%%%%%%%%%%%%%%%%%%%%%%%%%%%%%%%%%%%%%%%%%%%%%%%%%%%%%%%%%%%%%%%%%%%%%%%%%%%%
\subsection{Coverages}

A {\em coverage} $\mathcal{C}$ on an inverse semigroup $S$ is defined by the following data.
For each $a \in S$, there is a set $\mathcal{C}(a)$ of subsets of $a^{\downarrow}$, whose elements are called {\em coverings},
satisfying the following axioms: 
\begin{description}

\item[{\rm (R)}] $\{a \} \in \mathcal{C}(a)$ for all $a \in S$.

\item[{\rm (I)}] If $X \in \mathcal{C}(a)$ then $X^{-1} \in \mathcal{C}(a^{-1})$.

\item[{\rm (MS)}] $X \in \mathcal{C}(a)$ and $Y \in \mathcal{C}(b)$ imply that $XY \in \mathcal{C}(ab)$.

\item[{\rm (T)}] If $X \in \mathcal{C}(a)$ and  $X_{i} \in \mathcal{C}(x_{i})$ for each $x_{i} \in X$ then $\bigcup_{i} X_{i} \in \mathcal{C}(a)$.

\end{description}

The intuitive idea behind the definition of a coverage is that it axiomatizes the notion of join.
Thus $X \in \mathcal{C}(a)$ should be regarded as saying that {\em morally} the join of $X$ is $a$.

A proper filter $A$ on $S$ is called a {\em $\mathcal{C}$-filter} if $x \in A$ and $X \in \mathcal{C}(x)$ then $y \in A$ for some $y \in X$.
We shall use the word {\em family} to describe the set of all $\mathcal{C}$-filters for a given coverage $\mathcal{C}$.

Throughout the remainder of this section, $\mathcal{C}$ will be a coverage.
If $X \in \mathcal{C}(a)$ define $\mathbf{d}(X) = \{x^{-1}x \colon x \in X \}$.

\begin{lemma}\label{le: coverage_properties} Let $\mathcal{C}$ be a coverage on $S$.
\begin{enumerate}

\item $X \in \mathcal{C}(a)$ implies that $\mathbf{d}(X) \in \mathcal{C}(a^{-1}a)$.

\item If $X \subseteq a^{\downarrow}$ then $X \in \mathcal{C}(a)$ if and only if $\mathbf{d}(X) \in \mathcal{C}(a^{-1}a)$.

\item Let $X,Y \in \mathcal{C}(a)$ and define $X \wedge Y = \{x \wedge y \colon x \in X, y \in Y\}$.
Then $X \wedge Y \in \mathcal{C}(a)$ and $X \wedge Y = X\mathbf{d}(Y) = Y\mathbf{d}(X)$.

\item If $X \in \mathcal{C}(b)$ and $X,Y \in \mathcal{C}(a)$ then $X \wedge Y \in \mathcal{C}(b)$.
 
\end{enumerate}
\end{lemma}
\begin{proof} (1) We have that $x \in X$ implies that $x \leq a$ and so $x = ax^{-1}x$.
Thus $a^{-1}x = x^{-1}x$.
By (R) and (I), we have that $\{a^{-1} \} \in \mathcal{C}(a^{-1})$ and so by (MS) we have that $a^{-1}X \in \mathcal{C}(a^{-1}a)$
but $a^{-1}X = \{x^{-1}x \colon x \in X   \}$ and the claim follows.

(2) By (1), only one direction needs proving. Suppose that $X \subseteq a^{\downarrow}$ and  $\mathbf{d}(X) \in \mathcal{C}(a^{-1}a)$.
Then by (MS), we have that $a\mathbf{d}(X) \in \mathcal{C}(a)$.
But $X = a\mathbf{d}(X)$ and the result follows.

(3) Observe first that since $x,y \leq a$ the meet $x \wedge y$ is defined.
Since $x$ and $y$ are compatible, $x \wedge y = xy^{-1}y = yx^{-1}x$.
Thus $X \wedge Y = X\{y^{-1}y \colon y \in Y\}$ which belongs to $\mathcal{C}(a)$ by (1) and (MS).

(4) It remains to show that $X \wedge Y \in \mathcal{C}(b)$.
For each $x \in X$ we have that $x \leq a$ and so $x = xa^{-1}a$.
Thus for each $x \in X$ we have that $x\mathbf{d}(Y) \in \mathcal{C}(x)$.
But $X \in \mathcal{C}(b)$ and so by (T), we have that $X\mathbf{d}(Y) \in \mathcal{C}(b)$.
\end{proof}

If $\mathcal{C}$ is a coverage on a semigroup $S$ such that 
$\mathcal{C}(a) \cap \mathcal{C}(b) \neq \emptyset$ implies $a = b$
then
we say that $S$ is {\em separative (with respect to the coverage $\mathcal{C}$)}
and that the coverage is {\em separated}.
The coverage $\mathcal{C}$ is said to be {\em idempotent-pure} if $X \in \mathcal{C}(a)$
and $X \subseteq E(S)$ implies that $a \in E(S)$.

\begin{lemma} 
A separated coverage is idempotent-pure.
\end{lemma}
\begin{proof} Let $E \in \mathcal{C}(x)$ where $E \subseteq E(S)$.
By (R) and (MS) we have that $Ex^{-1} \in \mathcal{C}(xx^{-1})$.
But if $e \in E$ then $e \leq x$ and so $e = ex = xe$.
Thus $ex^{-1} = exx^{-1} = e$.
It follows that $Ex^{-1} = E$ and so by the separated assumption $x = xx^{-1}$.
\end{proof}

Let $S$ be an inverse semigroup equipped with a coverage $\mathcal{C}$.
Define the relation $\equiv$ on $S$ by
$$a \equiv b \Leftrightarrow \mathcal{C}(a) \cap \mathcal{C}(b) \neq \emptyset.$$

\begin{lemma} 
The relation $\equiv$ is a congruence on $S$.
\end{lemma}
\begin{proof} We show first that $\equiv$ is an equivalence relation.
We have that $\{a \} \in \mathcal{C}(a)$ and so $a \equiv a$.
It is immediate that $a \equiv b$ implies that $b \equiv a$.
Suppose that $a \equiv b$ and $b \equiv c$.
Let $X \in \mathcal{C}(a) \cap \mathcal{C}(b)$ and $Y \in \mathcal{C}(b) \cap \mathcal{C}(c)$.
By part (4) of Lemma~\ref{le: coverage_properties}, 
we have that $X \wedge Y \in \mathcal{C}(a) \cap \mathcal{C}(c)$
and so $a \equiv c$.
Thus $\equiv$ is an equivalence relation and it is a congruence by (MS).
\end{proof}

We denote by $\mathbf{S}$ the quotient of $S$ by $\equiv$,
and the $\equiv$-congruence class containing $s$ by $\mathbf{s}$.
There is a homomorphism $\sigma \colon S \rightarrow \mathbf{S}$ given by $s \mapsto \mathbf{s}$.

Our goal now is to show that the set of all $\mathcal{C}$-filters forms an \'etale groupoid.

\begin{lemma}\label{le: C-filters} Let $A$ be a filter on the inverse semigroup $S$.
\begin{enumerate}

\item $A$ is a $\mathcal{C}$-filter if and only if $A^{-1}$ is a $\mathcal{C}$-filter.

\item $A$ is a $\mathcal{C}$-filter if and only if $A^{-1} \cdot A$ is a $\mathcal{C}$-filter.

\end{enumerate}
\end{lemma}
\begin{proof} (1) Suppose that $A$ is a $\mathcal{C}$-filter.
Let $x \in A^{-1}$ where $X \in \mathcal{C}(x)$.
Then $x^{-1} \in A$ and $X^{-1} \in \mathcal{C}(x^{-1})$ by axiom (I).
By assumption, there exists $y \in X^{-1}$ such that $y \in A$.
But then $y^{-1} \in A^{-1}$ where $y^{-1} \in X$, as required.
 
(2) Suppose that $A$ is a $\mathcal{C}$-filter.
Let $x \in A^{-1} \cdot A$ where $X \in \mathcal{C}(x)$.
Then $a^{-1}b \leq x$ where $a,b \in A$.
It follows that $ax \in A$ where $aX \in \mathcal{C}(ax)$ by axiom (MS).
By assumption, $ay \in A$ for some $y \in X$.
Thus $a^{-1}ay \in A^{-1} \cdot A$ and so $y \in  A^{-1} \cdot A$, as required.

Suppose now that $A^{-1} \cdot A$ is a $\mathcal{C}$-filter.
Let $x \in A$ where $X \in \mathcal{C}(x)$.
Then $x^{-1}x \in A^{-1} \cdot A$ where $x^{-1}X \in \mathcal{C}(x^{-1}X)$.
By assumption, $x^{-1}y \in A^{-1} \cdot A$ where $x^{-1}X \in \mathcal{C}(x^{-1}x)$.
Thus $x^{-1}y \in A^{-1} \cdot A$ for some $y \in A$.
Thus $x^{-1}xy \in A$ and so $y \in A$, as required.
\end{proof}

\begin{lemma}\label{le: product_of_C_filters} 
If $A$ and $B$ are $\mathcal{C}$-filters and if $A \cdot B$ exists then $A \cdot B$ is a $\mathcal{C}$-filter.
\end{lemma}
\begin{proof} We have already seen that $A \cdot B$ is a filter and 
$\mathbf{d}(A \cdot B) = \mathbf{d}(B)$.
Thus the result follows from the lemma above.
\end{proof}

It follows that we may define the groupoid $\mathsf{G}_{\mathcal{C}}(S)$ of $\mathcal{C}$-filters of $S$.

For each $s \in S$, define $Z_{s}$ to be the set of all $\mathcal{C}$-filters that contain $s$.
Define $\xi$ to be the set of all such sets.

\begin{lemma}\label{le: C-bisections} \mbox{}
\begin{enumerate}

 \item $Z_{s}$ is a bisection.

\item $Z_{s}^{-1} = Z_{s^{-1}}$.

\item $Z_{s}Z_{t} = Z_{st}$.

\item $Z_{s} \cap Z_{t}$ is a union of elements of $\xi$.

\end{enumerate}
\end{lemma}
\begin{proof} (1) This follows by part (3) of Lemma~\ref{le: filter_properties}.

(2) This follows by part (1) of Lemma~\ref{le: C-filters}.

(3) The inclusion $Z_{s}Z_{t} \subseteq Z_{st}$ follows by Lemma~\ref{le: product_of_C_filters} 
and the reverse inclusion uses the same argument as part (5) of Lemma~\ref{le: prime_topology}
combined with Lemma~\ref{le: product_of_C_filters}. 

(4) Let $A \in Z_{s} \cap Z_{t}$.
Then $s,t \in A$.
Since $A$ is a filter there exists $a \in A$ such that $a \leq s,t$.
Observe that $Z_{a} \subseteq Z_{s} \cap Z_{t}$ and that $A \in Z_{a}$. 
\end{proof}

It follows that $\xi$ is a basis for a topology on  $\mathsf{G}_{\mathcal{C}}(S)$.
The proof of the next result follows the same lines as the main part of the proof
of Proposition~\ref{prop: spectral_groupoid} using Lemma~\ref{le: C-bisections}.

\begin{proposition} For each coverage $\mathcal{C}$ on the inverse semigroup $S$,  
the groupoid $\mathsf{G}_{\mathcal{C}}(S)$ is an \'etale topological groupoid.
\end{proposition}

We now show how these definitions unify what we have discussed so far.

\begin{examples} {\em Let $S$ be an inverse semigroup.
\begin{enumerate}

\item The simplest coverage is defined by putting $\mathcal{C}(x) =\{ \{x \} \}$.
We call this the {\em trivial coverage}.
The $\mathcal{C}$-filters are just the {\em filters}.

\item Let $S$ be a distributive inverse semigroup. Define $\mathcal{C}(x)$ to be those finite subsets of $x^{\downarrow}$
whose joins are $x$. This defines a coverage.
The $\mathcal{C}$-filters are just the {\em prime filters}.

\item Let $S$ be a pseudogroup.  Define $\mathcal{C}(x)$ to be those subsets of $x^{\downarrow}$
whose joins are $x$. This defines a coverage.
The $\mathcal{C}$-filters are just the {\em completely prime filters}.

\end{enumerate}
}
\end{examples}

%%%%%%%%%%%%%%%%%%%%%%%%%%%%%%%%%%%%%%%%%%%%%%%%%%%%%%%%%%%%%%%%%%%%%%%%%%%%%%%%%%%%%%%%%%%%%%%%%%%%%%%%%%%%%%%%%%%%%%%%%%%%%%%%%%%%%%%%%%%%%%%%%%%%%
\subsection{Tight completions}

Theorem~\ref{the: distributive_completions} tells us that every inverse semigroup can be completed to a distributive inverse semigroup.
In this section, we shall endow every inverse semigroup with additional structure called a {\em tight coverage}.
We shall then prove that inverse semigroups equipped with tight coverages can be completed to distributive inverse semigroups 
in a way that takes account of the coverage.
This result has important applications in the the way that inverse semigroup theory is used in constructing $C^{\ast}$-algebras.

Let $S$ be an inverse semigroup.
To define our coverage we need some notation. 
Let $a \in S$ and $B \subseteq S$.
Define $a \rightarrow B$ to mean that for each $0 \neq x \leq a$ there exists $b \in B$ such that $x^{\downarrow} \cap b^{\downarrow} \neq 0$.
We call this the {\em arrow relation} and it was first defined in \cite{L}.
For each $a \in S$ define $\mathcal{T}(a)$ to consist of those {\em finite} subsets
$B \subseteq a^{\downarrow}$ such that $a \rightarrow B$.

\begin{lemma} 
With the above definition $\mathcal{T}$ defines a coverage on $S$. 
\end{lemma}
\begin{proof}
It is immediate that (R) and (I) hold.
Suppose that $X \in \mathcal{T}(a)$ and $Y \in \mathcal{T}(b)$.
Then since $X \subseteq a^{\downarrow}$ and $Y \subseteq b^{\downarrow}$ we have that $XY \subseteq (ab)^{\downarrow}$.
Let $0 \neq z \leq ab$.
Then $aa^{-1}z = z$ and so $a^{-1}z \neq 0$.
It follows that $0 \neq a^{-1}z \leq a^{-1}ab \leq b$.
Thus there exists $y \in Y$ and a $u$ such that $u \leq y, a^{-1}z$.
Observe that $a^{-1}au = u$ and so $au \neq 0$.
Thus $0 \neq au \leq ay, z$.
We now carry out a similar calculation starting from $au \leq ay$.
Then $auy^{-1} \leq a$ and so there exists $x \in X$ and a $v$ such that $v \leq x, auy^{-1}$.
Observe that $vyy^{-1} = v$ and so $vy \neq 0$.
Thus $0 \neq vy \leq xy, auy^{-1}y = au \leq z$.
It follows that (MS) holds.
Finally, we check that (T) holds.
Let $X \in \mathcal{T}(a)$ and suppose that for each $x_{i} \in X$ we have that $X_{i} \in \mathcal{T}(x_{i})$.
We prove that $\bigcup_{i} X_{i} \in \mathcal{T}(x)$.
Let $0 \neq z \leq a$.
Then there exists $0 \neq u \leq z,x_{i}$ for some $x_{i} \in X$.
But $0 \neq u \leq x_{i}$ implies that there exists $y \in X_{i}$ and a $v$ such that
$0 \neq v \leq y,u$.
Thus there exists $0 \neq u \leq z, y$ where $y \in X_{i}$, as required.
\end{proof}

We call $\mathcal{T}$ the {\em tight coverage}, 
an element of $\mathcal{T}(a)$ is called a {\em tight cover of $a$},
and $\mathcal{T}$-filters are called {\em tight filters}.
In this section, 
a {\em separative} semigroup will be one with the property that 
$\mathcal{T}(a) \cap \mathcal{T}(b) \neq \emptyset$
implies that $a = b$.

\begin{lemma}\label{le: ultrafilter_tight} 
Every ultrafilter is a tight filter.
\end{lemma}
\begin{proof} By Lemma~\ref{le: C-filters},
we have that $A$ is a tight filter if and only if $\mathbf{d}(A)$ is an (idempotent) tight filter.

We show next that $F$ is an idempotent tight filter if and only if $E(F)$ is a tight filter in $E(S)$.
Let $F$ be an idempotent tight filter.
Then $E(F)$ is a filter in $E(S)$ and $F = E(F)^{\uparrow}$.
Let $e \in E(F)$ and suppose $X = \{f_{1}, \ldots, f_{m}\}$ is a tight cover of $e$ in $E(S)$.
Then $e \in F$ and $X$ is also a tight cover of $e$ in $S$.
It follows that $f_{i} \in F$ for some $i$ and so $f_{i} \in E(F)$.
We have therefore shown that $E(F)$ is a tight filter.
Conversely, suppose that $E(F)$ is a tight filter in $E(S)$.
Let $a \in F$ and $\{a_{1}, \ldots, a_{m} \}$ be a tight cover of $a$.
Then $\mathbf{d}(X)$ is a tight cover of $a^{-1}a$.
By assumption, $\mathbf{d}(a_{i}) \in E(F)$ for some $i$.
But $a \mathbf{d}(a_{i}) \in F$ and so $a_{i} = a\mathbf{d}(a_{i}) \in F$, as required. 

We therefore need only prove our result in the case where our inverse semigroup is a meet semilattice
which means that we may use Lemma~\ref{le: ultrafilter_result}.
Let $F$ be an ultrafilter in a meet semilattice.
Let $x \in F$ and let $\{a_{1}, \ldots, a_{m} \}$ be a tight cover of $x$.
Suppose that for all $i$ we have that $a_{i} \notin F$.
Then form each $i$ there is $f_{i}$ in $F$ such that $f_{i} \wedge a_{i} = 0$.
Put $f = \bigwedge_{i=1}^{m} a_{i}$.
Then $f \in F$ and we may assume without loss of generality that $f \leq x$.
Clearly $f \neq 0$.
But by assumption, we must have that $f^{\downarrow} \cap a_{i}^{\downarrow} \neq 0$
which leads to a contradition.
It follows that $a_{i} \in F$ for some $i$.
\end{proof}

It follows by the above result that if $s$ is non-zero the set $Z_{s}$ of tight filters containing $s$ is non-empty.

A semigroup homomorphism $\theta \colon S \rightarrow T$ to a distributive inverse semigroup
is said to be a {\em tight map}
if for each element $a \in S$ and $\mathcal{T}$-cover $A = \{a_{1}, \ldots, a_{n} \}$ of $a$ 
we have that $\theta (a) = \bigvee_{i=1}^{n} \theta (a_{i})$.

%The goal of this section is to prove the following theorem.

%\begin{theorem}[Tight completions] Let $S$ be an inverse semigroup.
%\begin{enumerate}

%\item There is a distributive inverse semigroup $\mathsf{D}_{t}(S)$ 
%and a tight map $\delta \colon S \rightarrow \mathsf{D}_{t}(S)$ which is universal for tight maps from $S$
%to distributive inverse semigroups.

%\item There is an order isomorphism between the poset of tight filters in $S$ and the poset of prime
%filters in $\mathsf{D}_{t}(S)$ under which ultrafilters correspond to ultrafilters.

%\item The \'etale groupoids $\mathsf{G}_{\mathcal{T}}(S)$ and $\mathsf{G}_{P}(\mathsf{D}_{t}(S))$ are homeomorphic.

%\end{enumerate}
%\end{theorem}

%We call the distributive inverse semigroup $\mathsf{D}_{t}(S)$ the {\em tight completion} of $S$.

%%%%%%%%%%%%%%%%%%%%%%%%%%%%%%%%%%%%%%%%%%%%%%%%%%%%%%%%%%%%%%%%%%%%%%%%%%%%%%%%%%%%%%%%%%%%%%%%%%%%%%%%%%%%%%%%%%%%%%%%%%%%%%%%%%%%%%%%%
We begin by examining the form taken by the tight coverage on distributive inverse semigroups.

\begin{lemma}\label{le: essential} Let $S$ be a distributive inverse semigroup.
Then $\{a_{1}, \ldots, a_{m} \} \subseteq a^{\downarrow}$ is a tight cover if and only if
the singleton-set $\{b\}$, where $b = \bigvee_{i=1}^{m} a_{i}$, is a tight cover of $a$.
\end{lemma}
\begin{proof} Suppose first that $\{a_{1}, \ldots, a_{m} \} \subseteq a^{\downarrow}$ is a tight cover.
Let $0 \neq x \leq a$.
Then by assumption there exists $0 \neq z$ such that $z \leq x,a_{i}$ for some $i$.
But clearly $z \leq x, b$.
Conversely, suppose that $b \leq a$ is a tight cover.
Let $0 \neq x \leq a$.
Then by assumption there exists $0 \neq z \leq x, b$.
By distributivity it follows that $z = \bigvee_{i=1}^{m} a_{i} \mathbf{d}(z)$.
It follows that since $z$ is non-zero we have that $a_{i} \mathbf{d}(z)$ is non-zero for some $i$.
Put $z' = a_{i}\mathbf{d}(z)$.
Then $0 \neq z'$ and $z' \leq x,a_{i}$, as required.
\end{proof}

The above lemma leads us to the following definition.
The non-zero element $x$ is said to be {\em essential in $s$} if $x \leq s$ and $s \rightarrow x$.
We shall write $x \preceq s$.
A morphism $\theta \colon S \rightarrow T$ between distributive inverse semigroups is said to be {\em essential}
if $x \preceq s$ implies that $\theta (x) = \theta (s)$.

\begin{lemma} 
A morphism $\theta \colon S \rightarrow T$ between distributive inverse semigroups is essential if and only if it is tight.
\end{lemma}
\begin{proof} Suppose that $\theta$ is essential.
Let $\{a_{1}, \ldots, a_{m} \} \subseteq a^{\downarrow}$ be a cover.
Then $b = \bigvee_{i=1}^{m} a_{i}$ is essential in $a$ by Lemma~\ref{le: essential}.
By assumption $\theta (a) = \theta (b)$.
We now use the fact that $\theta$ is a morphism and so $\theta (b) = \bigvee_{i=1}^{m} \theta (a_{i})$
which gives $\theta (a) = \bigvee_{i=1}^{m} \theta (a_{i})$, as required.
The proof of the converse is immediate. 
\end{proof}

The above lemma tells us that when we are dealing with distributive inverse semigroups from the perespective of tight coverages,
we may restrict our attention to essential elements.
We now extend the definition of $\rightarrow$.
Let $A$ and $B$ be two finite non-empty sets.
We write $A \rightarrow B$ if and only if $a \rightarrow B$ for each $a \in A$.

\begin{lemma} Let $S$ be an inverse semigroup and 
$A = \{a_{1}, \ldots, a_{m} \}^{\downarrow} \subseteq B = \{b_{1}, \ldots, b_{n} \}^{\downarrow}$
both be elements of $\mathsf{D}(S)$.
Then $A \preceq B$ if and only if $\{b_{1}, \ldots, b_{n} \} \rightarrow \{a_{1}, \ldots, a_{m}\}$.
\end{lemma}
\begin{proof} Suppose first that $A \preceq B$.
Let $0 \neq x \leq b_{i}$.
Then $0 \neq x^{\downarrow} \leq B$.
By assumption there exists $0 \neq C \leq x^{\downarrow}, A$.
Let $0 \neq c \in C$.
Then $c \leq x, a_{j}$ for some $j$.
To prove the converse, suppose that
 $\{b_{1}, \ldots, b_{n} \} \rightarrow \{a_{1}, \ldots, a_{m}\}$.
Let $0 \neq C \leq B$ where $C = \{c_{1}, \ldots, c_{p} \}^{\downarrow}$.
By assumption, for each $k$ we may find $x_{k}$ such that $0 \neq x_{k} \leq c_{k}, a_{i_{k}}$.
Put $X = \{x_{1}, \ldots, x_{p} \}^{\downarrow}$.
Then $X \neq 0$, $X \leq C$ and $X \leq A$.
\end{proof}         
 
We now have the following.

\begin{proposition} Let $S$ be an inverse semigroup.
Let $\theta \colon S \rightarrow T$ be a tight homomorphism to a distributive inverse semigroup.
Then the unique morphism $\theta^{\ast} \colon \mathsf{D}(S) \rightarrow T$ 
such that $\theta^{\ast}\iota = \theta$ is a tight morphism.
\end{proposition}
\begin{proof} It is enough to prove that $\theta^{\ast}$ is an essential map.
Let $A = \{a_{1}, \ldots, a_{m} \}^{\downarrow}$ and $B = \{b_{1}, \ldots, b_{n} \}^{\downarrow}$ be two elements of $\mathsf{D}(S)$
such that $B \preceq A$ in $\mathsf{D}(S)$.
We shall prove that $\theta^{\ast}(A) = \theta^{\ast}(B)$.
By definition $\theta^{\ast}(B) = \bigvee_{i=1}^{n} \theta (b_{i})$ and $\theta^{\ast}(A) = \bigvee_{j=1}^{m} \theta (a_{j})$.
Clearly $\theta^{\ast}(B) \leq \theta^{\ast}(A)$.
Thus by definition, we have that
$\theta^{\ast}(B) = \theta^{\ast}(A)(\bigvee_{j=1}^{n} \theta (\mathbf{d}(b_{j})))$.
We shall prove that $\theta (a_{i}) = \theta (a_{i})(\bigvee_{j=1}^{n} \theta (\mathbf{d}(b_{j})))$ from which the result follows
and to do that it is enough to prove that $a_{i} \rightarrow \{a_{i} \mathbf{d}(b_{1}), \ldots, a_{i} \mathbf{d}(b_{n})$.
Let $0 \neq z \leq a_{i}$.
Then $0 \neq z^{\downarrow} \leq A$.
It follows that there is a non-zero $C \in \mathsf{D}(S)$ such that $C \leq z^{\downarrow},B$.
We may therefore find $0 \neq c \in C$ such that $c \leq z, b_{j}$ for some $j$.
It follows that $0 \neq c \leq a_{i} \mathbf{d}(b_{j}),b_{j}$.
\end{proof}

%%%%%%%%%%%%%%%%%%%%%%%%%%%%%%%%%%%%%%%%%%%%%%%%%%%%%%%%%%%%%%%%%%%%%%%%%%%%%%%%%%%%%%%%%%%%%%%%%%%%%%%%%%%%%%%%%%%%%%%%%%%%%%%%%%%%%%%%%%%%%%%%%%
Denote by $\mathbf{S}$ the quotient of $S$ by $\equiv$, defined with respect to the tight coverage,
and the $\equiv$-congruence class containing $s$ by $\mathbf{s}$.
There is a homomorphism $\sigma \colon S \rightarrow \mathbf{S}$ given by $s \mapsto \mathbf{s}$.
The following is immediate from the definition.

\begin{lemma} 
The homomorphism $\sigma \colon S \rightarrow \mathbf{S}$ is $0$-restricted.
\end{lemma}

\begin{lemma}\label{le: tight_coverage_properties} Let $\mathcal{T}$ be the tight coverage on $S$.
\begin{enumerate}

\item Let $\mathbf{X} \in \mathcal{T}(\mathbf{a})$.
Then there exists $A \in \mathcal{T}(a)$ such that $\sigma (A) = \mathbf{X}$.

\item If $\mathcal{T}(\mathbf{a}) \cap \mathcal{T}(\mathbf{b}) \neq \emptyset$ 
then 
$\mathcal{T}(a) \cap \mathcal{T}(b) \neq \emptyset$.

\item Let $X \in \mathcal{T}(a)$. 
Then $\mathbf{X} \in \mathcal{T}(\mathbf{a})$.

\end{enumerate}
\end{lemma} 
\begin{proof} (1). Let $\mathbf{X} = \{\mathbf{x}_{1}, \ldots, \mathbf{x}_{m}\}$.
We have that $\{\mathbf{x}^{-1}\mathbf{x} \colon \mathbf{x} \in \mathbf{X} \} \in \mathcal{T}(\mathbf{a}^{-1}\mathbf{a})$.
For each $\mathbf{x}_{i} \in \mathbf{X}$, choose an idempotent $e_{i}$ such that $\sigma (e_{i}) = \mathbf{x}_{i}^{-1}\mathbf{x}_{i}$.
Put $A = \{ae_{i} \colon 1 \leq i \leq m \} \subseteq a^{\downarrow}$.
Observe that $\sigma (ae_{i}) = \mathbf{a}\mathbf{x}_{i}^{-1}\mathbf{x}_{i} = \mathbf{x}_{i}$.
Thus $\sigma (A) = \mathbf{X}$.
We prove that $a \rightarrow A$.
Let $0 \neq z \leq a$.
Then $z = ak$ for some idempotent $k$.
Thus $0 \neq \sigma (z) \leq \sigma (a)$ since $\sigma$ is $0$-restricted.
Observe that $\sigma (z) = \sigma (a)\sigma (k)$.
Thus there exists $0 \neq \mathbf{u}$ and $\mathbf{x}_{i} \in \mathbf{X}$ such that $\mathbf{u} \leq \mathbf{z}, \mathbf{x}_{i}$.
Choose any idempotent $f$ such that $\sigma (f) = \mathbf{u}^{-1}\mathbf{u}$ and put $u = ae_{i}fk$.
Then $\sigma (u) = \mathbf{u}$, using the fact that $\mathbf{u}\sigma (k) = \mathbf{u}$, and so in particular $u \neq 0$.
By construction $u \leq ae_{i}, z$.
We have therefore proved that $a \rightarrow A$.

(2). By definition, there exists $\mathbf{X} \in \mathcal{T}(\mathbf{a}) \cap \mathcal{T}(\mathbf{b})$.
By (1) above, we may $A \in \mathcal{T}(a)$ and $B \in \mathcal{T}(b)$ such that $\sigma (A) = \mathbf{X} = \sigma (B)$.
Each element of $A$ has the form $ae_{x}$ 
and each element of $B$ has the form $bf_{x}$
where $\sigma (ae_{x}) = \mathbf{x} = \sigma (be_{x})$.
Thus $ae_{x} \equiv be_{x}$.
Choose $C_{x} \in \mathcal{T}(ae_{x}) \cap \mathcal{T}(be_{x})$ and put $C = \bigcup_{x} C_{x}$.
Then by axiom (T), we have that $C \in \mathcal{T}(a) \cap \mathcal{T}(b)$.

(3). Let $0 \neq \sigma (b) \leq \sigma (a)$.
Then $\sigma (b) = \sigma (ab^{-1}b)$.
Thus $b \equiv ab^{-1}b$.
Therefore there exists $Y \in \mathcal{T}(b) \cap \mathcal{T}(ab^{-1}b)$.
But $X \in \mathcal{T}(a)$ implies by (MS) that $Xb^{-1}b \in \mathcal{T}(ab^{-1}b)$.
Thus by Lemma~\ref{le: coverage_properties}, 
we have that $Y \wedge Xb^{-1}b \in \mathcal{T}(ab^{-1}b)$.
Now $0 \neq ab^{-1}b$ and so there exists $0 \neq z$ such that $z \leq ab^{-1}b$ and $z \leq y \wedge xb^{-1}b$ for some
$y \wedge xb^{-1}b \in Y \wedge Xb^{-1}b$.
But then $0 \neq \sigma (z) \leq \sigma (b), \sigma (x)$, as required.
\end{proof}

The following is immediate by the above lemma.

\begin{corollary} Let $S$ be an inverse semigroup.
Then the semigroup $\mathbf{S} = S/\equiv$ is separative with respect to the tight coverage.
\end{corollary}

The following result tells is that we may replace $S$ by $\mathbf{S}$.

\begin{proposition}\label{prop: separated_tight} Let $S$ be an inverse semigroup.
\begin{enumerate}

\item For each tight homomorphism 
$\theta \colon S \rightarrow T$
to a distributive inverse semigroup
there exists a unique tight homomorphism  
$\bar{\theta} \colon \mathbf{S} \rightarrow T$ 
such that $\bar{\theta} \sigma = \theta$.

\item The posets of tight filters on $S$ and those on $\mathbf{S}$ are order-isomorphic and this induces
a homeomorphism between the groupoids $\mathsf{G}_{\mathcal{T}}(S)$ and $\mathsf{G}_{\mathcal{T}}(\mathbf{S})$.

\end{enumerate}
\end{proposition}
\begin{proof} (1). Suppose that $a \equiv b$.
Then there exists $X \in \mathcal{T}(a) \cap \mathcal{T}(b)$.
But $\theta$ is a tight map and so 
$\theta (a) = \bigvee_{x \in X} \theta (x)$ 
and
$\theta (b) = \bigvee_{x \in X} \theta (x)$.
Thus $\theta (a) = \theta (b)$.
We may therefore define $\bar{\theta} (\mathbf{a}) = \theta (a)$.

Let $\mathbf{X} \in \mathcal{T}(\mathbf{a})$ in $\mathbf{S}$.
By Lemma~\ref{le: tight_coverage_properties}, 
there exists $A \in \mathcal{T}(a)$ such that $\sigma (A) = \mathbf{X}$.
By assumption, $\theta (a) = \bigvee_{x \in A} \theta (x)$.
But $\bar{\theta}(\sigma (a)) = \theta (a)$ and
$\bigvee_{x \in X} \bar{\theta} (\sigma(x)) = \bigvee_{x \in A} \theta (x)$ and the result follows.

(2) Let $A$ be a tight filter in $S$.
We shall prove that the map $A \mapsto \sigma (A)$ is a homeomorphism.

Observe first that if $A$ is a tight filter in $S$ then $\sigma (x) \in \sigma (A)$ if and only if $x \in A$.
Suppose that $\sigma (x) \in \sigma (A)$.
Then $\sigma (x) = \sigma (a)$ for some $a \in A$.
Thus $x \equiv a$.
It follows that there exists $X \in \mathcal{T}(x) \cap \mathcal{T}(a)$.
But $a \in A$ and $A$ is a tight filter thus there exists $y \in X \cap A$.
But $y \leq x$ and so $x \in A$, as required.

Let $A$ and $B$ be tight filters.
If $A \subseteq B$ then clearly $\sigma (A) \subseteq \sigma (B)$.
Conversely, suppose that $\sigma (A) \subseteq \sigma (B)$.
Let $a \in A$.
Then $\sigma (a) \in \sigma (A) =\sigma (B)$.
Thus $\sigma (a) \in \sigma (B)$.
It follows by our observation above that $a \in B$ and so $A \subseteq B$.

We prove that if $A$ is a tight filter then $\sigma (A)$ is a tight filter.
By the observation above, it is clear that $\sigma (A)$ is a directed set.
Suppose that $\sigma (a) \in \sigma (A)$ where $a \in A$ and $\sigma (a) \leq \sigma (b)$.
Then  $\sigma (a) = \sigma (ba^{-1}a)$.
Thus again by the observation above, we have that $ba^{-1}a \in A$ and so $b \in A$ giving $\sigma (b) \in \sigma (A)$.
Let $\mathbf{X} \in \mathcal{T}(\mathbf{a})$ where $\mathbf{a} \in \sigma (A)$.
By lemma~\ref{le: tight_coverage_properties}, 
there exists $X \in \mathcal{D}(a)$ such that $\sigma (X) = \mathbf{X}$.
But $a \in A$ and so there exists $b \in A \cap X$.
Thus $\sigma (b) \in \mathbf{X} \cap \sigma (A)$, as required.

The map $A \mapsto \mathbf{A}$ is a bijection.
Suppose that $A$ and $B$ are dense filters such that $\sigma (A) = \sigma (B)$.
Let $a \in A$.
Then $\sigma (a) \in \sigma (A)$ and so there exists $b \in B$ such that $\sigma (a) = \sigma (b)$.
It follows by the observation above that $b \in B$.
By symmetry it follows that $A = B$ and so the map is injective.
We now prove that this map is surjective.
Let $\mathbf{A}$ be a tight filter in $\mathbf{S}$.
Put $A = \sigma^{-1}(\mathbf{A})$.
It is clear that $A$ is closed upwards.
Let $a,b \in A$.
Then $\sigma (a), \sigma (b) \in \mathbf{A}$.
Thus there exists $\sigma (c) \in \mathbf{A}$ such that $\sigma (c) \leq \sigma (a), \sigma (b)$.
It follows that $\sigma (c) = \sigma (ac^{-1}c) = \sigma (bc^{-1}c)$.
Hence $ac^{-1}c \equiv bc^{-1}c$.
Thus there exists $X \in \mathcal{T}(ac^{-1}c) \cap \mathcal{T}(bc^{-1}c)$.
By Lemma~\ref{le: tight_coverage_properties}, $\sigma (X) \in \mathcal{T}(\sigma (c))$.
Thus $\sigma (x) \in \mathbf{A}$, for some $x \in X$, since the filter is dense.
But then $x \in A$ and $x \leq a,b$. as required.
Finally, let $X \in \mathcal{T}(a)$ where $a \in A$.
Then by Lemma~\ref{le: tight_coverage_properties}, we have that $\mathbf{X} \in \mathcal{T}(\mathbf{a})$.
Thus $\sigma (x) \in \mathbf{A}$ for some $x \in X$ and so $x \in A$, as required.
It follows that we have shown that the map is a bijection.

It remains to show that it induces a functor between the groupoids and that it is a homeomorphism.
Let $A$ be a tight filter.
We prove that $\sigma (A^{-1} \cdot A) = \sigma (A)^{-1} \cdot \sigma (A)$.
Let $\sigma (x) \in \sigma (A^{-1} \cdot A)$.
Then $\sigma (a^{-1}b) \leq \sigma (x)$ for some $a,b \in A$.
Thus $\sigma (a^{-1}b) = \sigma (x \mathbf{d}(a^{-1}b))$.
But $A^{-1} \cdot A$ is a tight filter and so $x \mathbf{d}(a^{-1}b)) \in A^{-1} \cdot A$.
It follows that $x \in A^{-1} \cdot A$ and so we may find $c,d \in A$ such that $c^{-1}d \leq x$.
But then $\sigma (c)^{-1} \sigma (d) \leq \sigma (x)$.
We have therefore proved that  $\sigma (A^{-1} \cdot A) \subseteq \sigma (A)^{-1} \cdot \sigma (A)$.
We now prove the reverse inclusion.
Let $\sigma (x) \in \sigma (A)^{-1} \cdot \sigma (A)$.
Then $\sigma (a)^{-1}\sigma (b) \leq \sigma (x)$ for some $a,b \in A$.
Thus $\sigma (a^{-1}b) = \sigma (x \mathbf{d}(a^{-1}b))$.
But $a^{-1}b \in A^{-1} \cdot A$ a tight filter and so $x \mathbf{d}(a^{-1}b)) \in A^{-1} \cdot A$ giving $x \in A^{-1} \cdot A$. 
Thus $\sigma (x) \in \sigma (A^{-1} \cdot A)$.

Suppose that $A$ and $B$ are tight filters such that $A \cdot B$ exists.
Then it follows by the above that $\sigma (A) \cdot \sigma (B)$ exists.
The proof that $\sigma (A \cdot B) = \sigma (A) \cdot \sigma (B)$ is similar to the above proof.

It remains to show that our bijection is a homeomorphism.
Let $s \in S$.
Then $Z_{s}$ consists of all tight filters that contain $s$.
The fact that $\sigma (Z_{s}) = Z_{\sigma (s)}$ is immediate in one direction,
and the converse follows since if $\sigma (A)$ contains $\sigma (s)$ then $s \in A$ by our
observation at the head of the proof.  
It follows that our map is an open map.
Finally, the inverse image of $Z_{\sigma (s)}$ under our map is precisely $Z_{s}$ and so our map is continuous.
\end{proof}

%%%%%%%%%%%%%%%%%%%%%%%%%%%%%%%%%%%%%%%%%%%%%%%%%%%%%%%%%%%%%%%%%%%%%%%%%%%%%%%%%%%%%%%%%%%%%%%%%%%%%%%%%%%%%%%%%%%%%%%%%%%%%%%%%%%%%%%%%%
Let $S$ be a pseudogroup.
A function $\nu \colon S \rightarrow S$ is called a {\em nucleus} if it satisfies the following four conditions:

\begin{description}

\item[{\rm (N1)}] $a \leq \nu (a)$ for all $a \in S$.

\item[{\rm (N2)}] $a \leq b$ implies that $\nu (a) \leq \nu (b)$.

\item[{\rm (N3)}] $\nu^{2}(a) = \nu (a)$ for all $a \in S$.

\item[{\rm (N4)}] $\nu (a) \nu (b) \leq \nu (ab)$ for all $a,b \in S$.

\item[{\rm (N5)}] If $e$ is an idempotent then $\nu (e)$ is an idempotent.

\end{description}
 
\begin{lemma} 
Let $\nu$ be a nucleus on a pseudogroup.
If $e$ and $f$ are idempotents then $\nu (ef) = \nu (e) \nu (f)$.
\end{lemma}
\begin{proof} By (N4), we have that $\nu (e) \nu (f) \leq \nu (ef)$.
But $ef \leq e,f$. Thus by (N2), we have that $\nu (ef) \leq \nu (e), \nu (f)$.
It follows that $\nu (ef)^{2} \leq \nu (e) \nu (f)$.
But $\nu (ef)$ is an idempotent by (N5) and the result follows.
\end{proof}

It follows by the above lemma, that $\nu$ restricted to $E(S)$ is a nucleus in the usual frame-theoretic sense.
The following is a routine derivation from the axioms.

\begin{lemma}\label{le: nuclear_lemma} Let $\nu$ be a nucleus on an inverse semigroup $S$.
Then
$$\nu (ab) = \nu (a \nu (b)) = \nu ( \nu (a) b  ) = \nu ( \nu (a) \nu (b) ).$$
\end{lemma}

Let $S$ be an inverse semigroup equipped with a nucleus $\nu$.
Define
$$S_{\nu} = \{a \in S \colon \nu (a) = a    \},$$ 
the set of {\em $\nu$-closed} elements of $S$.
On the set $S_{\nu}$ define
$$a \cdot b = \nu (ab).$$
A homomorphism $\theta \colon S \rightarrow T$ is said to be {\em idempotent-pure} if $\theta (s)$ an idempotent implies that $s$ is an idempotent.

\begin{lemma} The structure $(S_{\nu},\cdot)$ is a pseudogroup and the map $S \rightarrow S_{\nu}$ given
by $a \mapsto \nu (a)$ is a surjective idempotent-pure semigroup morphism of pseudogroups.
The natural partial order in $S_{\nu}$ coincides with the one in $S$.
\end{lemma} 
\begin{proof} The proof that the operation yields a semigroup follows from Lemma~\ref{le: nuclear_lemma} 
as does the proof that the map is a semigroup map.
The image of an inverse semigroup under a homomorphism is an inverse semigroup.  
Observe that if $\nu (s) = \nu (t)$ then $s$ and $t$ are bounded above and so are compatible.
Thus the kernel of $\nu$ is a subset of the compatibility relation and so idempotent-pure by \cite{Law1}.

We denote the natural partial order in $S_{\nu}$ temporarily by $\preceq$.
Let $a,b \in S_{\nu}$.
Suppose first that $a \preceq b$.
Then $a = b \cdot a^{-1} \cdot a$.
Thus $a = \nu (ba^{-1}a)$.
But $ba^{-1}a \leq b$ and so $\nu (ba^{-1}a) \leq \nu (b) = b$ by (N2).
It follows that $a \leq b$. 
On the other hand, if $a \leq b$ then $a = ba^{-1}a$.
Thus $a = \nu (a) = \nu (ba^{-1}a)$ and so $a \preceq b$, as required. 

From frame theory, we know that $E(S_{\nu})$ forms a frame.

Let $A = \{a_{i} \colon i \in I \}$ be a compatible subset of $S_{\nu}$.
Since the kernel of $\nu$ is contained in the compatibility relation, 
$A$ is obviously a compatible subset of $S$ and so by assumption has a join $a$ in $S$.
Put $a' = \nu (a) \in S_{\nu}$.
We claim that $a'$ is the join of the $a_{i}$ in $S_{\nu}$.
First $a_{i} \leq a \leq \nu (a) = a'$ and so it is an upper bound of the $a_{i}$.
Suppose that $b \in S_{\nu}$ and $a_{i} \leq b$ for all $i$.
Then $a \leq b$ and so $a' = \nu (a) \leq \nu (b) = b$.  
For clarity we shall denote the join operation on $S_{\nu}$ by $\bigsqcup$.

%We shall prove that if $\bigsqcup a_{i}$ exists then $\bigsqcup b \cdot a_{i}$ exists and that
%$b \cdot \bigsqcup a_{i} = \bigsqcup ba_{i}$ where $b$ and the $a_{i}$ are all $\nu$-closed elements.
%In the semigroup $S$, the existence of $\bigvee a_{i}$ implies the existence of $\bigvee ba_{i}$.
%Now $ba_{i} \leq \bigvee ba_{i}$ implies that $\nu (ba_{i}) \leq \nu (\bigvee ba_{i})$.
%Thus the  $\nu (ba_{i})$ are pairwise compatible and so $\bigvee \nu (ba_{i})$ exists.
%It follows that $\bigsqcup b \cdot a_{i}$ exists.
%Now
%$$b \cdot \sqcup a_{i} = \nu (\bigvee ba_{i})$$
%and
%$$\bigsqcup b \cdot a_{i} = \nu (\bigvee \nu (ba_{i})).$$
%Since $ba_{i} \leq \nu (ba_{i})$ it is immediate that
%$\nu (\bigvee ba_{i}) \leq \nu (\bigvee \nu (ba_{i}))$.
%To prove the reverse inequality, we start with
%$ba_{i} \leq \bigvee ba_{i}$ and so
%$\nu (ba_{i}) \leq \nu (\bigvee ba_{i} )$ from which the inequality readily follows.

It remains to show that the map $S \rightarrow S_{\nu}$ given by $a \mapsto \nu (a)$ is a morphism.
Suppose that $a_{i}$ is a compatible set of elements in $S$.
We need to prove that $\nu (\bigvee a_{i} ) = \bigsqcup \nu (a_{i})$.
Since $a_{i} \leq \nu (a_{i})$ we have that $\bigvee a_{i} \leq \bigvee \nu (a_{i})$
and so $\nu (\bigvee a_{i} ) \leq \bigsqcup \nu (a_{i})$.
The proof of the reverse inequality starts with $a_{i} \leq \bigvee_{i} a_{i}$,
and the desired inequality then follows readily.
\end{proof}

%%%%%%%%%%%%%%%%%%%%%%%%%%%%%%%%%%%%%%%%%%%%%%%%%%%%%%%%%%%%%%%%%%%%%%%%%%%%%%%%%%%%%%%%%%%%%%%%%%%%%%%%%%%%%%%%%%%%%%%%%%%
The next result puts our nuclei into context.

\begin{proposition} 
Surjective idempotent-pure pseudogroup morphisms may be described by means of nuclei.
\end{proposition}
\begin{proof}
Let $\theta \colon S \rightarrow T$ be a surjective idempotent-pure pseudogroup morphism.
For each $t \in T$, the inverse image $\theta^{-1} (t)$ is a compatible non-empty subset of $S$ since $\theta$ is idempotent-pure.
Define $\theta_{\ast} \colon T \rightarrow S$ by 
$$\theta_{\ast} (t) = \bigvee \{ s \in S \colon  \theta (s) \leq t \}.$$
Define $\nu \colon S \rightarrow S$ by $\nu (s) = \theta_{\ast} (\theta (s))$.
Observe that $\theta_{\ast}$ is an order-preserving map and that
$s \leq \theta (\theta_{\ast} (s))$ for all $s \in S$ 
and
$\theta (\theta_{\ast}(t)) = t$ for all $t \in T$ since $\theta$ is assumed surjective. 
It therefore follows that $\theta = \theta \theta_{\ast} \theta$ 
and 
$\theta_{\ast} = \theta_{\ast} \theta \theta_{\ast}$.
We claim that $\nu$ is a nucleus on $S$.
The proofs that (N1), (N2), (N3) and (N5) hold are straightforward.
The proof of (N4) follows from the fact that multiplication distributes over compatible joins.

It remains to show that $(S_{\nu},\cdot)$ is isomorphic to $T$.
Let $s,t \in S_{\nu}$.
Then
$$\theta (s \cdot t) = \theta (\nu (st)) = \theta \theta_{\ast} \theta (st) = \theta (st) = \theta (s) \theta (t).$$
If $s,t \in S_{\nu}$ and $\theta (s) = \theta (t)$.
Then $\theta_{\ast}\theta (s) = \theta_{\ast} \theta (t)$ and so $\nu (s) = \nu (t)$ giving $s = t$.
Finally, let $t \in T$.
Then there exists $s \in S$ such that $\theta (s) = t$.
Then $\theta \theta_{\ast} \theta (s) = \theta \theta_{\ast} (t)$ giving $\theta (\nu (s)) = t$.
\end{proof}

%%%%%%%%%%%%%%%%%%%%%%%%%%%%%%%%%%%%%%%%%%%%%%%%%%%%%%%%%%%%%%%%%%%%%%%%%%%%%%%%%%%%%%%%%%%%%%%%%%%%%%%%%%%%%%%%%%%%%%%%%%%%%%%%
Let $S$ be an inverse semigroup.
We recall that $\mathsf{C}(S)$ is the Schein completion of $S$.
We shall define a nucleus on $\mathsf{C}(S)$.

A subset $A$ of $S$ is said to be {\em tightly closed} if $X \in \mathcal{T}(x)$ and $X \subseteq A$ implies that $x \in A$.
Given a subset $A$ of $S$ define $\overline{A}$ to be the set of all elements $x$ such that there exists $X \in \mathcal{T}(x)$
where $X \subseteq A$.

\begin{lemma}\label{le: lemma} \mbox{} Let $S$ be a separative inverse semigroup.
\begin{enumerate}

%1
\item Let $A$ be a compatible order ideal. Then  $\overline{A}$ is a tightly closed compatible order ideal.

%2
\item If $A$ is a compatible order ideal then $\overline{A}$ is equal to the intersection of all tightly closed compatible order ideals that contain $A$.

%3
\item If $E,F \subseteq E(S)$ are tightly closed order ideals of the semilattice of idempotents then so too is $EF$.

%4
\item $\overline{s^{\downarrow}} = s^{\downarrow}$ for all $s \in S$.

\end{enumerate}
\end{lemma}
\begin{proof} (1). We show first that $\overline{A}$ is tightly closed.
Let $X \subseteq \overline{A}$ be such that $X \in \mathcal{T}(a)$.
We need to prove that $a \in \overline{A}$.
Let $x \in X$.
Then either $x \in A$ or $x \in \overline{A} \setminus A$.
If the latter then there exists $A_{x} \subseteq A$ such that $A_{x} \in \mathcal{T}(x)$.
If the former put $A_{x} = \{x \} \in \mathcal{T}(x)$ by (R).
Put $B = \bigcup_{x \in X} A_{x} \subseteq A$.
By (T), we have that $B \in \mathcal{T}(a)$.
Thus $a \in \overline{A}$.

Next we show that $\overline{A}$ is an order ideal.
Let $x \in \overline{A}$ and $y \leq x$ so that $y = xy^{-1}y$.
Suppose that $X \in \mathcal{T}(x)$ where $X \subseteq A$.
By (MS), we have that $Xy^{-1}y \in \mathcal{T}(y)$.
But $A$ is an order ideal and so $Xy^{-1}y \subseteq A$.
Thus $y \in \overline{A}$, as required.

Finally, we show that $\overline{A}$ is a compatible subset.
Let $a,b \in \overline{A}$ where $X,Y \subseteq A$ are such that $X \in \mathcal{T}(a)$ and $Y \in \mathcal{T}(b)$.
Then by (I) and (MS), we have that $X^{-1}Y \in \mathcal{T}(a^{-1}b)$
But $A$ is a compatible set and so $X^{-1}Y$ consists entirely of idempotents.
It follows by separativity that $a^{-1}b$ is an idempotent.
Similarly $ab^{-1}$ is an idempotent.
Thus $a$ and $b$ are compatible, as required.

(2). This is immediate.

(3). Let $g \in \overline{EF}$. By assumption, $g$ is an idempotent.
There exists $X \subseteq EF$ such that $X \in \mathcal{C}(g)$.
But $EF \subseteq E,F$.
Thus $g \in E$ and $g \in F$ and so $g \in EF$, as required.

(4). Let $X \in \mathcal{T}(a)$ where $X \subseteq s^{\downarrow}$.
We prove that $\mathcal{T}(sa^{-1}a) \cap \mathcal{T}(a) \neq \emptyset$ from which we get $a = sa^{-1}a$ and so $a \leq s$.
Now $X \subseteq s^{\downarrow}$ implies that $X = s\mathbf{d}(X)$.
But $\mathbf{d}(X) \in \mathcal{T}(a^{-1}a)$ and so $s\mathbf{d}(X) \in \mathcal{T}(sa^{-1}a)$.
\end{proof}

Let $A$ and $B$ be subsets of $S$.
Define the following sets
$$A^{-1}B = \{s \in S \colon As \subseteq B \}
\text{ and }
BA^{-1} = \{s \in S \colon sA \subseteq B \}.$$

\begin{lemma}\label{le: quotients} Let $B$ be a tightly closed order ideal.
Then $A^{-1}B$ is a tightly closed order ideal for any $A$, and dually.
\end{lemma}
\begin{proof} We show first that $A^{-1}B$ is an order ideal.
Let $s \in A^{-1}B$.
Then by definition $As \subseteq B$.
Let $t \leq s$ and let $a \in A$ be arbitrary.
Then $at \leq as$. 
But $as \in B$ and $B$ is an order ideal and so $at \in B$.
It follows that $t \in A^{-1}B$ and so $A^{-1}B$ is an order ideal.

It remains to prove that $A^{-1}B$ is tightly closed.
Let $X \subseteq A^{-1}B$ where  $X \in \mathcal{T}(x)$.
Then for each $x_{i} \in X$ we have that $Ax_{i} \subseteq B$.
It follows that for each $a \in A$, 
we have that 
$aX \in \mathcal{T}(ax)$ and $aX \subseteq B$.
But $B$ is tightly closed and so $ax \in B$ for every $a \in A$.
It follows by definition that $x \in A^{-1}B$, as required. 
\end{proof}

The following is the crux of our construction.

\begin{lemma}\label{le: it_is_a_nucleus} Let $S$ be a separative inverse semigroup.
Then the function defined on $\mathsf{C}(S)$ by $A \mapsto \overline{A}$
is a nucleus.
\end{lemma}
\begin{proof} It is clear that axioms (N1), (N2) and (N3) hold.
The axiom (N5) holds because a separated coverage is idempotent-pure.

It remains to show that (N4) holds.
Let $A,B \in \mathsf{C}(S)$.
We prove that $\overline{A} \, \overline{B} \subseteq \overline{AB}$.
Let $C$ be any  tightly compatible order ideal containing $AB$.
Thus $AB \subseteq C$.
It follows that $B \subseteq A^{-1}C$.
By Lemma~\ref{le: quotients}, $A^{-1}C$ is a tightly closed order ideal.
Now $B \subseteq \overline{B}$ and so $B = B \cap \overline{B} \subseteq A^{-1}C \cap \overline{B}$.
Observe that the intersection of a tightly compatible order ideal
and a tightly closed order ideal is a tightly closed compatible order ideal.
Thus $A^{-1}C \cap \overline{B}$ is a tightly closed compatible order ideal containing $B$.
It follows that it must contain $\overline{B}$.
Hence $\overline{B} \subseteq A^{-1}C$. 
Thus $A\overline{B} \subseteq C$.
A dual argument shows that $\overline{A} \, \overline{B} \subseteq C$, as required.
\end{proof}

We shall call the nucleus defined above the {\em tight nucleus} defined on $\mathsf{C}(S)$.

\begin{center}
{\bf Let $S$ be a separative inverse semigroup. }
\end{center}

Define $\mathsf{D}_{t}(S)$ to be the subset of $C(S)$ consisting of all tightly closed order ideals of the form $\overline{A}$
where $A$ is an element of $\mathsf{D}(S)$.
We define a multiplication in $\mathsf{D}_{t}(S)$ by $\overline{A} \cdot \overline{B} = \overline{AB}$.
The proof of the following is straightforward.

\begin{lemma} 
$\mathsf{D}_{t}(S)$  is a distributive inverse semigroup and inverse subsemigroup of the pseudogroup arising
from the tight nucleus on $\mathsf{C}(S)$.
\end{lemma} 

Observe that $\mathsf{D}_{t}(S)$ is the image of the map from $\mathsf{D}(S)$ to $\mathsf{C}(S)$
given by $A \mapsto \overline{A}$.
Define $\kappa \colon S \rightarrow \mathsf{D}_{f}(S)$ by $\kappa (s) = s^{\downarrow}$,
which is well-defined by Lemma~\ref{le: lemma}.

\begin{proposition}\label{prop: separated_tight_completions} Let $S$ be a separative inverse semigroup.
Then $\mathsf{D}_{t}(S)$ is a distributive inverse semigroup and the map
$\kappa \colon S \rightarrow \mathsf{D}_{t}(S)$ is a tight map which is universal.
\end{proposition}
\begin{proof} We show first that $\kappa$ is a tight map.
Let $\{x_{1}, \ldots, x_{m} \}$ be a tight cover of $a$ in $S$.
Clearly, $\kappa (x_{i}) \leq \kappa (a)$ for $1 \leq i \leq m$.
Suppose that $\overline{A} \in \mathsf{D}_{t}(S)$ is such that $\kappa (x_{i}) \leq \overline{A}$ for $1 \leq i \leq m$.
Then $x_{1}, \ldots, x_{m} \in \overline{A}$, but the latter is tightly closed.
It follows that $a \in \overline{A}$.
It follows that $\kappa (a) \leq \overline{A}$.
We have therefore proved that $\kappa (a) = \bigvee_{i=1}^{m} \kappa (x_{i})$, as required.

Let $\theta \colon S \rightarrow T$ be any tight map to a distributive inverse semigroup.
Define $\bar{\theta} \colon \mathsf{D}_{t}(S) \rightarrow T$ by
$\bar{\theta}(\overline{\{a_{1}, \ldots, a_{m}\}^{\downarrow}}) = \bigvee_{i=1}^{m} \theta (a_{i})$.
We show first that this map is well-defined.
Suppose that $\overline{\{a_{1}, \ldots, a_{m}\}^{\downarrow}} = \overline{\{b_{1}, \ldots, b_{n}\}^{\downarrow}}$.
Now for each $i$, we have that either $a_{i}$ is an element of $\{b_{1}, \ldots, b_{n} \}^{\downarrow}$
in which case it is less than or equal to some $b_{j}$ or there is a cover $\{x_{1}, \ldots, x_{p} \}$
of $a_{i}$ which is a subset of $\{b_{1}, \ldots, b_{n} \}^{\downarrow}$.
Because $\theta$ is a tight map we have that $\theta (a_{i}) = \bigvee_{k=1}^{p} \theta (x_{k})$.
But in both cases we have that $\theta (a_{i}) \leq \bigvee_{j=1}^{n} \theta (b_{j})$.
It now readly follows that $\bar{\theta}$ is a well-defined map.
It is evident that $\bar{\theta} \kappa = \theta$ and it is routine to check that $\bar{\theta}$ is a morphism.
To prove uniqueness, simply observe that 
$$\overline{\{a_{1}, \ldots, a_{m}  \}^{\downarrow}} = \bigsqcup_{i=1}^{m} a_{i}^{\downarrow}.$$
\end{proof}

\begin{center}
{\bf We may now prove the analogue of the above theorem for arbitrary inverse semigroups.}
\end{center}

Let $S$ be an arbitrary inverse semigroup.
Define $\mathsf{D}_{t}(S) = \mathsf{D}_{t}(\mathbf{S})$ and define $\delta \colon S \rightarrow \mathsf{D}_{t}(S)$
by $\delta (s) = \mathbf{s}^{\downarrow}$.

%%%%%%%%%%%%%%%%%%%%%%%%%%%%%%%%%%%%%%%%%%%%%%%%%%%%%%%%%%%%%%%%%%%%%%%%%%%%%%%%%%%%%%%%%%%%%%%%%%%%%%%%%%%%%%%%%%%%%%%%%%%%%%%%%%%%%%%%%%
\begin{theorem}[Tight completions]\label{the: tight_completions} Let $S$ be an inverse semigroup.
Then $\delta \colon S \rightarrow \mathsf{D}_{t}(S)$ is a tight map which is universal for tight maps from $S$ to distributive inverse semigroups.
\end{theorem}
\begin{proof} We prove first that $\delta$ is a tight map.
Observe that it is the composition of the maps $\sigma \colon S \rightarrow \mathbf{S}$ 
and $\kappa \colon \mathbf{S} \rightarrow \mathsf{D}_{t}(\mathbf{S})$.
By Proposition~\ref{prop: separated_tight_completions},
the latter map is tight.
If $X$ is a tight cover of $s$ in $S$ then $\sigma (X)$ is a tight cover of $\sigma (s)$ in $\mathbf{S}$
by Lemma~\ref{le: tight_coverage_properties}.
It now follows that $\delta$ is a tight map.
The universal property of $\delta$ follows by part (1) of Proposition~\ref{prop: separated_tight}
and Proposition~\ref{prop: separated_tight_completions}.
\end{proof}

%%%%%%%%%%%%%%%%%%%%%%%%%%%%%%%%%%%%%%%%%%%%%%%%%%%%%%%%%%%%%%%%%%%%%%%%%%%%%%%%%%%%%%%%%%%%%%%%%%%%%%%%%%%%%%%%%%%%%%%%%%%%%%%%%%%%%%%%%%%%%%%% 
\begin{theorem}\label{the: order_isomorphism} 
The poset of tight filters on the inverse semigroup $S$ is order isomorphic to the poset of  prime filters on $\mathsf{D}_{t}(S)$
and under this order isomorphism ultrafilters correspond to ultrafilters.
\end{theorem}
\begin{proof} By Proposition~\ref{prop: separated_tight} we may assume that $S$ is separated.
Put $D = \mathsf{D}_{t}(S)$.

%%%%%%%%%%%%%%%%%%%%%%%%%%%%%%%%%%%%%%%%%%%%%%%%%%%%%%%%%%%%%%%%%%%%%%%%%%%%%%%%%%%%%%%%%%%%%%%%%%%%%%%%%%%%%%%%%%%%%%%%
Let $F$ be a tight  filter in $S$.
Define 
$$F^{u} = \{A \in D \colon A \cap F \neq \emptyset \}.$$
We prove that $F^{u}$ is a prime filter in $D$.
Let $A,B \in F^{u}$.
Then $A \cap F \neq \emptyset$ and $B \cap F \neq \emptyset$.
Thus we may find elements $f_{1} \in F \cap A$ and $f_{2} \in F \cap B$.
But $f_{1},f_{2} \in F$ implies that there exists $f \in F$ such that $f \leq f_{1},f_{2}$.
Furthermore, $A$ and $B$ are order ideals.
Thus $f \in A \cap B$.
It follows that $f \in F \cap A \cap B$ and so $f^{\downarrow} \leq A,B$ and $f^{\downarrow} \in F^{u}$.
Now let $A \in F^{u}$ and $A \leq B$.
Then $A \cap F \neq \emptyset$ and $A \subseteq B$ and so $B \cap F \neq \emptyset$ giving $B \in F^{u}$.
Finally, let $\bigsqcup_{i=1}^{m} A_{i} \in F^{u}$.
Then $\overline{\bigcup_{i=1}^{m} A_{i} } \in F^{u}$.
Thus  $\overline{\bigcup_{i=1}^{m} A_{i} } \cap  F \neq \emptyset$.
It follows that there exists $a \in  \overline{\bigcup A_{i}}$ such that $a \in F$.
By definition, there exists $X \in \mathcal{T}(a)$ such that $X \subseteq \bigcup A_{i}$.
But $a \in F$ and  $X \in \mathcal{T}(a)$ implies that there exists $x \in F \cap X$ since $F$ is a tight filter.
But $x \in X \subseteq \bigcup_{i=1}^{m} A_{i}$ and so $x \in A_{i}$.
It follows that $A_{i} \in F^{u}$, as required.
Therefore $F \mapsto F^{u}$ is a well-defined map and it is clearly order-preserving.

%%%%%%%%%%%%%%%%%%%%%%%%%%%%%%%%%%%%%%%%%%%%%%%%%%%%%%%%%%%%%%%%%%%%%%%%%%%%%%%%%%%%%%%%%%%%%%%%%%%%%%%%%%%%%%%%%%%%%%%
Let $P$ be a prime filter in  $D$.
Define
$$P^{d} = \{s \in S \colon \overline{s^{\downarrow}} \in P \}.$$
We prove that $P$ is a tight filter in $S$.
Observe first that $P^{d}$ is non-empty.
Let $\overline{\{a_{1}, \ldots, a_{m}\}^{\downarrow}} \in P$.
Then $\sqcup_{i=1}^{m} a_{i}^{\downarrow} \in P$.
But $P$ is a prime filter and so $a_{i}^{\downarrow} \in P$ for some $i$ and so $a_{i} \in P^{d}$.
We show first that $P^{d}$ is a directed set. 
Let $s,t \in P^{d}$.
Then $s^{\downarrow}, t^{\downarrow} \in P$.
But $P$ is down-directed and so there exists $A \in P$ such that $A \leq s^{\downarrow}, t^{\downarrow}$.
Using the fact that $P$ is a prime filter, we may deduce as above that there is $a \in A$ such that
$a^{\downarrow} \in P$.
But $a \leq s,t$ and $a \in P^{d}$.
It is immediate that $P^{d}$ is closed upwards.
It remains to show that it is a tight filter. 
Let $x \in P^{d}$ where $X \in \mathcal{T}(x)$.
Then $x^{\downarrow} = \bigsqcup_{y \in X} y^{\downarrow} \in P$ and so $y^{\downarrow} \in P$
for some $y \in X$ since $P$ is prime, and so $y \in P^{d}$, as required.
Therefore $P \mapsto P^{d}$ is a well-defined map and it is clearlt order-preserving.

%%%%%%%%%%%%%%%%%%%%%%%%%%%%%%%%%%%%%%%%%%%%%%%%%%%%%%%%%%%%%%%%%%%%%%%%%%%%%%%%%%%%%%%%%%%%%%%%%%%%%%%%%%%%%%%%%%%%%%%%%%%%%%%%%%%%%%%%%%%%%%%
It remains to show that the above two operations are mutually inverse.
We begin by showing that $F = (F^{u})^{d}$.
If $a \in F$ then $a^{\downarrow} \in F^{u}$ and so $a \in (F^{u})^{d}$.
Thus $F \subseteq (F^{u})^{d}$.
Let $s \in (F^{u})^{d}$.
Then $s^{\downarrow} \in F^{u}$.
Thus $s^{\downarrow} \cap F \neq \emptyset$.
It follows that there is $X \in \mathcal{T}(x)$ such that $X \subseteq s^{\downarrow}$.
But $F$ is a tight filter and so there exists $y \in X$ such that $y \in F$.
But $y \leq s$ and so $s \in F$, as required.

Next we show that $P = (P^{d})^{u}$.
Let $A \in (P^{d})^{u}$.
Then $A \cap P^{d} \neq \emptyset$.
Let $a \in A \cap P^{d}$.
Then $a^{\downarrow} \in P$ and $a^{\downarrow} \leq A$.
Thus $A \in P$.
We have shown that $(P^{d})^{u} \subseteq P$.
To prove the reverse inclusion let $A \in P$.
We have that $A = \bigsqcup_{a \in A} a^{\downarrow}$.
But $P$ is prime and so $a^{\downarrow} \in P$ for some $a \in A$.
Thus $a \in P^{d}$ and $A \in (P^{d})^{u}$, as required.

%%%%%%%%%%%%%%%%%%%%%%%%%%%%%%%%%%%%%%%%%%%%%%%%%%%%%%%%%%%%%%%%%%%%%%%%%%%%%%%%%%%%%%%%%%%%%%%%%%%%%%%%%%%%%%%%%%%%%%%%%
It follows that we have defined the desired order-isomorphism.
We proved in Lemma~\ref{le: ultrafilter_tight} that every ultrafilter is a tight filter.
Thus the order isomorphism restricts to a bijection between the corresponding sets of
ultrafilters.
\end{proof}

We now apply the order-isomorphism above.

\begin{theorem} Let $S$ be an inverse semigroup.
The order isomorphism we established in the previous theorem
induces a homeomorphism between the groupoids 
$\mathsf{G}_{\mathcal{T}}(S)$ and $\mathsf{G}_{P}(\mathsf{D}_{t}(S))$.
\end{theorem}
\begin{proof} We use the bijection established in Theorem~\ref{the: order_isomorphism}
and we may continue to assume that $S$ is separated by Proposition~\ref{prop: separated_tight}.
We have therefore set up a bijection between the two groupoids.
We show now that this bijection is a functor.

We show first that  $(F^{-1} \cdot F)^{u} = (F^{u})^{-1} \cdot F^{u}$.
Let $A \in (F^{-1} \cdot F)^{u}$.
Then there is $f \in F$ such that $f^{-1}f \in A$ and $f \in F$.
But $X = f^{\downarrow} \in F^{u}$ and $X^{-1}X \leq A$, as required.
The proof of the reverse inclusion is straightforward.
Let $F$ and $G$ be two tight filters in $S$ such that the product $F \cdot G$ is defined.
Let $A \in (F \cdot G)^{u}$.
Then $fg \in A$ for some $f \in F$ and $g \in G$.
Observe that $f^{\downarrow} \cdot g^{\downarrow} = (fg)^{\downarrow}$.
It follows that $A \in F^{u} \cdot G^{u}$.
To prove the reverse inclusion let $A \in F^{u} \cdot G^{u}$.
Then $XY \leq A$ where $X \in F^{u}$ and $Y \in G^{u}$.
Thus $X \cap F \neq \emptyset$ and $Y \cap G \neq \emptyset$.
Let $f \in X \cap F$ and $g \in Y \cap G$.
But then $fg \in A$ and $fg \in F \cdot G$.
Thus $A \in (F \cdot G)^{u}$.
It follows that the two groupoids are isomorphic.

It remains to show that this isomorphism induces a homeomorphism.
The basic open sets in $\mathsf{G}_{\mathcal{T}}(S)$ have the form $Z_{s}$ where $s \in S$.
We claim that the image of this set under the map $F \mapsto F^{u}$ is the set $X_{s^{\downarrow}}$.
Let $s \in F$ where $F$ is a tight filter.
Then $F^{u}$ is, as we have seen, a prime filter.
But $s^{\downarrow} \cap F \neq \emptyset$ and so $s^{\downarrow} \in F^{u}$.
Conversely, if $A \in X_{s^{\downarrow}}$ then $s \in A^{d}$ and $(A^{d})^{u} = A$.
Thus our isomorphism is an open map.
Finally, consider the open subset $X_{A}$ where $\mathsf{D}_{t}(S)$.
Now $X_{A} = \bigcup_{i=1}^{m} X_{s_{i}^{\downarrow}}$ where $s_{i} \in A$.
Thus we need only determine the inverse images of the sets $X_{s^{\downarrow}}$ for some $s \in S$.
But this is just the set $Z_{s}$.
It follows that our mapping is continuous and open and so it is a homeomorphism.
\end{proof}

%%%%%%%%%%%%%%%%%%%%%%%%%%%%%%%%%%%%%%%%%%%%%%%%%%%%%%%%%%%%%%%%%%%%%%%%%%%%%%%%%%%%%%%%%%%%%%%%%%%%%%%%%%%%%%%%%%%%%%%%%%%%%%%%%%%%%%
\subsection{Exel's tight groupoid and other applications}

An inverse semigroup $S$ is said to be {\em pre-Boolean} if its tight completion is Boolean.

\begin{proposition} 
An inverse semigroup $S$ is pre-Boolean if and only if every (idempotent) tight filter of $S$ is an (idempotent) ultrafilter.
\end{proposition}
\begin{proof}
The distributive inverse semigroup $\mathsf{D}_{t}(S)$ is Boolean
if and only if every prime filter in $\mathsf{D}_{t}(S)$ is an ultrafilter by Proposition~\ref{prop: semigroup_filters}.
But by Theorem~\ref{the: order_isomorphism} there is an order-isomorphism
between the poset of tight filters in $S$ and the poset of prime filters in 
$\mathsf{D}_{t}(S)$ under which ultrafilters correspond.
It follows that every tight filter in $S$ is an ultrafilter.
The restriction to the idempotent case follows by the restriction of Lemma~\ref{le: C-filters} 
to the case of tight coverages.
\end{proof}

%Let $S$ be an inverse semigroup.
%Let $x_{1}, \ldots, x_{m} \leq x$.
%Recall that $U_{x;x_{1},\ldots,x_{n}}$ denotes the set of all tight filters that contain $x$ and omit $x_{1}, \ldots, x_{n}$.
%We denote by $V_{x;x_{1},\ldots,x_{n}}$ denotes the set of all ultrafilters that contain $x$ and omit $x_{1}, \ldots, x_{n}$.

%\begin{proposition} An inverse semigroup $S$ is pre-Boolean if and only if
%whenever $x_{1}, \ldots, x_{m} \leq x$ there exist elements $y_{1}, \ldots, y_{n}$ such that
%$$V_{x;x_{1},\ldots,x_{n}} = \bigcup_{i=1}^{n} V_{y_{i}}.$$
%\end{proposition}
%\begin{proof}
%Suppose that $S$ is pre-Boolean.
%We proved in Proposition~\ref{prop: patch_equal_usual} 
%that a distributive inverse semigroup is Boolean if and only if the patch topology is the same as the usual topology.
%It follows that each set  $U_{x;x_{1},\ldots,x_{n}}$  can be written as a finite union of sets of the form $V_{y_{i}}$.
%We now prove the converse.
%We shall prove that every tight filter is an ultrafilter.

Pre-Boolean semigroups arise naturally: the polycyclic inverse monoids are good examples \cite{Law11, Law3, Law4}.

Finally, Exel's {\em tight groupoid} \cite{Exel1, Exel2} 
may be identified with the groupoid $\mathsf{G}_{P}(\mathsf{D}_{t}(S))^{\dagger}$.

%%%%%%%%%%%%%%%%%%%%%%%%%%%%%%%%%%%%%%%%%%%%%%%%%%%%%%%%%%%%%%%%%%%%%%%%%%%%%%%%%%%%%%%%%%%%%%%%%%%%%%%%%%%%%%%%%%%%%%%%%%%%%%%%%%%%%%%%%%%%%%%%%%%%


\begin{thebibliography}{99}

\bibitem{E} Ch.~Ehresmann, {\em Oeuvres compl\`etes et comment\'ees}, (ed A.~C.~Ehresmann) 
Supplements to {\em Cahiers de topologie et g\'eom\'etrie diff\'erentielle}, Amiens, 1980--83.

\bibitem{Exel1} R.~Exel, Inverse semigroups and combinatorial $C^{\ast}$-algebras, {\em Bull. Braz. Math. Soc. (N.S.)} {\bf 39} (2008), 191--313.

\bibitem{Exel2} R.~Exel, Tight representations of semilattices and inverse semigroups, \textit{Semigroup forum} {\bf 79} (2009), 159--182.

%\bibitem{Exel3} R.~Exel, Reconstructing a totally disconnected groupoid from its ample semigroup, {\em Proc. Amer. Math. Soc.}  
%\textbf{138}  (2010), 2991--3001.

\bibitem{J} P.~T.~Johnstone, {\em Stone spaces}, CUP, 1986.

\bibitem{KM} K.~Kaarli, L.~M\'arki, A characterization of the inverse monoid of bi-congruences of certain algebras, {\em Int. J. Alg. Comput.}
{\bf 19} (2009), 791--808.

%%%%%%%%%%%%%%%%%%%%%%%%%%%%%%%%%%%%%%%%%%%%%%%%%%%%%%%%%%%%%%%%%%%%%%%%%%%%%%%%%%%%%%%%%%%%%%%%%%%%%%%%%%%%%%%%%%%%%%%%%%%%%%%%%%%%%%%%%%%%%%%%%%%%%%%%%%%%%%%
\bibitem{Kel1} J.~Kellendonk, The local structure of tilings and their integer group of coinvariants, {\it Commun. Math. Phys.} {\bf 187} (1997), 115--157.

\bibitem{Kel2} J.~Kellendonk, Topological equivalence of tilings, {\it  J. Math. Phys.} {\bf 38} (1997), 1823--1842.

%\bibitem{KL1} J.~Kellendonk, M. V. Lawson :  Tiling semigroups, {\it  J.  Algebra}, {\bf 224} (2000), 140--150.

%\bibitem{KL2} J.~Kellendonk, M. V. Lawson, Universal groups for point-sets and tilings, {\it J. Algebra}, {\bf  276}  (2004),  462--492

%%%%%%%%%%%%%%%%%%%%%%%%%%%%%%%%%%%%%%%%%%%%%%%%%%%%%%%%%%%%%%%%%%%%%%%%%%%%%%%%%%%%%%%%%%%%%%%%%%%%%%%%%%%%%%%%%%%%%%%%%%%%%%%%%%%%%%%%%%%%%%

%\bibitem{K} G.~Kudryavtseva, A refinement of Stone duality to skew boolean algebras, arXiv:1102.1242v2.

%%%%%%%%%%%%%%%%%%%%%%%%%%%%%%%%%%%%%%%%%%%%%%%%%%%%%%%%%%%%%%%%%%%%%%%%%%%%%%%%%%%%%%%%%%%%%%%%%%%%%%%%%%%%%%%%%%%%%%%%%%%%%%%%%%%%%%%%%%%%%%%%%%%
\bibitem{Law0} M.~V.~Lawson, Coverings and embeddings of inverse semigroups, {\em Proc. Edinb. math. Soc.} {\bf 36} (1993), 399--419.

\bibitem{Law1} M.~V.~Lawson, {\em Inverse semigroups: the theory of partial symmetries}, World Scientific, 1998.

\bibitem{Law11} M.~V.~Lawson, The polycyclic monoids $P_{n}$ and the Thompson groups $V_{n,1}$, {\em Comms. Alg.} {\bf 35} (2007), 4068--4087. 

\bibitem{Law2} M.~V.~Lawson, A non-commutative generalization of Stone duality, {\em J. Aust. Math. Soc.} {\bf 88} (2010), 385--404.

\bibitem{Law3} M.~V.~Lawson, Compactable semilattices, {\em Semigroup Forum} {\bf 81} (2010), 187--199.

\bibitem{Law4} M.~V.~Lawson, Non-commutative Stone duality: inverse semigroups, topological groupoids and $C^{\ast}$-algebras, 
{\em IJAC} {\bf 22}, 1250058 (2012) DOI:10.1142/S0218196712500580. 

%\bibitem{Law5} M.~V.~Lawson, D.~H.~Lenz, Pseudogroups and their \'etale groupoids, arXiv:1107.5511v2. 

\bibitem{LMS} M.~V.~Lawson, S.~Margolis, B.~Steinberg, The \'etale groupoid of an inverse semigroup as a groupoid of filters, 
to appear in {\em J. Austral. Math. Soc.}.

%%%%%%%%%%%%%%%%%%%%%%%%%%%%%%%%%%%%%%%%%%%%%%%%%%%%%%%%%%%%%%%%%%%%%%%%%%%%%%%%%%%%%%%%%%%%%%%%%%%%%%%%%%%%%%%%%%%%%%%%%%%%%%%%%%%%%%%%%%%%%%%%%%%%%%%%
\bibitem{Leech1} J.~Leech, Inverse monoids with a natural semilattice ordering, {\em Proc. Lond. Math.Soc.} (3) {\bf 70} (1995), 146--182.

\bibitem{Leech2} J.~Leech, On the foundations of inverse monoids and inverse algebras, {\em Proc. Edinb. Math. Soc.}
{\bf 41} (1998), 1--21.

%%%%%%%%%%%%%%%%%%%%%%%%%%%%%%%%%%%%%%%%%%%%%%%%%%%%%%%%%%%%%%%%%%%%%%%%%%%%%%%%%%%%%%%%%%%%%%%%%%%%%%%%%%%%%%%%%%%%%%%%%%%%%%%%%%%%%%%%%%%%%%%%%%
\bibitem{L} D.~H.~Lenz, On an order-based construction of a topological groupoid from an inverse semigroup, 
\textit{Proc. Edinb. Math. Soc.} {\bf 51} (2008), 387--406.

\bibitem{MM} S.~Mac~Lane, I.~Moerdijk, {\em Sheaves in geometry and logic}, Springer-Verlag, 1992.

%\bibitem{MR} D.~Matsnev, P.~Resende, Etale groupoids as germ groupoids and their base extensions, {\em Proc. Edinb. Math. Soc.} {\bf 53} (2010), 765--785.

\bibitem{P} A.~L.~T.~Paterson, {\it  Groupoids, Inverse Semigroups, and their Operator Algebras}, Progress in Mathematics,  
{\bf 170}, Birkh\"auser,  Boston, 1998.

\bibitem{Ren} J.~Renault, {\it A groupoid approach to $ C^*$-algebras},  Lecture Notes in Mathematics,  {\bf 793}, Springer, 1980.

%%%%%%%%%%%%%%%%%%%%%%%%%%%%%%%%%%%%%%%%%%%%%%%%%%%%%%%%%%%%%%%%%%%%%%%
\bibitem{R1} P.~Resende, Lectures on \'etale groupoids, inverse semigroups and quantales, Lecture Notes, August, 2006.

\bibitem{R2} P.~Resende, Etale groupoids and their quantales, {\em Adv. Math.} {\bf 208} (2007), 147--209.

\bibitem{R3} P.~Resende, A note on infinitely distributive inverse semigroups, {\em Semigroup Forum} {\bf 73} (2006), 156--158.
%%%%%%%%%%%%%%%%%%%%%%%%%%%%%%%%%%%%%%%%%%%%%%%%%%%%%%%%%%%%%%%%%%%%%%%%%%%%%%%%%%%%

\bibitem{R} W.~Rinow, \"Uber die Vervollst\"andigung induktiver Gruppoide, {\em Math. Nachr.} {\bf 33} (1963), 199--222.

\bibitem{S} B.~Schein, Completions, translational hulls, and ideal extensions of inverse semigroups, 
{\em Czechoslovak Math. J.} {\bf 23} (1973), 575--610.

\bibitem{Stei} B.~Steinberg, A groupoid approach to discrete inverse semigroup algebras, {\em Adv. Math.} \textbf{223}  (2010), 689--727.

\end{thebibliography}
\end{document}